\documentclass{article}
\usepackage{amsmath, amssymb, amsthm, stackrel, graphicx, enumerate,etoolbox}
\usepackage{wrapfig, tikz, spath3,pgf}

\usepackage{tikz-cd,units}
\usetikzlibrary{arrows,knots}
\newtheorem{thma}{Theorem}
\renewcommand*{\thethma}{\Alph{thma}}
\newtheorem{thm}{Theorem}

\newtheorem{lem}[thm]{Lemma}

\theoremstyle{definition}

\newtheorem{df}[thm]{Definition}

\newtheorem{notation}[thm]{Notation}

\makeatletter
\newcommand\xleftrightarrow[2][]{%
	\ext@arrow 9999{\longleftrightarrowfill@}{#1}{#2}}
\newcommand\longleftrightarrowfill@{%
	\arrowfill@\leftarrow\relbar\rightarrow}
\makeatother

\newcommand{\Hom}{\textnormal{Hom}}
\newcommand{\rar}{\rightarrow}
\newcommand{\tr}{\textnormal{tr}}
\newcommand{\C}{\mathbb{C}}
\newcommand{\cat}[1]{\mathcal{#1}}
\newcommand{\Obj}{\textnormal{Obj}}
\newcommand{\id}{\textnormal{id}}
\newcommand{\op}{\textnormal{op}}
\newcommand{\End}{\textnormal{End}}
\newcommand{\Id}{\textnormal{Id}}
\newcommand{\Op}{\textnormal{Op}}
\newcommand{\ev}{\textnormal{ev}}
\newcommand{\coev}{\tn{coev}}
\newcommand{\Vect}{\mathbf{Vect}}
\newcommand{\Tr}{\textnormal{Tr}}
\newcommand{\qdim}{\textnormal{qdim}}
\newcommand{\tn}[1]{\textnormal{#1}}
\renewcommand{\bf}[1]{\textbf{#1}}
\newcommand{\inc}{\hookrightarrow}
\newcommand{\mbb}[1]{\mathbb{#1}}
\newcommand{\TR}{\otimes_\R}
\newcommand{\GT}{\hat{\otimes}}
\newcommand{\svec}{\mathbf{sVect}}
\newcommand{\cattens}[1]{\mathop{\boxtimes}\limits_{#1}}
\newcommand{\slin}{\mathbf{sLinCat}}
\newcommand{\lincat}{\mathbf{LinCat}}
\newcommand{\sbx}{\stackrel{s}{\boxtimes}}
\newcommand{\Rep}{\textnormal{Rep}}
\newcommand{\lcat}[1]{\cat{#1}\lincat}
\newcommand{\bxc}[1]{\cattens{\cat{#1}}}
\newcommand{\Fun}{\textnormal{Fun}}
\newcommand{\Rar}{\Rightarrow}
\newcommand{\cattimes}[1]{\mathop{\times}\limits_{\cat{#1}}}
\newcommand{\ehom}{\underline{\Hom}}
\newcommand{\eend}{\underline{\End}}
\newcommand{\raru}[1]{\mathop{\rar}\limits^{#1}}
\newcommand{\bx}{\boxtimes}
\newcommand{\tens}[1]{\mathop{\otimes}\limits_{#1}}
\newcommand{\rlaru}[1]{\mathop{\longrightarrow}\limits^{#1}}
\newcommand{\dcentcat}[1]{\cat{Z}(\cat{#1})}
\newcommand{\symt}{\mathop{\otimes}\limits_{s}}
\newcommand{\cont}{\mathop{\otimes}\limits_{c}}
\newcommand{\cond}{\mathop{\cdot}\limits_c}
\newcommand{\SA}{\cat{O}(\cat{A})}
\newcommand{\SC}{\cat{O}(\cat{C})}
\newcommand{\rt}{\cattens{\tn{red}}^\cat{A}}

\title{The Drinfeld Centre of a Symmetric Fusion Category is 2-Fold Monoidal}
\author{Thomas A. Wasserman}

\begin{document}
	
\maketitle

\begin{abstract}
	We show that the Drinfeld centre of a symmetric fusion category over an algebraically closed field of characteristic zero is a bilax 2-fold monoidal category. That is, it carries two monoidal structures, the convolution and symmetric tensor products, that are bilax monoidal functors with respect to each other. We additionally show that the braiding and symmetry for the convolution and symmetric tensor products are compatible with this bilax structure.
	
	We establish these properties without referring to Tannaka duality for the symmetric fusion category. This has the advantage that all constructions are done purely in terms of the fusion category structure, making the result easy to use in other contexts. 
\end{abstract}

\tableofcontents

\section{Introduction}
By the Eckmann-Hilton argument, any two monoid structures $\cdot$ and $\star$ on a set $S$ that have the same unit and are compatible in the sense that for all $a,b,c,d\in S$
$$
(a\cdot b)\star (c \cdot d)= (a \star c) \cdot (b \star d),
$$ 
i.e. that are mutually homomorphic, are commutative and equal. When considering monoidal categories a similar argument holds. Categorifying (see \cite{Baez1998} for an introduction to this concept) the above relation, replacing the elements $a,b,c,d$ with objects and the equality with a natural isomorphism, we obtain the notion of strong compatibility of monoidal structures. The Eckmann-Hilton argument in this setting then shows that any two strongly compatible monoidal structures on a category are naturally isomorphic and braided. However, if one relaxes the compatibility to be lax, (so not given by a natural isomorphism, but rather just a natural transformation), the Eckmann-Hilton argument no longer holds. This allows for the existence of lax 2-fold monoidal categories. In the linear case, these are called lax 2-fold tensor categories. 

2-fold (and more generally $n$-fold) monoidal stuctures have been studied in the literature because of their connection to $n$-fold loop spaces \cite{Balteanu1998}. By the loop space recognition theorem, the classifying space of a monoidal category is a loop space, and a braided monoidal category corresponds to a 2-fold loop space. If the braiding is symmetric, the classifying space is in fact an infinite loop space. This jump 
begs the question: what kind of category has an $n$-fold loop space as its classifying space? The notion of $n$-fold monoidal category turns out to be an answer to this question.

In previous work \cite{Wasserman2017}, it was established that the Drinfeld centre \cite{Majid1991,Joyal1991a} of a symmetric fusion category over an algebraically closed field of characteristic zero carries two monoidal structures: its usual tensor product (that we will refer to as the convolution tensor product) as well as a symmetric tensor product. The Drinfeld centre construction produces a (non-degenerate) braided fusion category out of a fusion category. This construction has been well studied, and is of particular interest in the context of fully-extended three-dimensional topological quantum field theories with values in the tricategory of tensor categories. There, the value of the circle turns out to be the Drinfeld centre of the value of the point for these theories. By Tannaka duality\footnote{Tannaka duality for symmetric fusion categories was proved by Deligne \cite{Deligne1990,Deligne2002}. It says that any symmetric fusion category is the representation category of a finite (super) group. The non-super case is referred to as Tannakian.}, if the input of the Drinfeld centre construction is a symmetric fusion category, then the Drinfeld centre will be a category $\Vect_G[G]$ of equivariant (for the conjugation action) vector bundles over a finite group. This category carries two monoidal structures: the fibrewise tensor product $\otimes_f$ of vector bundles, and the convolution tensor product with fibres $(V\otimes_\tn{conv}W)_g=\bigoplus_{g_1g_2=g} V_{g_1}\otimes W_{g_2}$. The fibrewise tensor product and the convolution tensor product satisfy $(V\otimes_f W) \otimes_\tn{conv} (V'\otimes_f W') \hookrightarrow (V \otimes_\tn{conv} V')\otimes_f (W\otimes_\tn{conv} W')$. In the Tannakian case, the symmetric tensor product on the Drinfeld centre corresponds to the fibrewise tensor product, while the usual tensor product corresponds to convolution (hence the name for the usual tensor product on the Drinfeld centre of an abstract symmetric fusion category). The paper \cite{Wasserman2017} can be viewed as defining the fibrewise tensor product without reference to Tannaka duality. The present paper will continue along these lines: we show, without reference to Tannaka duality, that the symmetric tensor product and the convolution tensor product satisfy a similar relation to the one the fibrewise and convolution tensor products satisfy, and that the natural transformation expressing this relation is compatible with the structure natural isomorphisms (associators, unitors and braidings) of these two monoidal structures.\footnote{Alternatively, invoking Tannaka duality, one could establish this result by performing the corresponding computations in $\Vect_G[G]$. The method presented here has the advantage of treating the Tannakian and super-Tannakian cases on equal footing. It additionally allows one to perform the necessary computations using string diagrams, which makes for a visually more appealing presentation of these computations.}

The goal for this paper is then to show that the Drinfeld centre of a symmetric fusion category over an algebraically closed field of characteristic zero is a 2-fold tensor category for its convolution tensor product together with the symmetric tensor product defined in \cite{Wasserman2017}, with compatibility given by the relation from the previous paragraph. We will additionally show that these tensor products are also oplaxly compatible, making it into a bilax 2-fold tensor category. Furthermore, the compatibility morphisms for the oplax structures are one-sided inverses for those of lax structure. To capture this extra property, we define the notion of strongly inclusive bilax 2-fold tensor category. Additionally, the braiding and symmetry for the convolution and symmetric tensor product are compatible with the lax structure, so the Drinfeld centre is a braided 2-fold tensor category, for which one of the braidings is symmetric. We will refer to this last property as being vertically symmetric. All in all, we will show:

\begin{thma}\label{DCtha}
	Let $\cat{A}$ be a symmetric fusion category over an algebraically closed field of characteristic zero, and let $\otimes_c$ and $\otimes_s$ denote the usual (convolution) and symmetric tensor products on its Drinfeld centre $\dcentcat{A}$, respectively. Then $(\dcentcat{A},\otimes_c,\otimes_s)$ is a vertically symmetric braided strongly inclusive bilax 2-fold tensor category.
\end{thma}

Our proofs will make use of string diagram calculus, this gives a visually appealing presentation of the computations. Additionally, it makes the results easier to use in future work, and increases the potential portability of the proofs to contexts outside of symmetric fusion categories over an algebraically closed field of characteristic zero. 

In upcoming work, \cite{Wasserman2017b} we will use this structure to define the notion of a $\dcentcat{A}$-crossed braided category. These are categories $\cat{C}$ enriched and tensored over $(\dcentcat{A},\otimes_s)$, that have a monoidal structure that factors through $\cat{C}\cattens{c}\cat{C}$. Here $\cat{C}\cattens{c}\cat{C}$ denotes the $(\dcentcat{A},\otimes_s)$-enriched and tensored category which has as objects pairs $c\boxtimes c'$ with $c,c'\in\cat{C}$, and as hom-objects between $c_1\boxtimes c_1'$ and $c_2\boxtimes c_2'$ the tensor product $\cat{C}(c_1,c_2)\otimes_c \cat{C}(c_1',c_2')$. This category carries a switch functor that swaps the objects in the pairs and applies the braiding for $\otimes_c$ on the hom-objects. The monoidal structure on such a $\dcentcat{A}$-crossed braided category $\cat{C}$ is required to be braided with respect to this switch functor.

In recent work \cite{Tham2020}, studying four dimensional topological quantum field theories with values in braided tensor categories, Tham independently discovered a corresponding second braided monoidal structure on the Drinfeld centre of a braided fusion category.

The outline of this paper is as follows. We start by recalling some definitions and notation in Section \ref{DCsetupsect}, and the definition and properties of the symmetric tensor product in Section \ref{DCsymtensect}. In Section \ref{DClax2foldsect}, we define the notion of lax 2-fold monoidal category. We spell out what it means for such a category to be braided (or symmetric). We also give definitions that capture the extra properties that the lax structures on the Drinfeld centre of a symmetric fusion category exhibit. The rest of this paper, Section \ref{DCdcaslax2foldsect}, is then devoted to proving the main Theorem \ref{DCtha} of this paper.

\subsubsection*{Acknowledgements}
The author thanks the Engineering and Physical Sciences Research Council for the United Kingdom, the Prins Bernhard Cultuurfonds, the Hendrik Mullerfonds, the Foundation Vrijvrouwe van Renswoude and the Vreedefonds for their financial support. The author is supported by the by the Centre for Symmetry and Deformation at the University of Copenhagen (CPH-SYM-DNRF92) and Nathalie Wahl's European Research Council Consolidator Grant (772960).

Chris Douglas and Andr\'e Henriques have the author's thanks for their support and advise while preparing this document. Thanks go to Mark Penney for pointing the author to the definition of 2-fold monoidal category found in the literature. We would like to thank Michael M\"{u}ger and Ulrike Tillmann for helpful comments and suggestions. The author is grateful to anonymous referees for their comments and advice.

\section{Preliminaries}\label{DCprelims}

In this section we give some notation, recall the definition of the symmetric tensor product $\otimes_s$ on the Drinfeld centre $\dcentcat{A}$ of a symmetric ribbon fusion category $\cat{A}$ over an algebraically closed field of characteristic zero, and set up the theory of lax 2-fold monoidal categories.

\subsection{Notation}\label{DCsetupsect}
We will suppress the associators of $\cat{A}$ (and hence of $\dcentcat{A}$) throughout this paper, and suppress the symbol $\otimes$, when there is no risk of confusion. We will make use of the string diagram calculus for ribbon categories, reading the diagrams from bottom to top. The notation used here is the same as in \cite{Wasserman2017}, we recall it for convenience of the reader.

\subsubsection{String Diagrams}
We will make use of string diagram calculus as introduced in \cite{Joyal1991}.
We will use the following conventions when drawing string diagrams in $\dcentcat{A}$. First of all, recall that objects of the Drinfeld centre are pairs $d=(a,\beta)\in\dcentcat{A}$, with $a\in \cat{A}$ and half-braiding $\beta\colon  -\otimes a \Rightarrow a\otimes-$. For $c \in \dcentcat{A}$ we will denote
$$
\beta_c=
\hbox{
	\begin{tikzpicture}[baseline=(current  bounding  box.center)]
	\node (c) at (0.5,0){\strut$c$};
	\node (a) at (1, 0){\strut$d$};

	\coordinate (co) at (1,1);
	\coordinate (ao) at ( 0.5,1);
	\begin{knot}[clip width=4]
	\strand [thick] (a) to [out=90,in=-90] (ao);
	\strand [thick] (c) to [out=90,in=-90] (co);
	\end{knot}
	\end{tikzpicture}
}
.
$$ 
We will use unresolved crossings to signal the use of the symmetry $s$ in $\cat{A}$. That is, for $(a,\beta),(a',\beta')\in\dcentcat{A}$,
$$
s_{a',a}=:
\hbox{
	\begin{tikzpicture}[baseline=(current  bounding  box.center)]
	\node (c) at (0.5,0){\strut$a'$};
	\node (a) at (1, 0){\strut$a$};

	\coordinate (co) at (1,1);
	\coordinate (ao) at ( 0.5,1);
	\begin{knot}[clip width=4]
	\strand [thick] (a) to [out=90,in=-90] (ao);
	\end{knot}
	\draw [thick] (c) to [out=90,in=-90] (co);
	\end{tikzpicture}
}.
$$
To make manipulations of string diagrams easier to follow, we will sometimes choose to resolve crossings between objects in $\cat{A}\subset \dcentcat{A}$ and objects in $\dcentcat{A}$. So, given $(a,s_{-,a})\in\cat{A}\subset \dcentcat{A}$ and $c\in\dcentcat{A}$, 
$$
s_{c,a}=:
\hbox{
	\begin{tikzpicture}[baseline=(current  bounding  box.center)]
	\node (c) at (0.5,0){\strut$c$};
	\node (a) at (1, 0){\strut$a$};

	\coordinate (co) at (1,1);
	\coordinate (ao) at ( 0.5,1);
	\begin{knot}[clip width=4]
	\strand [thick] (a) to [out=90,in=-90] (ao);
	\end{knot}
	\draw [thick] (c) to [out=90,in=-90] (co);
	\end{tikzpicture}
}
=
\hbox{
	\begin{tikzpicture}[baseline=(current  bounding  box.center)]
	\node (c) at (0.5,0){\strut$c$};
	\node (a) at (1, 0){\strut$a$};

	\coordinate (co) at (1,1);
	\coordinate (ao) at ( 0.5,1);
	\begin{knot}[clip width=4]
	\strand [thick] (a) to [out=90,in=-90] (ao);
	\strand [thick] (c) to [out=90,in=-90] (co);
	\end{knot}
	\end{tikzpicture}
},
$$
When also $c=(a',s_{-,a'})\in\cat{A}\subset\dcentcat{A}$, we have:
\begin{equation}\label{DCtransparancy}
s_{a',a}=:
\hbox{
	\begin{tikzpicture}[baseline=(current  bounding  box.center)]
	\node (c) at (0.5,0){\strut$c$};
	\node (a) at (1, 0){\strut$a$};

	\coordinate (co) at (1,1);
	\coordinate (ao) at ( 0.5,1);
	\begin{knot}[clip width=4]
	\strand [thick] (a) to [out=90,in=-90] (ao);
	\end{knot}
	\draw [thick] (c) to [out=90,in=-90] (co);
	\end{tikzpicture}
}
=
\hbox{
	\begin{tikzpicture}[baseline=(current  bounding  box.center)]
	\node (c) at (0.5,0){\strut$c$};
	\node (a) at (1, 0){\strut$a$};

	\coordinate (co) at (1,1);
	\coordinate (ao) at ( 0.5,1);
	\begin{knot}[clip width=4]
	\strand [thick] (a) to [out=90,in=-90] (ao);
	\strand [thick] (c) to [out=90,in=-90] (co);
	\end{knot}
	\end{tikzpicture}
}
=
\hbox{
	\begin{tikzpicture}[baseline=(current  bounding  box.center)]
	\node (c) at (0.5,0){\strut$c$};
	\node (a) at (1, 0){\strut$a$};

	\coordinate (co) at (1,1);
	\coordinate (ao) at ( 0.5,1);
	\begin{knot}[clip width=4]
	\strand [thick] (a) to [out=90,in=-90] (ao);
	\strand [thick] (c) to [out=90,in=-90] (co);
	\flipcrossings{1}
	\end{knot}
	\end{tikzpicture}
},
\end{equation}
as then both half-braidings correspond to the symmetry in $\cat{A}$. 

By the naturality and monoidality of the symmetry, the resolved and unresolved crossings interact as follows:
\begin{equation}\label{DCcrossinginteraction}
\hbox{
	\begin{tikzpicture}[baseline=(current  bounding  box.center)]
	\coordinate (c) at (0.5,0);
	\coordinate (a) at (1, 0);
	
	\coordinate (b) at (0,0);
	\coordinate (bo) at (0.5,2);
	
	\coordinate (co) at (1,2);
	\coordinate (ao) at ( 0,2);
	\begin{knot}[clip width=4]
	\strand [thick] (a) to [out=90,in=-90] (ao);
	\strand [thick] (c) to [out=90,in=-90] (co);
	\flipcrossings{1}
	\end{knot}
	\draw [thick] (b) to [out=90,in=-90] (bo);
	\end{tikzpicture}
}=
\hbox{
	\begin{tikzpicture}[baseline=(current  bounding  box.center)]
	\coordinate (c) at (0.5,0);
	\coordinate (a) at (1, 0);
	
	\coordinate (b) at (0,0);
	\coordinate (bo) at (0.5,2);
	
	\coordinate (lc) at (0.1,1);
	\coordinate (rc) at (0.5,1);
	
	\coordinate (co) at (1,2);
	\coordinate (ao) at ( 0,2);
	\begin{knot}[clip width=4]
	\strand [thick] (a) to [out=90,in=-90] (ao);
	\strand [thick] (c) to [out=90,in=-90] (lc) to [out=90,in=-90] (co);
	\flipcrossings{1}
	\end{knot}
	\draw [thick] (b) to [out=90,in=-90] (rc) to [out=90,in=-90] (bo);
	\end{tikzpicture}
}.	
\end{equation}

We will often use the following language:
\begin{df}\label{DCtransparentdef}
	Let $c,c'\in\cat{C}$ be objects of a braided monoidal category. If 
	$$
	\hbox{
		\begin{tikzpicture}[baseline=(current  bounding  box.center)]
		\node (c) at (0.5,0){\strut$c$};
		\node (a) at (1, 0){\strut$c'$};

		\coordinate (co) at (1,1);
		\coordinate (ao) at ( 0.5,1);
		\begin{knot}[clip width=4]
		\strand [thick] (a) to [out=90,in=-90] (ao);
		\strand [thick] (c) to [out=90,in=-90] (co);
		\end{knot}
		\end{tikzpicture}
	}
	=
	\hbox{
		\begin{tikzpicture}[baseline=(current  bounding  box.center)]
		\node (c) at (0.5,0){\strut$c$};
		\node (a) at (1, 0){\strut$c'$};

		\coordinate (co) at (1,1);
		\coordinate (ao) at ( 0.5,1);
		\begin{knot}[clip width=4]
		\strand [thick] (a) to [out=90,in=-90] (ao);
		\strand [thick] (c) to [out=90,in=-90] (co);
		\flipcrossings{1}
		\end{knot}
		\end{tikzpicture}
	},
	$$
	then $a$ and $c$ are said to be \emph{transparent} to each other.
\end{df}

In the rest of this work, we will denote by $\cat{O}(\cat{A})$ a set of representatives of the isomorphism classes of simple objects of $\cat{A}$. For such an $i\in\cat{O}(\cat{A})$ we will denote its quantum dimension by 
$$
d_i=
\hbox{
	\begin{tikzpicture}[baseline=(current  bounding  box.center)]
	\coordinate (west) at (-1,0);
	\coordinate (north) at (0,1);
	\coordinate (east) at (1,0);
	\coordinate (south) at (0,-1);
	\node (i) at (-0.8,0) {\strut$i$};
	
	\begin{knot}[clip width=4]
	\strand [thick] (east)
	to [out=-90,in=0] (south)
	to [out=-180,in=-90] (west)
	to [out=90,in=-180] (north)
	to [out=0,in=90] (east);
	\end{knot}
	\end{tikzpicture}
},
$$
here we have suppressed, as we will do throughout, the pivotal structure $i\cong i^{**}$ on the right hand side of the loop. We will also denote:
\begin{equation}\label{DCbulletdef}
\hbox{
	\begin{tikzpicture}[baseline=(current  bounding  box.center)]
	\node (i) at (0,0) {\strut$i$};
	\node (idd) at (0,2) {\strut$i^{**}$};
	
	\coordinate (dot) at (0,1);
	
	\begin{knot}[clip width=4]
	\strand [thick] (i)
	to [out=90,in=-90] (dot)
	to [out=90,in=-90] (idd);
	\end{knot}
	\fill (dot) circle[radius=2pt];
	\end{tikzpicture}
}
:=
\hbox{
	\begin{tikzpicture}[baseline=(current  bounding  box.center)]
	\node (i) at (0,0) {\strut$i$};
	\node (idd) at (0,2) {\strut$i^{**}$};
	
	\coordinate (dot) at (0.25,1.2);
	\coordinate (cc) at (0.25,0.8);
	
	\begin{knot}[clip width=4]
	\strand [thick] (i)
	to [out=90,in=180] (dot)
	to [out=0,in=0] (cc)
	to [out=180,in=-90] (idd);
	\end{knot}
	\end{tikzpicture}
}.
\end{equation}
We define composing this morphism with the pivotal structure to be the twist $\theta_i$ on $i$. We see that, because $\cat{A}$ is symmetric, the twist will be $\pm\id$ on simple objects. The global dimension of $\cat{A}$ will be denoted by
$$
D:=\sum_{i\in\cat{O}(\cat{A})}d_i^2.
$$
As we are working with fusion categories over an algebraically closed field of characteristic zero, this global dimension will always be non-zero \cite[Theorem 2.3]{Etingof2002}.

We will additionally use the notation
\begin{equation}\label{DCringdef}
\hbox{
	\begin{tikzpicture}[baseline=(current  bounding  box.center)]
	\coordinate (west) at (-1,0);
	\coordinate (north) at (0,1);
	\coordinate (east) at (1,0);
	\coordinate (south) at (0,-1);

	\begin{knot}[clip width=4]
	\strand [blue, thick] (west)
	to [out=90,in=-180] (north)
	to [out=0,in=90] (east)
	to [out=-90,in=0] (south)
	to [out=-180,in=-90] (west);
	\end{knot}
	\end{tikzpicture}
}
=\sum_{i\in\cat{O}(\cat{A})}\frac{d_i}{D}
\hbox{
	\begin{tikzpicture}[baseline=(current  bounding  box.center)]
	\coordinate (west) at (-1,0);
	\coordinate (north) at (0,1);
	\coordinate (east) at (1,0);
	\coordinate (south) at (0,-1);
	
	\node (i) at (-0.8,0) {\strut$i$};
	\begin{knot}[clip width=4]
	\strand [thick] (east)
	to [out=-90,in=0] (south)
	to [out=-180,in=-90] (west)
	to [out=90,in=-180] (north)
	to [out=0,in=90] (east);
	\end{knot}
	\end{tikzpicture}
},
\end{equation}
whenever we encounter an unlabelled loop in a string diagram.

\subsubsection{Direct Sum Decompositions}
In our proofs we will make frequent use of the following basic techniques from the theory of fusion categories, that we introduce in the setting of a ribbon fusion category $\cat{C}$. 

\begin{notation}
	Given $i,j,k\in\cat{C}$, we will denote by $B(ij,k)$ a basis for the vector space $\cat{C}(ij,k)$.
\end{notation}

This choice $B(ij,k)$ for each $k\in\SC$, gives for fixed $i,j$ a direct sum decomposition of $ij$. This corresponds to a decomposition of the identity on $ij$ as:
\begin{equation}\label{DCresid}
\hbox{
	\begin{tikzpicture}[baseline=(current  bounding  box.center)]

	\node (a) at (-0.25,-2) {\strut$i$};
	\node (i) at (0.25,-2) {\strut$j$};
	\coordinate (ao) at (-0.25,1);
	\coordinate (io) at (0.25,1);

	\begin{knot}[clip width=4]
	
	\strand [thick] (a) to [out=90,in=-90] (ao);
	\strand [thick] (i) to (io);
	
	\end{knot}
	\end{tikzpicture}
}
= \sum_{k\in\cat{O}(\cat{C})}\sum_{\phi\in B(ij,k)}
\hbox{
	\begin{tikzpicture}[baseline=(current  bounding  box.center)]

	\node (a) at (-0.25,-2) {\strut$i$};
	\node (i) at (0.25,-2) {\strut$j$};
	\coordinate (ao) at (-0.25,1);
	\coordinate (io) at (0.25,1);
	
	\node (phi) at (0,-1) [draw,thick]{\strut$\phi$};
	\node (phid) at (0,0.25) [draw,thick]{\strut$\phi^t$};
	
	\node (k) at (-0.25,-0.5) {\strut$k$};
	
	\begin{knot}[clip width=4]
	\strand [thick] (a) to [out=90,in=-110] (phi);
	\strand [thick] (phi) to [out=90,in=-90] (phid);
	\strand [thick] (phid) to [out =110,in=-90] (ao);
	\strand [thick] (i) to  [out=90, in =-70] (phi);
	\strand [thick] (phid) to [out=70,in=-90] (io);
	\end{knot}
	\end{tikzpicture}
}
.
\end{equation}

The $\phi^t$ are defined below. The pairs $(\phi,\phi^t)$ for a given simple object $k$ are (projection, inclusion) pairs for subobjects of $ij$ isomorphic to $k$. Running through all $\phi$ from the basis $B(ij,k)$ ensures we exhaust the $k$-summands of $ij$ without dependence.
\begin{df}\label{DCtransposedef}
	Let $\phi\in B(ij,k)$, for simple objects $i,j,k$. Then a \emph{transpose} $\phi^t$ of $\phi$ is the dual basis element to $\phi$ in a dual basis for $\cat{C}(k,ij)=\cat{C}(ij,k)$, with respect to the non-degenerate pairing:
	$$
	\circ\colon \cat{C}(ij,k)\otimes\cat{C}(k,ij)\rar \cat{C}(k,k).
	$$ 
	Note that $\phi^t$ is such that $\phi\circ \phi^t=\id_k$ and $\psi\circ \phi^t=0$ for $\psi\in B(ij,k)-\{\phi\}$. 
\end{df}

Choosing direct sum decompositions also induces other direct sum decompositions:

\begin{lem}\label{DCotherdirectsum}
	Pick, for a fixed $i\in\SC$ and all $j\in\SC$, a resolution of the identity on $ij$ as in Equation \eqref{DCresid}. Then, for all $k\in\SC$:
	\begin{equation}\label{DChelpfulid}
	\hbox{
		\begin{tikzpicture}[baseline=(current  bounding  box.center)]
		
		\node (a) at (0.9,-2) {$i$};
		\node (k) at (0.1,-2) {$k^*$};
		
		\node (ao) at (0.9,1.5){$i$};
		\node (ko) at (0.1,1.5){$k^*$};

		\begin{knot}[clip width=4]
		\strand [thick] (k) to [out=90,in=-90] (ko);
		\strand [thick] (a) to [out=90,in=-90] (ao);
		\end{knot}
		\end{tikzpicture}
	}
	=
	\sum_{j\in\cat{O}(\cat{C})}\sum_{\phi\in B(ij,k)}\frac{d_j}{d_k}
	\hbox{
		\begin{tikzpicture}[baseline=(current  bounding  box.center)]
		
		\node (a) at (0.9,-2) {$i$};
		\node (k) at (0.1,-2) {$k^*$};
		
		\node (ao) at (0.9,1.5){$i$};
		\node (ko) at (0.1,1.5){$k^*$};

		\node (phi) at (1,-0.7) [draw,thick]{$\phi$};
		\node (phid) at (1,0.8) [draw,thick]{$\phi^t$};
		
		\coordinate (lc) at (1.5,-0.8);
		\coordinate (uc) at (1.5, 1);
		
		\node (i) at (1.75,0.5) {$j^*$};
		
		\begin{knot}[clip width=4]
		\strand [thick] (ko) to [out=-90,in=-90] (phid);
		\strand [thick] (phid) to [out=70,in=90] (uc) to [out=-90,in=90] (lc) to [out=-90,in=-70] (phi);
		\strand [thick] (phi) to [out=90,in=90] (k);
		\strand [thick] (a) to [out=90,in=-110] (phi);
		\strand [thick] (phid) to [out=110,in=-90] (ao);
		\end{knot}
		\end{tikzpicture}
	}.
	\end{equation}
\end{lem}
\begin{proof}
	A proof of this well-known result can be found in \cite{Wasserman2017}.
\end{proof}

It will be useful to have to following notation:
\begin{equation}\label{DCphistardef}
\hbox{
	\begin{tikzpicture}[baseline=(current  bounding  box.center)]

	\node (a) at (-0.25,-3) {\strut$i^*$};
	\node (i) at (0.25,-3) {\strut$a^*$};
	
	\node (phi) at (0,-1) [draw,thick]{\strut$\phi^*$};
	
	\node (k) at (0,1) {\strut$j^*$};
	
	\begin{knot}[clip width=4]
	\strand [thick] (a) to [out=90,in=-110] (phi);
	\strand [thick] (phi) to [out=90,in=-90] (k);
	\strand [thick] (i) to  [out=90, in =-70] (phi);
	\end{knot}
	\end{tikzpicture}
}
:=
\hbox{
	\begin{tikzpicture}[baseline=(current  bounding  box.center)]

	\node (i) at (-0.25,-3) {\strut$i^*$};
	\node (a) at (0.25,-3) {\strut$a^*$};
	
	\coordinate (ac) at (0.9,-0.6);
	\coordinate (ic) at (0.6,-0.8);
	
	\node (phi) at (0,-1) [draw,thick]{\strut$\phi^t$};
	
	\node (k) at (0,1) {\strut$j^*$};
	\coordinate (kc) at (-0.6,-1.3);
	
	\begin{knot}[clip width=4]
	\strand [thick] (a) to [out=90,in=-90] (ac) to [out=90,in=110] (phi);
	\strand [thick] (phi) to [out=-90,in=-90] (kc) to [out=90,in=-90] (k);
	\strand [thick] (i) to [in=-90,out=90] (ic) to [out=90, in =70] (phi);
	\end{knot}
	\end{tikzpicture}
}.
\end{equation}

The twist from Equation \eqref{DCbulletdef} interacts with the $\phi\in B(ij,k)$ in the following way:
\begin{equation}\label{DCtwisttrick}
\hbox{
	\begin{tikzpicture}[baseline=(current  bounding  box.center)]
	
	\node (phi) at (0,1)[draw]{\strut$\phi$};
	
	\coordinate (wc) at (-0.2,0.5);
	\coordinate (ec) at (0.2,0.5);
	
	\coordinate (nc) at (0,1.5);
	
	\node (k) at (-0.3,0){\strut$i$};
	\node (i) at (0.3,0){\strut$j$};
	
	\node (j) at (0,2){\strut$k$};
	
	\begin{knot}[clip width=4,clip radius=3pt]
	\strand [thick] (k) to [out=90,in=-110] (wc) to [out=70,in=-110] (phi);
	\strand [thick] (i) to [out=90,in=-70] (phi);
	\strand [thick] (phi) to [out=90,in=-90] (nc) to [out=90,in=-90] (j);
	\end{knot}
	\fill (wc) circle[radius=2pt];
	\end{tikzpicture}
}
=
\hbox{
	\begin{tikzpicture}[baseline=(current  bounding  box.center)]
	
	\node (phi) at (0,1)[draw]{\strut$\phi$};
	
	\coordinate (wc) at (-0.2,0.5);
	\coordinate (ec) at (0.2,0.5);
	
	\coordinate (nc) at (0,1.5);
	
	\node (k) at (-0.3,0){\strut$i$};
	\node (i) at (0.3,0){\strut$j$};
	
	\node (j) at (0,2){\strut$k$};
	
	\begin{knot}[clip width=4,clip radius=3pt]
	\strand [thick] (k) to [out=90,in=-110] (phi);
	\strand [thick] (i) to [out=90,in=-70] (ec) to [out=110,in=-70] (phi);
	\strand [thick] (phi) to [out=90,in=-90] (nc) to [out=90,in=-90] (j);
	\end{knot}
	\fill (ec) circle[radius=2pt];
	\fill (nc) circle[radius=2pt];
	\end{tikzpicture}
},
\end{equation}
which follows from the naturality of the twist, together with the fact that in a symmetric fusion category the twist is a monoidal automorphism of the identity functor that squares to $1$.

\subsection{The Symmetric Tensor Product on the Drinfeld Centre}\label{DCsymtensect}
From here onward, we fix a symmetric fusion category $\cat{A}$ over an algebraically closed field of characteristic zero. Its Drinfeld centre will be denoted $\dcentcat{A}$, and the usual braided tensor product (that we we will refer to as the convolution tensor product) on the Drinfeld centre will be denoted by $\otimes_c$. 

\subsubsection{Definition of the Symmetric Tensor Product}
The symmetric tensor product on the Drinfeld centre $\dcentcat{A}$ of the symmetric fusion category $\cat{A}$ is defined in \cite{Wasserman2017} in terms of the following idempotent on $c\otimes_c d$ for $c,d\in\dcentcat{A}$:
\begin{equation}\label{DCidempotent}
\Pi_{c,d}:=
\hbox{
	\begin{tikzpicture}[baseline=(current  bounding  box.center)]
	\coordinate (west) at (-1,0);
	\coordinate (north) at (0,0.5);
	\coordinate (east) at (1,0);
	\coordinate (south) at (0,-0.4);
	\node (a) at (-0.5,-1.4) {\strut$c$};
	\node (b) at (0.5,-1.4) {\strut$d$};
	\coordinate (ao) at (-0.5,1);
	\coordinate (bo) at (0.5,1);
	
	\begin{knot}[clip width=4]
	\strand [blue, thick] (west)
	to [out=90,in=-180] (north)
	to [out=0,in=90] (east)
	to [out=-90,in=0] (south)
	to [out=-180,in=-90] (west);
	\strand [thick] (a) to (ao);
	\strand [thick] (b) to (bo);
	\flipcrossings{1,4}
	\end{knot}
	\end{tikzpicture}
}
=\sum_{i\in\cat{O}(\cat{A})}\frac{d_i}{D}
\hbox{
	\begin{tikzpicture}[baseline=(current  bounding  box.center)]
	\coordinate (west) at (-1,0);
	\coordinate (north) at (0,0.5);
	\coordinate (east) at (1,0);
	\coordinate (south) at (0,-0.4);
	\node (a) at (-0.5,-1.4) {\strut$c$};
	\node (b) at (0.5,-1.4) {\strut$d$};
	\coordinate (ao) at (-0.5,1);
	\coordinate (bo) at (0.5,1);
	
	\node (i) at (-1.25,0){\strut$i$};
	
	\begin{knot}[clip width=4]
	\strand [thick] (west)
	to [out=90,in=-180] (north)
	to [out=0,in=90] (east)
	to [out=-90,in=0] (south)
	to [out=-180,in=-90] (west);
	\strand [thick] (a) to (ao);
	\strand [thick] (b) to (bo);
	\flipcrossings{1,4}
	\end{knot}
	\end{tikzpicture}
}.
\end{equation}

The idempotent $\Pi_{c,d}$ from Equation \eqref{DCidempotent} has an associated subobject denoted $c\otimes_\Pi d\in\dcentcat{A}$. We will use the following notation:
\begin{equation}\label{DCincprojprop}
\hbox{
	\begin{tikzpicture}[baseline=(current  bounding  box.center)]

	\node (truw) at (-0.3,1.2){\strut$c$};
	\node (true) at (0.3,1.2){\strut$d$};

	\node (tr1) at (0,2.2) {\strut$\bigtriangleup$};
	\coordinate (trd1w) at (-0.1,2.12);
	\coordinate (trd1e) at (0.1,2.12);
	\coordinate (tru1) at (0,2.35);
	\node (c1) at (0,3) {\strut$c\otimes_\Pi d$};

	\begin{knot}[clip width=4]
	\strand [thick] (true) to [out=90, in=-70] (trd1e);
	\strand [thick] (tru1) to (c1);
	\strand [thick] (truw) to [out=90,in=-110] (trd1w);
	\end{knot}
	\end{tikzpicture}
}
\tn{ and }
\hbox{
	\begin{tikzpicture}[baseline=(current  bounding  box.center)]
	
	\node (cf) at (0,-0.5){\strut$c\otimes_\Pi d$};

	\node (tr) at (0,0.1) {\strut$\bigtriangledown$};
	\coordinate (trd) at (0,0);
	\coordinate (truw) at (-0.1,00.2);
	\coordinate (true) at (0.1,0.2);

	\node (trd1w) at (-0.3,1.12){\strut$c$};
	\node (trd1e) at (0.3,1.12){\strut$d$};

	\begin{knot}[clip width=4]
	\strand [thick] (cf) to (trd);
	\strand [thick] (true) to [out=70, in=-90] (trd1e);
	\strand [thick] (truw) to [out=110,in=-90] (trd1w);
	\end{knot}
	\end{tikzpicture}
},
\end{equation}
for the projection and inclusion for the subobject associated to $\Pi_{c,d}$. These morphism satisfy
\begin{equation}\label{DCinclprojprops}
\hbox{
	\begin{tikzpicture}[baseline=(current  bounding  box.center)]

	\node (tr1) at (0,0.6) {\strut$\bigtriangleup$};
	\coordinate (trd1w2) at (-0.1,0.52);
	\coordinate (trd1e2) at (0.1,0.52);
	\coordinate (tru1) at (0,0.75);

	\node (tr) at (0,1.2) {\strut$\bigtriangledown$};
	\coordinate (trd) at (0,1.1);
	\coordinate (truw) at (-0.1,1.3);
	\coordinate (true) at (0.1,1.3);

	\node (trd1w) at (-0.3,2.2){\strut$c$};
	\node (trd1e) at (0.3,2.2){\strut$d$};
	
	\node (a) at (-0.3,-0.5){\strut$c$};
	\node (b) at (0.3,-0.5){\strut$d$};

	\begin{knot}[clip width=4]
	\strand [thick] (tru1) to (trd);
	\strand [thick] (true) to [out=70, in=-90] (trd1e);
	\strand [thick] (truw) to [out=110,in=-90] (trd1w);
	\strand [thick] (a) to [out=90,in=-110] (trd1w2);
	\strand [thick] (b) to [out=90,in=-70] (trd1e2);
	\end{knot}
	\end{tikzpicture}
}
=
\hbox{
	\begin{tikzpicture}[baseline=(current  bounding  box.center)]
	\coordinate (west) at (-0.5,-0.3);
	\coordinate (north) at (0,0);
	\coordinate (east) at (0.5,-0.3);
	\coordinate (south) at (0,-0.6);
	\node (a) at (-0.25,-1.5) {\strut$c$};
	\node (b) at (0.25,-1.5) {\strut$d$};
	\coordinate (ao) at (-0.25,1);
	\coordinate (bo) at (0.25,1);

	\begin{knot}[clip width=4]
	\strand [blue, thick] (west)
	to [out=90,in=-180] (north)
	to [out=0,in=90] (east)
	to [out=-90,in=0] (south)
	to [out=-180,in=-90] (west);
	\strand [thick] (a) to [out=90,in=-90] (ao);
	\strand [thick] (b) to (bo);
	\flipcrossings{1,4}
	\end{knot}
	\end{tikzpicture}
}
\tn{ and }
\hbox{
	\begin{tikzpicture}[baseline=(current  bounding  box.center)]
	
	\node (cf) at (0,-0.5){\strut$c\otimes_\Pi d$};

	\node (tr) at (0,0.1) {\strut$\bigtriangledown$};
	\coordinate (trd) at (0,0);
	\coordinate (truw) at (-0.1,00.2);
	\coordinate (true) at (0.1,0.2);

	\node (tr1) at (0,1.9) {\strut$\bigtriangleup$};
	\coordinate (trd1w) at (-0.1,1.82);
	\coordinate (trd1e) at (0.1,1.82);
	\coordinate (tru1) at (0,2.05);
	\node (c1) at (0,2.5) {};

	\begin{knot}[clip width=4]
	\strand [thick] (cf) to (trd);
	\strand [thick] (true) to [out=70, in=-70] (trd1e);
	\strand [thick] (tru1) to (c1);
	\strand [thick] (truw) to [out=110,in=-110] (trd1w);
	\end{knot}
	\end{tikzpicture}
}
=
\hbox{
	\begin{tikzpicture}[baseline=(current  bounding  box.center)]

	\node (inc) at (-0.5,-0.4){\strut$c\otimes_\Pi d$};

	\coordinate (outg) at (-0.5,2.5);
	
	\begin{knot}[clip width=4]
	\strand [thick] (inc)  to [out=90, in =-90] (outg);
	\end{knot}
	\end{tikzpicture}
}.
\end{equation}

The subobject associated to $\Pi_{c,d}$ has the crucial property that the half-braidings associated to both factors agree. For a complete treatment see \cite{Wasserman2017}.

\begin{df}[\cite{Wasserman2017}]\label{DCsymtensobj}
	The \emph{symmetric tensor product} on $\dcentcat{A}$ is the symmetric tensor product
	$$
	\otimes_s\colon  \dcentcat{A}\boxtimes\dcentcat{A}\rar \dcentcat{A}
	$$
	defined as follows. Let $c,d\in\dcentcat{A}$, and write $\Phi\colon \dcentcat{A}\rar\cat{A}$ for the forgetful functor. Then $c\otimes_s d\in \dcentcat{A}$ is the object $(\Phi(c\otimes_{\Pi} d),\beta)$, where $c\otimes_\Pi d$ is the subobject associated to $\Pi_{c,d}$, and $\beta$ is the half-braiding with components, for $a\in\cat{A}$:
	\begin{equation}\label{DChalfbraid}
	\beta_c=
	\hbox{
		\begin{tikzpicture}[baseline=(current  bounding  box.center)]
		
		\node (cf) at (0,-0.5){\strut$c\otimes_s d$};
		\node (inc) at (-0.75,-0.5){\strut$a$};

		\node (c1) at (0,3) {};
		
		\coordinate (outg) at (0.75,2.8);
		
		\begin{knot}[clip width=4]
		\strand [thick] (cf) to (c1);
		\strand [thick] (inc)  to [out=90, in =-90] (outg);
		\end{knot}
		\end{tikzpicture}
	}
	:=
	\hbox{
		\begin{tikzpicture}[baseline=(current  bounding  box.center)]
		
		\node (cf) at (0,-0.5){\strut$c\otimes_s d$};
		\node (inc) at (-0.75,-0.5){\strut$a$};
		
		\node (tr) at (0,0.6) {\strut$\bigtriangledown$};
		\coordinate (trd) at (0,0.5);
		\coordinate (truw) at (-0.1,0.7);
		\coordinate (true) at (0.1,0.7);

		\node (tr1) at (0,2.2) {\strut$\bigtriangleup$};
		\coordinate (trd1w) at (-0.1,2.12);
		\coordinate (trd1e) at (0.1,2.12);
		\coordinate (tru1) at (0,2.35);
		\node (c1) at (0,3) {};
		
		\coordinate (outg) at (0.75,2.8);
		
		\begin{knot}[clip width=4]
		\strand [thick] (cf) to (trd);
		\strand [thick] (true) to [out=70, in=-70] (trd1e);
		\strand [thick] (tru1) to (c1);
		\strand [thick] (inc)  to [out=90, in =-90] (outg);
		\strand [thick] (truw) to [out=110,in=-110] (trd1w);
		\end{knot}
		\end{tikzpicture}
	}
	=
	\hbox{
		\begin{tikzpicture}[baseline=(current  bounding  box.center)]
		
		\node (cf) at (0,-0.5){\strut$c\otimes_s d$};
		\node (inc) at (-0.75,-0.5){\strut$a$};
		
		\node (tr) at (0,0.5) {\strut$\bigtriangledown$};
		\coordinate (trd) at (0,0.4);
		\coordinate (truw) at (-0.1,00.6);
		\coordinate (true) at (0.1,0.6);

		\node (tr1) at (0,2.2) {\strut$\bigtriangleup$};
		\coordinate (trd1w) at (-0.1,2.12);
		\coordinate (trd1e) at (0.1,2.12); 
		\coordinate (tru1) at (0,2.35);
		\node (c1) at (0,3) {};
		
		\coordinate (outg) at (0.5,2.8);
		
		\begin{knot}[clip width=4]
		\strand [thick] (cf) to (trd);
		\strand [thick] (truw) to [out=110,in=-110] (trd1w);
		\strand [thick] (tru1) to (c1);
		\strand [thick] (inc)  to [out=90, in=-90] (outg);
		\strand [thick] (true) to [out=70, in=-70] (trd1e);
		\flipcrossings{}
		\end{knot}
		\end{tikzpicture}
	} ,
	\end{equation}
	where the last equality expresses the fact that the idempotent picks out the subobject on which the half-braidings of $c$ and $d$ agree.
	
	On morphisms $f\colon c\rar c'$, $g\colon d\rar d'$, the symmetric tensor product $\otimes_s$ is given by 
	\begin{equation}\label{DCsymtensmorph}
	\hbox{
		\begin{tikzpicture}[baseline=(current  bounding  box.center)]
		
		\node (cf) at (-1.1,-0.4){\strut$c$};
		\node (inc) at (0.3,-0.4){\strut$d$};
		
		\node (f) at (-1.1,1.5)[draw,minimum height=15pt]{\strut$f$};
		\node (ot) at (-0.375,1.4){\strut$\mathop{\otimes}\limits_\Vect$};
		\node (g) at (0.3,1.5)[draw,,minimum height=15pt]{\strut$g$};
		
		\node (c1) at (-1.1,3) {\strut$c'$};
		
		\node (outg) at (0.3,3){\strut$d'$};
		
		\begin{knot}[clip width=4]
		\strand [thick] (cf) to (f) to (c1);
		\strand [thick] (inc) to (g) to  (outg);
		\end{knot}
		\end{tikzpicture}
	}
	\mapsto
	\hbox{
		\begin{tikzpicture}[baseline=(current  bounding  box.center)]
		
		\node (cf) at (0,-0.5){\strut$c\otimes_s d$};

		\node (f) at (-0.4,1.2)[draw]{\strut$f$};
		\node (g) at (0.4,1.2)[draw]{\strut$g$};

		\node (tr) at (0,0.1) {\strut$\bigtriangledown$};
		\coordinate (trd) at (0,0);
		\coordinate (truw) at (-0.1,00.2);
		\coordinate (true) at (0.1,0.2);

		\node (tr1) at (0,2.2) {\strut$\bigtriangleup$};
		\coordinate (trd1w) at (-0.1,2.12);
		\coordinate (trd1e) at (0.1,2.12);
		\coordinate (tru1) at (0,2.35);
		\node (c1) at (0,3) {\strut$c'\otimes_s d'$};

		\begin{knot}[clip width=4]
		\strand [thick] (cf) to (trd);
		\strand [thick] (true) to [out=70, in=-90] (g);
		\strand [thick] (g) to [out=90,in=-70] (trd1e);
		\strand [thick] (tru1) to (c1);
		\strand [thick] (truw) to [out=110,in=-90] (f);
		\strand [thick] (f) to [out=90,in=-110] (trd1w);
		\end{knot}
		\end{tikzpicture}
	}.
	\end{equation}
	The associators for $\otimes_s$ are induced from the associators for $\cat{A}$, and will be suppressed. 
	The unit for the $\otimes_s$  is the object $\mathbb{I}_s$ which is $\sum_{i\in \cat{O}(\cat{A})} ii^*$, equipped with the half braiding:
	\begin{equation}\label{DCunithalfbraid}
	\hbox{
		\begin{tikzpicture}[baseline=(current  bounding  box.center)]
		\node (i) at (0,0) {\strut$\mathbb{I}_s$};
		\coordinate (ip) at (-0.1,0.27);
		\coordinate (o) at (0,3);
		\coordinate (op) at (-0.1,3);
		
		\node (a) at (-0.5,0){\strut$a$};
		\coordinate (ao) at (0.5,3);

		\begin{knot}[clip width=4]
		\strand [thick,blue] (i) to (o);
		\strand [thick,blue] (ip) to (op);
		\strand [thick] (a) to [out=90,in=-90] (ao);
		\end{knot}
		
		\end{tikzpicture}
	}
	:=
	\sum_{i,j\in\cat{O}(\cat{A})}\sum_{\phi\in B(ai,j)}
	\hbox{
		\begin{tikzpicture}[baseline=(current  bounding  box.center)]
		\node (i) at (0,0) {\strut$i$};
		\node (o) at (0,3){\strut$j$};
		
		\node (id) at (0.5,0){\strut$i^*$};
		\node (do) at (0.5,3){\strut$j^*$};
		
		\node (a) at (-0.5,0){\strut$a$};
		\node (phi) at (0,1) [draw]{\strut$\phi$};
		
		\node (phid) at (0.5,2)[draw]{\strut$\phi^*$};
		\node (ao) at (1,3){\strut$a$};
		
		\begin{knot}[clip width=4]
		\strand [thick] (i) to (phi);
		\strand [thick] (phi) to (o);
		\strand [thick] (a) to [out=90,in=-110] (phi);
		\strand [thick] (id) to (phid);
		\strand [thick] (phid) to (do);
		\strand [thick] (phid) to [out=-60,in=-90] (ao);
		\end{knot}
		
		\end{tikzpicture}
	}.
	\end{equation}
	The double strand will henceforth be used to denote the identity on $\mathbb{I}_s$, and $\phi^*$ was defined in Equation \eqref{DCphistardef}.

	The left unitor built from evaluation morphisms
	\begin{equation}\label{DCleftunitor}
	\hbox{
		\begin{tikzpicture}[baseline=(current  bounding  box.center)]
	\node (inc) at (0,0.5){\strut$\mathbb{I}_s\otimes_s c$};
	\node (bo) at (0,3){\strut$c$};
	\node (tr) at (0,1.1) {\strut$\bigtriangledown$};
	\coordinate (trd) at (0,1);
	\coordinate (truw) at (-0.15,1.2);
	\coordinate (true) at (0.1,1.2);
	\coordinate (trum) at (-0.1,1.2);
	
	\coordinate (ctrl) at (0.175,2.3);
	
	\begin{knot}[clip width=4,clip radius=3pt]
	\strand [thick] (inc) to (trd);
	\strand [thick,blue] (trum) to [out=120, in=0] (ctrl) to [out=180,in=120] (truw);
	\strand [thick] (true) to [out=60,in =-90] (bo);
	\flipcrossings{2}
	\end{knot}
	\end{tikzpicture}
}
:=
\sum\limits_{i\in\cat{O}(\cat{A})}
\hbox{
	\begin{tikzpicture}[baseline=(current  bounding  box.center)]
	\node (inc) at (0,0.5){\strut$\mathbb{I}_s\otimes_s c$};
	\node (bo) at (0,3){\strut$c$};
	\node (il) at (-0.35,1.4){\strut$i$};
	\node (tr) at (0,1.1) {\strut$\bigtriangledown$};
	\coordinate (trd) at (0,1);
	\coordinate (truw) at (-0.15,1.2);
	\coordinate (true) at (0.1,1.2);
	\coordinate (trum) at (-0.1,1.2);
	
	\coordinate (ctrl) at (0.175,2.3);
	
	\begin{knot}[clip width=4,clip radius = 3pt]
	\strand [thick] (inc) to (trd);
	\strand [thick] (trum) to [out=115, in=0] (ctrl) to [out=180,in=120] (truw);
	\strand [thick] (true) to [out=60,in =-90] (bo);
	\flipcrossings{2}
	\end{knot}
	\end{tikzpicture}
},
\end{equation}
where the double strand coming out of the inclusion on the left hand side denotes the object $\mathbb{I}_s$. The right unitor is obtained by reflecting the above diagram in a vertical line.
The inverse for the left unitor is given by:
\begin{equation}\label{DCunitorinverses}
\hbox{
	\begin{tikzpicture}[baseline=(current  bounding  box.center)]
	\node (inc) at (0,2){\strut$\mathbb{I}_s\otimes_s c$};
	\node (bo) at (0,-0.7){\strut$c$};

	\node (tr1n) at (0,1.2) {\strut$\bigtriangleup$};
	\coordinate (trd) at (0,1.35);
	\coordinate (truw) at (-0.15,1.12);
	\coordinate (true) at (0.1,1.12);
	\coordinate (trum) at (-0.1,1.12);
	
	\coordinate (ctrl) at (0.4,0.0);
	
	\begin{knot}[clip width=4, clip radius = 3pt]
	\strand [thick] (inc) to (trd);
	\strand [thick,blue] (trum) to [out=-120, in=0] (ctrl) to [out=180,in=-120] (truw);
	\strand [thick] (true) to [out=-55,in =90] (bo);
	\flipcrossings{2}
	\end{knot}
	\end{tikzpicture}
}
:=
\sum\limits_{i\in\cat{O}(\cat{A})}\frac{d_i}{D}
\hbox{
	\begin{tikzpicture}[baseline=(current  bounding  box.center)]
	\node (inc) at (0,2){\strut$\mathbb{I}_s\otimes_s c$};
	\node (bo) at (0,-0.7){\strut$c$};
	
	\node (il) at (-0.35,0.9){\strut$i$};
	
	\node (tr1n) at (0,1.2) {\strut$\bigtriangleup$};
	\coordinate (trd) at (0,1.35);
	\coordinate (truw) at (-0.15,1.12);
	\coordinate (true) at (0.1,1.12);
	\coordinate (trum) at (-0.1,1.12);
	
	\coordinate (ctrl) at (0.4,0.0);
	
	\begin{knot}[clip width=4,clip radius = 3pt]
	\strand [thick] (inc) to (trd);
	\strand [thick] (trum) to [out=-115, in=0] (ctrl) to [out=180,in=-120] (truw);
	\strand [thick] (true) to [out=-55,in =90] (bo);
	\flipcrossings{2}
	\end{knot}
	\end{tikzpicture}
},
\end{equation}
and the inverse for the right unitor is again given by reflecting in a vertical line. The symmetry for $\otimes_s$ is induced from the symmetry in $\cat{A}$ as follows:
\begin{equation}\label{DCsymmorph}
\hbox{
	\begin{tikzpicture}[baseline=(current  bounding  box.center)]
	
	\node (cf) at (0,-0.5){\strut$c\otimes_s c'$};

	\node (tr) at (0,0.1) {\strut$\bigtriangledown$};
	\coordinate (trd) at (0,0);
	\coordinate (truw) at (-0.1,00.2);
	\coordinate (true) at (0.1,0.2);

	\node (tr1) at (0,1.2) {\strut$\bigtriangleup$};
	\coordinate (trd1w) at (-0.1,1.12);
	\coordinate (trd1e) at (0.1,1.12);
	\coordinate (tru1) at (0,1.35);
	\node (c1) at (0,2) {\strut$c'\otimes_s c$};

	\begin{knot}[clip width=4]
	\strand [thick] (cf) to (trd);
	\strand [thick] (true) to [out=70, in=-110] (trd1w);
	\strand [thick] (tru1) to (c1);
	\end{knot}
	\draw [thick] (truw) to [out=110,in=-70] (trd1e);
	\end{tikzpicture}
}.
\end{equation}
\end{df}

\subsubsection{Properties of the Symmetric Tensor Product}
We will need the following properties of the idempotent from Equation \eqref{DCidempotent} and the symmetric tensor product. We will refer to these properties as cloaking, slicing and snapping.

\begin{lem}[Cloaking]\label{DCcloaking}
	Let $c,c'\in\dcentcat{A}$ and $a\in \cat{A}$. Then the following identity holds:
	$$
	\hbox{
		\begin{tikzpicture}[baseline=(current  bounding  box.center)]
		\coordinate (west) at (-0.5,-0.3);
		\coordinate (north) at (0,0);
		\coordinate (east) at (0.5,-0.3);
		\coordinate (south) at (0,-0.6);
		\node (a) at (-0.75,-1.5) {\strut$a$};
		\node (b) at (-0.25,-1.5) {\strut$c$};
		\node (c) at (0.25,-1.5) {\strut$c'$};
		\coordinate (ao) at (0.75,1);
		\coordinate (bo) at (-0.25,1);
		\coordinate (co) at (0.25,1);
		
		\coordinate (lc) at (-0.75,-0.2);

		\begin{knot}[clip width=4]
		\strand [blue, thick] (west)
		to [out=90,in=-180] (north)
		to [out=0,in=90] (east)
		to [out=-90,in=0] (south)
		to [out=-180,in=-90] (west);
		\strand [thick] (a) to [out=90,in=-90] (lc) to [out=90,in=-90] (ao);
		\strand [thick] (b) to (bo);
		\strand [thick] (c) to (co);
		\flipcrossings{5,1,4}
		\end{knot}
		\end{tikzpicture}
	}=
	\hbox{
		\begin{tikzpicture}[baseline=(current  bounding  box.center)]
		\coordinate (west) at (-0.5,-0.3);
		\coordinate (north) at (0,0);
		\coordinate (east) at (0.5,-0.3);
		\coordinate (south) at (0,-0.6);
		\node (a) at (-0.75,-2) {\strut$a$};
		\node (b) at (-0.25,-2) {\strut$c$};
		\node (c) at (0.25,-2) {\strut$c'$};
		\coordinate (ao) at (0.75,0.5);
		\coordinate (bo) at (-0.25,0.5);
		\coordinate (co) at (0.25,0.5);
		
		\coordinate (rc) at (0.75,-0.4);

		\begin{knot}[clip width=4]
		\strand [blue, thick] (west)
		to [out=90,in=-180] (north)
		to [out=0,in=90] (east)
		to [out=-90,in=0] (south)
		to [out=-180,in=-90] (west);
		\strand [thick] (a) to [out=90,in=-90] (rc) to [out=90,in=-90] (ao);
		\strand [thick] (b) to (bo);
		\strand [thick] (c) to (co);
		\flipcrossings{6,1,4}
		\end{knot}
		\end{tikzpicture}
	}.
	$$
\end{lem}

\begin{proof}
	See \cite{Wasserman2017}.
\end{proof}

\begin{lem}[Slicing]\label{DCslicing}
	The half-braiding on $c\otimes_s c'$ and the inclusion and projection maps for $\Pi_{c,c'}$ interact as follows:
	$$
	
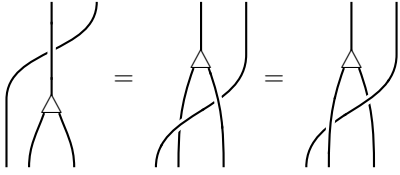
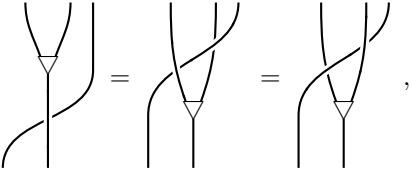
\hbox{
		\begin{tikzpicture}[baseline=(current  bounding  box.center)]
		
		\coordinate (c) at (-0.6,1);
		\coordinate (a) at (-0.3,1);
		\coordinate (b) at (0.3,1);
		
		\coordinate (ctrl) at (-0.6,1.9);
		
		\node (tr1) at (0,1.8) {\strut$\bigtriangleup$};
		\coordinate (trd1w) at (-0.1,1.72);
		\coordinate (trd1e) at (0.1,1.72);
		\coordinate (tru1) at (0,1.95);
		\coordinate (c1) at (0,3.2);
		
		\coordinate (co) at (0.6,3.2);
		
		\begin{knot}[clip width=4]
		\strand [thick] (c) to [out=90,in=-90] (ctrl) to [out=90,in=-90] (co);
		\strand [thick] (b) to [out=90, in=-70] (trd1e);
		\strand [thick] (tru1) to (c1);
		\strand [thick] (a) to [out=90,in=-110] (trd1w);
		\flipcrossings{1}
		\end{knot}
		\end{tikzpicture}
	}
	=
	\hbox{
		\begin{tikzpicture}[baseline=(current  bounding  box.center)]
		
		\coordinate (c) at (-0.6,0.8);
		\coordinate (a) at (-0.3,0.8);
		\coordinate (b) at (0.3,0.8);
		
		\coordinate (ctrl) at (0.6,2.3);
		
		\node (tr1) at (0,2.2) {\strut$\bigtriangleup$};
		\coordinate (trd1w) at (-0.1,2.12);
		\coordinate (trd1e) at (0.1,2.12);
		\coordinate (tru1) at (0,2.35);
		\coordinate (c1) at (0,3);
		
		\coordinate (co) at (0.6,3);
		
		\begin{knot}[clip width=4]
		\strand [thick] (c) to [out=90,in=-90] (ctrl) to [out=90,in=-90] (co);
		\strand [thick] (b) to [out=90, in=-70] (trd1e);
		\strand [thick] (tru1) to (c1);
		\strand [thick] (a) to [out=90,in=-110] (trd1w);
		\flipcrossings{1}
		\end{knot}
		\end{tikzpicture}
	}
	=
	\hbox{
		\begin{tikzpicture}[baseline=(current  bounding  box.center)]
		
		\coordinate (c) at (-0.6,0.8);
		\coordinate (a) at (-0.3,0.8);
		\coordinate (b) at (0.3,0.8);
		
		\coordinate (ctrl) at (0.6,2.3);
		
		\node (tr1) at (0,2.2) {\strut$\bigtriangleup$};
		\coordinate (trd1w) at (-0.1,2.12);
		\coordinate (trd1e) at (0.1,2.12);
		\coordinate (tru1) at (0,2.35);
		\coordinate (c1) at (0,3);
		
		\coordinate (co) at (0.6,3);
		
		\begin{knot}[clip width=4]
		\strand [thick] (c) to [out=90,in=-90] (ctrl) to [out=90,in=-90] (co);
		\strand [thick] (b) to [out=90, in=-70] (trd1e);
		\strand [thick] (tru1) to (c1);
		\strand [thick] (a) to [out=90,in=-110] (trd1w);
		\flipcrossings{2}
		\end{knot}
		\end{tikzpicture}
	}
	$$
	and
	$$
	\rotatebox{180}{
		\begin{tikzpicture}[baseline=(current  bounding  box.center)]
		
		\coordinate (c) at (-0.6,1);
		\coordinate (a) at (-0.3,1);
		\coordinate (b) at (0.3,1);
		
		\coordinate (ctrl) at (-0.6,1.9);
		
		\node (tr1) at (0,1.8) {\strut$\bigtriangleup$};
		\coordinate (trd1w) at (-0.1,1.72);
		\coordinate (trd1e) at (0.1,1.72);
		\coordinate (tru1) at (0,1.95);
		\coordinate (c1) at (0,3.2);
		
		\coordinate (co) at (0.6,3.2);
		
		\begin{knot}[clip width=4]
		\strand [thick] (c) to [out=90,in=-90] (ctrl) to [out=90,in=-90] (co);
		\strand [thick] (b) to [out=90, in=-70] (trd1e);
		\strand [thick] (tru1) to (c1);
		\strand [thick] (a) to [out=90,in=-110] (trd1w);
		\flipcrossings{1}
		\end{knot}
		\end{tikzpicture}
	}
	=
	\rotatebox{180}{
		\begin{tikzpicture}[baseline=(current  bounding  box.center)]
		
		\coordinate (c) at (-0.6,0.8);
		\coordinate (a) at (-0.3,0.8);
		\coordinate (b) at (0.3,0.8);
		
		\coordinate (ctrl) at (0.6,2.3);
		
		\node (tr1) at (0,2.2) {\strut$\bigtriangleup$};
		\coordinate (trd1w) at (-0.1,2.12);
		\coordinate (trd1e) at (0.1,2.12);
		\coordinate (tru1) at (0,2.35);
		\coordinate (c1) at (0,3);
		
		\coordinate (co) at (0.6,3);
		
		\begin{knot}[clip width=4]
		\strand [thick] (c) to [out=90,in=-90] (ctrl) to [out=90,in=-90] (co);
		\strand [thick] (b) to [out=90, in=-70] (trd1e);
		\strand [thick] (tru1) to (c1);
		\strand [thick] (a) to [out=90,in=-110] (trd1w);
		\flipcrossings{1}
		\end{knot}
		\end{tikzpicture}
	}
	=
	\rotatebox{180}{
		\begin{tikzpicture}[baseline=(current  bounding  box.center)]
		
		\coordinate (c) at (-0.6,0.8);
		\coordinate (a) at (-0.3,0.8);
		\coordinate (b) at (0.3,0.8);
		
		\coordinate (ctrl) at (0.6,2.3);
		
		\node (tr1) at (0,2.2) {\strut$\bigtriangleup$};
		\coordinate (trd1w) at (-0.1,2.12);
		\coordinate (trd1e) at (0.1,2.12);
		\coordinate (tru1) at (0,2.35);
		\coordinate (c1) at (0,3);
		
		\coordinate (co) at (0.6,3);
		
		\begin{knot}[clip width=4]
		\strand [thick] (c) to [out=90,in=-90] (ctrl) to [out=90,in=-90] (co);
		\strand [thick] (b) to [out=90, in=-70] (trd1e);
		\strand [thick] (tru1) to (c1);
		\strand [thick] (a) to [out=90,in=-110] (trd1w);
		\flipcrossings{2}
		\end{knot}
		\end{tikzpicture}
	},
	$$
	where the diagonal strand is labelled by an object of $\cat{A}$.
\end{lem}

\begin{proof}
	See \cite{Wasserman2017}.
\end{proof}

\begin{lem}[Snapping]\label{DCsnapping}
	For any $c\in\dcentcat{A}$ we have:
	$$
	
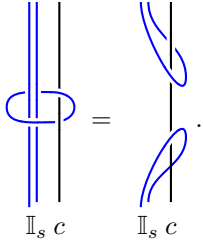
\hbox{
		\begin{tikzpicture}[baseline=(current  bounding  box.center)]
		\node (i) at (0,0) {\strut$\mathbb{I}_s$};
		\coordinate (ip) at (-0.1,0.27);
		\coordinate (o) at (0,3);
		\coordinate (op) at (-0.1,3);
		
		\node (a) at (0.3,0){\strut$c$};
		\coordinate (ao) at (0.3,3);
		
		\coordinate (west) at (-0.4,1.6);
		\coordinate (north) at (0.2,1.8);
		\coordinate (east) at (0.5,1.6);
		\coordinate (south) at (0.2,1.4);

		\begin{knot}[clip width=4,clip radius =3pt]
		\strand [thick,blue] (i) to (o);
		\strand [thick,blue] (ip) to (op);
		\strand [thick] (a) to [out=90,in=-90] (ao);
		\strand [blue, thick] (west) to [out=90,in=-180] (north) to [out=0,in=90] (east) to [out=-90,in=0] (south) to [out=-180,in=-90] (west);
		\flipcrossings{2,4,5}
		\end{knot}
		\end{tikzpicture}
	}
	=
	\hbox{
		\begin{tikzpicture}[baseline=(current  bounding  box.center)]
		\node (i) at (0,0) {\strut$\mathbb{I}_s$};
		\coordinate (ip) at (-0.1,0.27);
		\coordinate (o) at (0,3);
		\coordinate (op) at (-0.1,3);
		
		\node (a) at (0.3,0){\strut$c$};
		\coordinate (ao) at (0.3,3);
		
		\coordinate (uc) at (0.5,2);
		\coordinate (lc) at (0.5,1.2);
		
		\begin{knot}[clip width=4,clip radius=4pt]
		\strand [thick,blue] (i) to [out=90,in=-90] (lc) to [out=90,in=90] (ip);
		\strand [thick,blue] (o) to [out=-90,in=90] (uc) to [out=-90,in=-90] (op);
		\strand [thick] (a) to [out=90,in=-90] (ao);
		\flipcrossings{2,4,5}
		\end{knot}
		\end{tikzpicture}
	}.
	$$
\end{lem}

\begin{proof}
	See \cite{Wasserman2017}.
\end{proof}

\subsection{2-fold Monoidal Categories}\label{DClax2foldsect}

The following definition is inspired by \cite{Balteanu1998}, but allows for the units of the two monoidal structures to be different.
\begin{df}\label{DClax2fold}
	Let $\cat{C}$ be a category equipped with two monoidal structures $\otimes_1$ and $\otimes_2$, with units $\mathbb{I}_1$ and $\mathbb{I}_2$, respectively. The associator and right and left unitor isomorphisms for monoidal structures will be denoted $\alpha_1,\rho_1,\lambda_1$ and $\alpha_2,\rho_2,\lambda_2$, respectively. Then $(\cat{C},\otimes_1,\otimes_2)$ is called \emph{lax 2-fold monoidal} if it comes equipped with a natural transformation $\eta$ with components
	$$
	\eta_{c,c',d,d'}\colon  (c\otimes_1 c')\otimes_2 (d\otimes_1 d')\rar (c\otimes_2 d)\otimes_1(c'\otimes_2 d'),
	$$
	and morphisms 
	\begin{align*}
	u_0\colon &\mathbb{I}_2\rar \mathbb{I}_1\\
	u_1\colon &\mathbb{I}_1\otimes_2\mathbb{I}_1\rar \mathbb{I}_1\\
	u_2\colon &\mathbb{I}_2\rar\mathbb{I}_2\otimes_1\mathbb{I}_2.
	\end{align*}
	We will refer to these morphisms as \emph{compatibility morphisms}.
	These morphisms are	such that the following diagrams commute for all $c,c',d,d'\in\cat{C}$.
	
	\begin{enumerate}[(a)]
		\item\label{DCunitor1}
		\begin{tikzcd}
			\mathbb{I}_2 \otimes_1 \mathbb{I}_1\arrow[r,"u_0\otimes \tn{id}"]\arrow[d,"\rho_1"] & \mathbb{I}_1\otimes_1 \mathbb{I}_1\arrow[d,"\rho_1"]\\
			\mathbb{I}_2 \arrow[r,"u_0"] & \mathbb{I}_1
		\end{tikzcd}, 
		
		together with the corresponding diagram for $\lambda_1$, and the corresponding diagrams for $\lambda_2$ and $\rho_2$,

		\item \label{DCunitor4}
		\begin{tikzcd}
			\mathbb{I}_2\otimes_2(d\otimes_1d')	\arrow[d,"u_2\otimes_2\tn{id}"]	\arrow[r,"\lambda_2"]	&	d\otimes_1d'\\
			(\mathbb{I}_2\otimes_1\mathbb{I}_2)\otimes_2(d\otimes_1d')				\arrow[r,"\eta"]		&	(\mathbb{I}_2\otimes_2d)\otimes_1(\mathbb{I}_2\otimes_2d')	\arrow[u,"\lambda_2\otimes_2\lambda_2"],
		\end{tikzcd}
		
		where $\lambda_2$ is the left-unitor for $\otimes_2$. We similarly require the corresponding diagrams for the right-unitor to commute.
		\item \label{DCunitor3}
		\begin{tikzcd}
			(\mathbb{I}_1\otimes_1c')\otimes_2(\mathbb{I}_1\otimes_1 d')	\arrow[r,"\eta"]\arrow[d,"\lambda_1\otimes_2\lambda_1"]		&	(\mathbb{I}_1\otimes_2\mathbb{I}_1)\otimes_1(c'\otimes_2 d')\arrow[d,"u_1\otimes_2 \tn{id}"]\\
			c'\otimes_2 d' 	&	\mathbb{I}_1\otimes_1(c'\otimes_2d'), \arrow[l,"\lambda_1"]
		\end{tikzcd}
		
		where $\lambda_1$ denotes the left-unitor for $\otimes_1$, and the corresponding diagram for the right-unitor is also required to commute.
		
		\item \label{DCunitor2}
		\begin{tikzcd}
			(\mathbb{I}_1\otimes_2 \mathbb{I}_1)\otimes_2 \mathbb{I}_1 \arrow[rr,"\alpha_2"]	\arrow[d,"u_1\otimes_2\tn{id}"]& &\mathbb{I}_1\otimes_2(\mathbb{I}_1\otimes_2 \mathbb{I}_1)\arrow[d,"\tn{id}\otimes_2 u_1"]\\
			\mathbb{I}_1 \otimes_2 \mathbb{I}_1 \arrow[r, "u_1"] & \mathbb{I}_1& \mathbb{I}_1 \otimes_2 \mathbb{I}_1\arrow[l,"u_1"'],
		\end{tikzcd}
		
		and the corresponding diagram for $u_2$ and $\alpha_1$,

		\item \label{DCtensor1}
		\begin{tikzcd}
			((c\otimes_1d)\otimes_2(c'\otimes_1d'))\otimes_2(c''\otimes_1d'')	\arrow[r,"\alpha_2"]\arrow[d,"\eta\otimes_2 \tn{id}"] 	&	(c\otimes_1d)\otimes_2((c'\otimes_1d')\otimes_2(c''\otimes_1d''))	\arrow[d,"\tn{id}\otimes_2\eta"]\\
			((c\otimes_2c')\otimes_1(d\otimes_2d'))\otimes_2(c''\otimes_1d'')	\arrow[d,"\eta"]	&	(c\otimes_1 d)\otimes_2((c'\otimes_2c'')\otimes_1(d'\otimes_2d''))	\arrow[d,"\eta"]\\
			((c\otimes_2c')\otimes_2c'')\otimes_1((d\otimes_2d')\otimes_2d'') \arrow[r,"\alpha_2\otimes_1\alpha_2"]	&	(c\otimes_2(c'\otimes_2c''))\otimes_1(d\otimes_2(d'\otimes_2d'')),
		\end{tikzcd}
		where $\alpha_2$ denotes the associator for $\otimes_2$.
		\item \label{DCtensor2}
		\begin{tikzcd}
			((c\otimes_1c')\otimes_1c'')\otimes_2((d\otimes_1d')\otimes_1d'')\arrow[r,"\alpha_1\otimes_2\alpha_1"] \arrow[d,"\eta"]	&	(c\otimes_1(c'\otimes_1c''))\otimes_2(d\otimes_1(d'\otimes_1d''))\arrow[d,"\eta"]\\
			((c\otimes_1c')\otimes_2(d\otimes_1d'))\otimes_1(c''\otimes_2d'')\arrow[d,"\eta"]	&	(c\otimes_2d)\otimes_1((c'\otimes_1c'')\otimes_2(d'\otimes_1d'')) \arrow[d,"\eta"]\\
			((c\otimes_2d)\otimes_1(c'\otimes_2d'))\otimes_1 (c''\otimes_2d'')	\arrow[r,"\alpha_1"]	&	(c\otimes_2 d)\otimes_1((c'\otimes_2d')\otimes_1(c''\otimes_2d'')),
		\end{tikzcd}
		where $\alpha_1$ denotes the associator for $\otimes_1$.
	\end{enumerate}
\end{df}

\begin{df}\label{DCbilaxdef}
	Reversing the direction of the compatibility morphisms in Definition \ref{DClax2fold} yields the notion of \emph{oplax 2-fold monoidal category}\footnote{This also corresponds to switching the roles of $\otimes_1$ and $\otimes_2$.}. A category that is both lax and oplax 2-fold monoidal will be called \emph{bilax 2-fold monoidal}. 
	
	If $\cat{C}$ is a bilax 2-fold monoidal category with lax compatibility morphisms $\eta,u_0,u_1,u_2$ and $\zeta,v_0,v_1,v_2$ that satisfy
	$$
	\eta\circ\zeta = \tn{id},\quad u_0\circ v_0=\tn{id}, \quad u_1\circ v_2=\tn{id},\quad v_1\circ u_2=\tn{id},
	$$
	then $\cat{C}$ will be called an \emph{inclusive bilax 2-fold monoidal category}. If, additionally, $u_1$ and $v_2$ are isomorphisms, we will call $\cat{C}$ \emph{strongly inclusive}.
\end{df}

The notion of bilax 2-fold monoidal is known in the community. The notion of (strong) inclusivity is introduced here to capture the structure $\dcentcat{A}$ has. 

\begin{df}\label{DCbraideddef}
	Let $(\cat{C},\otimes_1,\otimes_2)$ be lax 2-fold monoidal and let $\beta_1$ be a braiding for $\otimes_1$. Then $\cat{C}$ is called \emph{horizontally braided lax 2-fold monoidal} if the braiding is such that the following diagrams commute:
	\begin{enumerate}[(a)]
		\item \label{DCu2braid}
		\begin{tikzcd}
			\mathbb{I}_2 \arrow[r,"u_2"] \arrow[d,"u_2"] & \mathbb{I}_2\otimes_1\mathbb{I}_2 \arrow[dl,"\beta_1"]\\
			\mathbb{I}_2\otimes_1\mathbb{I}_2 &,
		\end{tikzcd}
		\item 
		\begin{tikzcd}\label{DCetabraid}
			(c\otimes_1 c')\otimes_2 (d\otimes_1 d')\arrow[r,"\eta"]\arrow[d,"\beta_1\otimes_2\beta_1"] &(c\otimes_2 d)\otimes_1(c'\otimes_2 d')\arrow[d,"\beta_1"]\\
			(c'\otimes_1 c)\otimes_2 (d'\otimes_1 d)\arrow[r,"\eta"] &(c'\otimes_2 d')\otimes_1(c\otimes_2 d)
		\end{tikzcd}.
	\end{enumerate}
	If $\beta_2$ is a braiding for $\otimes_2$ and satisfies the analogous compatibility with $\eta$ and $u_1$, we call $\cat{C}$ \emph{vertically braided}. If $\cat{C}$ is both horizontally and vertically braided we will refer to it as \emph{braided lax 2-fold monoidal}. If, in a braided lax 2-fold monoidal category, the horizontal (resp. vertical) braiding is symmetric the category will be called \emph{horizontally (reps. vertically) symmetric}.
	
	If $\cat{C}$ is additionally a bilax 2-fold monoidal category, and the horizontal or vertical braiding satisfies the corresponding compatibility with the oplax compatibility morphisms, $\cat{C}$ will be called \emph{(horizontally or vertically) braided bilax 2-fold monoidal}.
\end{df}

\begin{df}\label{DClaxtensordef}
	When a lax 2-fold monoidal category $\cat{C}$ is enriched and tensored over $\Vect$, and the monoidal structures are tensor structures, we will call $\cat{C}$ a \emph{lax 2-fold tensor category}.
\end{df}

\section{The Drinfeld Centre as a Lax 2-Fold Monoidal Category}\label{DCdcaslax2foldsect}
This section is devoted to proving the main theorem of this article:

\begin{thm}\label{DCdc2foldmon}
	Let $\cat{A}$ be a symmetric fusion category over an algebraically closed field of characteristic zero and $\dcentcat{A}$ its Drinfeld centre. Denote by $\otimes_c$ and $\otimes_s$ its usual (convolution) and its symmetric tensor product (Definition \ref{DCsymtensobj}) on $\dcentcat{A}$, respectively. Then $(\dcentcat{A},\otimes_c,\otimes_s)$ is a vertically symmetric braided strongly inclusive bilax 2-fold tensor category, cf. Definitions \ref{DClax2fold}, \ref{DCbilaxdef}, \ref{DCbraideddef} and \ref{DClaxtensordef}.
\end{thm}

The structure of this section is as follows. We will first define the compatibility morphisms from Definition \ref{DClax2fold}, we will denote these by $(\eta,u_0,u_1,u_2)$ for the lax direction and $(\zeta, v_0,v_1,v_2)$ for the oplax direction. We will then proceed to check their coherence, combining the necessary proofs for the two cases whenever possible.

\subsection{Lax Compatibility Morphisms}
Our first order of business is to define the lax compatibility morphism for $\otimes_s$ and $\otimes_c$.

\subsubsection{The Comparison Morphism}
The following lemma allows us to define $\eta$ and $\zeta$. 
\begin{lem}
	Let $c,c',d,d'\in\dcentcat{A}$, then the following string diagrams define morphisms in $\dcentcat{A}$:
	\begin{equation}\label{DClaxproj}
	\zeta_{c,d,c',d'}=
	\hbox{
		\begin{tikzpicture}[baseline=(current  bounding  box.center)]
		
		\node (cf) at (-0.7,-0.5){\strut$c\otimes_s d$};
		\node (cf2) at (0.7,-0.5){\strut$c'\otimes_s d'$};
		
		\node (tr) at (-0.7,0.1) {\strut$\bigtriangledown$};
		\coordinate (trd) at (-0.7,0);
		\coordinate (truw) at (-0.8,00.2);
		\coordinate (true) at (-0.6,0.2);
		
		\node (tr2) at (0.7,0.1) {\strut$\bigtriangledown$};
		\coordinate (trd2) at (0.7,0);
		\coordinate (truw2) at (0.6,00.2);
		\coordinate (true2) at (0.8,0.2);

		\node (tr1) at (0,1.2) {\strut$\bigtriangleup$};
		\coordinate (trd1ww) at (-0.13,1.12);
		\coordinate (trd1e) at (0.1,1.12);
		\coordinate (trd1w) at (-0.1,1.12);
		\coordinate (trd1ee) at (0.13,1.12);
		\coordinate (tru1) at (0,1.35);
		\node (c1) at (0,2) {\strut$(c\otimes_c c')\otimes_s(d\otimes_cd')$};

		\begin{knot}[clip width=4]
		\strand [thick] (cf) to (trd);
		\strand [thick] (cf2) to (trd2);
		\strand [thick] (true) to [out=70, in=-70] (trd1e);
		\strand [thick] (true2) to [out=70, in=-70] (trd1ee);
		\strand [thick] (tru1) to (c1);
		\strand [thick] (truw) to [out=110,in=-110] (trd1ww);
		\end{knot}
		\draw [thick] (truw2) to [out=110,in=-110] (trd1w);
		\end{tikzpicture}
	}
	\tn{ and }
	\eta_{c,c',d,d'}=
	\hbox{
		\begin{tikzpicture}[baseline=(current  bounding  box.center)]
		
		\node (cf) at (0,0.5){\strut$(c\otimes_c c')\otimes_s (d\otimes_c d')$};

		\node (tr) at (0,1.1) {\strut$\bigtriangledown$};
		\coordinate (trd) at (0,1);
		\coordinate (truww) at (-0.13,1.2);
		\coordinate (truw) at (-0.1,1.2);
		\coordinate (true) at (0.1,1.2);
		\coordinate (truee) at (0.13,1.2);

		\node (tr1) at (-0.7,2.2) {\strut$\bigtriangleup$};
		\coordinate (trd1w) at (-0.8,2.12);
		\coordinate (trd1e) at (-0.6,2.12);
		\coordinate (tru1) at (-0.7,2.35);
		\node (c1) at (-0.7,3) {\strut$c\otimes_s d$};
		
		\node (tr2) at (0.7,2.2) {\strut$\bigtriangleup$};
		\coordinate (trd2w) at (0.6,2.12);
		\coordinate (trd2e) at (0.8,2.12);
		\coordinate (tru2) at (0.7,2.35);
		\node (c2) at (0.7,3) {\strut$c'\otimes_s d'$};

		\begin{knot}[clip width=4]
		\strand [thick] (cf) to (trd);
		\strand [thick] (true) to [out=70, in=-70] (trd1e);
		\strand [thick] (truee) to [out=70, in=-70] (trd2e);
		\strand [thick] (tru2) to (c2);
		\strand [thick] (tru1) to (c1);
		\strand [thick] (truww) to [out=110,in=-110] (trd1w);
		\end{knot}
		\draw [thick] (truw) to [out=110,in=-110] (trd2w);
		\end{tikzpicture}
	},
	\end{equation}
	respectively. Here the unresolved crossing denotes the symmetry in $\cat{A}$ between the underlying objects in $\cat{A}$ of $d$ and $c'$, c.f. Section \ref{DCsetupsect}. These morphisms exhibit the object $(c\otimes_s d)\otimes_c(c'\otimes_sd')$ as a subobject of $(c\otimes_c c')\otimes_s(d\otimes_c d')$ with inclusion $\zeta_{c,d,c',d'}$ and projection $\eta_{c,c',d,d'}$.
\end{lem}

\begin{proof}
	We have to show that the composite along $(c\otimes_c c')\otimes_s (d\otimes_c d')$ of the two maps is the identity, and that they define morphisms in $\dcentcat{A}$. For the former:
	$$
	\hbox{
		\begin{tikzpicture}[baseline=(current  bounding  box.center)]
		
		\node (cf) at (-0.7,0.5){\strut$c\otimes_s d$};
		\node (cf2) at (0.7,0.5){\strut$c'\otimes_s d'$};
		
		\node (tr) at (-0.7,1.1) {\strut$\bigtriangledown$};
		\coordinate (trd) at (-0.7,1);
		\coordinate (truw) at (-0.8,1.2);
		\coordinate (true) at (-0.6,1.2);
		
		\node (tr2) at (0.7,1.1) {\strut$\bigtriangledown$};
		\coordinate (trd2) at (0.7,1);
		\coordinate (truw2) at (0.6,1.2);
		\coordinate (true2) at (0.8,1.2);

		\node (tr1) at (0,2.2) {\strut$\bigtriangleup$};
		\coordinate (trd1ww) at (-0.13,2.12);
		\coordinate (trd1e) at (0.1,2.12);
		\coordinate (trd1w) at (-0.1,2.12);
		\coordinate (trd1ee) at (0.13,2.12);
		\coordinate (tru1) at (0,2.35);

		\node (tru) at (0,3.1) {\strut$\bigtriangledown$};
		\coordinate (trdu) at (0,3);
		\coordinate (truwwu) at (-0.13,3.2);
		\coordinate (truwu) at (-0.1,3.2);
		\coordinate (trueu) at (0.1,3.2);
		\coordinate (trueeu) at (0.13,3.2);

		\node (tr1u) at (-0.7,4.2) {\strut$\bigtriangleup$};
		\coordinate (trd1wu) at (-0.8,4.12);
		\coordinate (trd1eu) at (-0.6,4.12);
		\coordinate (tru1u) at (-0.7,4.35);
		\node (c1) at (-0.7,5) {\strut$c\otimes_s d$};
		
		\node (tr2u) at (0.7,4.2) {\strut$\bigtriangleup$};
		\coordinate (trd2wu) at (0.6,4.12);
		\coordinate (trd2eu) at (0.8,4.12);
		\coordinate (tru2u) at (0.7,4.35);
		\node (c2) at (0.7,5) {\strut$c'\otimes_s d'$};
		
		\begin{knot}[clip width=4]
		\strand [thick] (cf) to (trd);
		\strand [thick] (cf2) to (trd2);
		\strand [thick] (true) to [out=70, in=-70] (trd1e);
		\strand [thick] (true2) to [out=70, in=-70] (trd1ee);
		\strand [thick] (truw) to [out=110,in=-110] (trd1ww);
		\strand [thick] (tru1) to (trdu);
		\strand [thick] (trueu) to [out=70, in=-70] (trd1eu);
		\strand [thick] (trueeu) to [out=70, in=-70] (trd2eu);
		\strand [thick] (tru2u) to (c2);
		\strand [thick] (tru1u) to (c1);
		\strand [thick] (truwwu) to [out=110,in=-110] (trd1wu);
		\end{knot}
		\draw [thick] (truw2) to [out=110,in=-110] (trd1w);
		\draw [thick] (truwu) to [out=110,in=-110] (trd2wu);
		\end{tikzpicture}
	}
	=
	\hbox{
		\begin{tikzpicture}[baseline=(current  bounding  box.center)]
		
		\node (cf) at (-0.7,0.5){\strut$c\otimes_s d$};
		\node (cf2) at (0.7,0.5){\strut$c'\otimes_s d'$};
		
		\node (tr) at (-0.7,1.1) {\strut$\bigtriangledown$};
		\coordinate (trd) at (-0.7,1);
		\coordinate (truw) at (-0.8,1.2);
		\coordinate (true) at (-0.6,1.2);
		
		\node (tr2) at (0.7,1.1) {\strut$\bigtriangledown$};
		\coordinate (trd2) at (0.7,1);
		\coordinate (truw2) at (0.6,1.2);
		\coordinate (true2) at (0.8,1.2);

		\coordinate (trd1e) at (0.1,1.82);
		\coordinate (trd1w) at (-0.1,1.82);

		\coordinate (truwu) at (-0.1,3.5);
		\coordinate (trueu) at (0.1,3.5);

		\node (tr1u) at (-0.7,4.2) {\strut$\bigtriangleup$};
		\coordinate (trd1wu) at (-0.8,4.12);
		\coordinate (trd1eu) at (-0.6,4.12);
		\coordinate (tru1u) at (-0.7,4.35);
		\node (c1) at (-0.7,5) {\strut$c\otimes_s d$};
		
		\node (tr2u) at (0.7,4.2) {\strut$\bigtriangleup$};
		\coordinate (trd2wu) at (0.6,4.12);
		\coordinate (trd2eu) at (0.8,4.12);
		\coordinate (tru2u) at (0.7,4.35);
		\node (c2) at (0.7,5) {\strut$c'\otimes_s d'$};

		\coordinate (west) at (-1.1,2.7);
		\coordinate (north) at (0,3);
		\coordinate (east) at (1.2,2.7);
		\coordinate (south) at (0,2.4);
		
		\begin{knot}[clip width=4,clip radius=2pt]
		\strand [blue, thick] (west) to [out=90,in=-180] (north) to [out=0,in=90] (east) to [out=-90,in=0] (south) to [out=-180,in=-90] (west);
		\strand [thick] (cf) to (trd);
		\strand [thick] (cf2) to (trd2);
		\strand [thick] (true) to [out=70, in=-90] (trd1e) to (trueu) to [out=90, in=-70] (trd1eu);
		\strand [thick] (true2) to [out=90, in=-90] (trd2eu);
		\strand [thick] (truw) to [out=90,in=-90] (trd1wu);
		\strand [thick] (tru2u) to (c2);
		\strand [thick] (tru1u) to (c1);
		\strand [thick] (trd1w) to (truwu);
		\flipcrossings{5,7,2,4}
		\end{knot}
		\draw [thick] (truw2) to [out=110,in=-90] (trd1w);
		\draw [thick] (truwu) to [out=90,in=-110] (trd2wu);
		\end{tikzpicture}
	},
	$$
	where we have replaced a projection followed by an inclusion with the idempotent from Lemma \ref{DCidempotent}, cf. Equation \eqref{DCinclprojprops}. We can now pull the top of the ring up and the bottom of the ring down, using the way the unresolved and resolved crossings interact, see Equation \eqref{DCcrossinginteraction}, to get:
	$$
\hbox{
	\begin{tikzpicture}[baseline=(current  bounding  box.center)]
	
	\node (cf) at (-0.7,0.5){\strut$c\otimes_s d$};
	\node (cf2) at (0.7,0.5){\strut$c'\otimes_s d'$};
	
	\node (tr) at (-0.7,1.1) {\strut$\bigtriangledown$};
	\coordinate (trd) at (-0.7,1);
	\coordinate (truw) at (-0.8,1.2);
	\coordinate (true) at (-0.6,1.2);
	
	\node (tr2) at (0.7,1.1) {\strut$\bigtriangledown$};
	\coordinate (trd2) at (0.7,1);
	\coordinate (truw2) at (0.6,1.2);
	\coordinate (true2) at (0.8,1.2);

	\coordinate (trd1e) at (0.1,2.5);
	\coordinate (trd1w) at (-0.1,2.5);

	\coordinate (truwu) at (-0.1,2.9);
	\coordinate (trueu) at (0.1,2.9);
	
	\coordinate (bw) at (-0.1,2.55);
	\coordinate (be) at (0.1,2.55);

	\node (tr1u) at (-0.7,4.2) {\strut$\bigtriangleup$};
	\coordinate (trd1wu) at (-0.8,4.12);
	\coordinate (trd1eu) at (-0.6,4.12);
	\coordinate (tru1u) at (-0.7,4.35);
	\node (c1) at (-0.7,5) {\strut$c\otimes_s d$};
	
	\node (tr2u) at (0.7,4.2) {\strut$\bigtriangleup$};
	\coordinate (trd2wu) at (0.6,4.12);
	\coordinate (trd2eu) at (0.8,4.12);
	\coordinate (tru2u) at (0.7,4.35);
	\node (c2) at (0.7,5) {\strut$c'\otimes_s d'$};

	\coordinate (west) at (-1.1,2.7);
	\coordinate (north) at (0,3.7);
	\coordinate (east) at (1.2,2.7);
	\coordinate (south) at (0,1.6);
	
	\begin{knot}[clip width=4,clip radius=2pt]
	\strand [blue, thick] (west) to [out=90,in=-180] (north) to [out=0,in=90] (east) to [out=-90,in=0] (south) to [out=-180,in=-90] (west);
	\strand [thick] (cf) to (trd);
	\strand [thick] (cf2) to (trd2);
	\strand [thick] (true) to [out=70, in=-90] (trd1e);
	\strand [thick] (trd1e) to (trueu);
	\strand [thick] (trueu) to [out=90, in=-70] (trd1eu);
	\strand [thick] (true2) to [out=90, in=-90] (trd2eu);
	\strand [thick] (truw) to [out=90,in=-90] (trd1wu);
	\strand [thick] (tru2u) to (c2);
	\strand [thick] (tru1u) to (c1);
	\strand [thick] (bw) to [out=-90,in=90] (trd1w) to [out=-90,in=70] (truw2);
	\strand [thick] (bw) to [out=90,in=-90] (truwu);
	\strand [thick] (truwu) to [out=90,in=-110] (trd2wu);
	\flipcrossings{1,5,4,8}
	\end{knot}
	\end{tikzpicture}
}
	=
	\hbox{
		\begin{tikzpicture}[baseline=(current  bounding  box.center)]
		
		\node (cf) at (-0.7,0.5){\strut$c\otimes_s d$};
		\node (cf2) at (0.7,0.5){\strut$c'\otimes_s d'$};
		
		\node (tr) at (-0.7,1.6) {\strut$\bigtriangledown$};
		\coordinate (trd) at (-0.7,1.5);
		\coordinate (truw) at (-0.8,1.7);
		\coordinate (true) at (-0.6,1.7);
		
		\node (tr2) at (0.7,1.6) {\strut$\bigtriangledown$};
		\coordinate (trd2) at (0.7,1.5);
		\coordinate (truw2) at (0.6,1.7);
		\coordinate (true2) at (0.8,1.7);

		\node (tr1u) at (-0.7,3.7) {\strut$\bigtriangleup$};
		\coordinate (trd1wu) at (-0.8,3.62);
		\coordinate (trd1eu) at (-0.6,3.62);
		\coordinate (tru1u) at (-0.7,3.85);
		\node (c1) at (-0.7,5) {\strut$c\otimes_s d$};
		
		\node (tr2u) at (0.7,3.7) {\strut$\bigtriangleup$};
		\coordinate (trd2wu) at (0.6,3.62);
		\coordinate (trd2eu) at (0.8,3.62);
		\coordinate (tru2u) at (0.7,3.85);
		\node (c2) at (0.7,5) {\strut$c'\otimes_s d'$};

		\coordinate (west) at (-1.1,2.7);
		\coordinate (north) at (0,3.1);
		\coordinate (east) at (1.2,2.7);
		\coordinate (south) at (0,2.3);
		
		\begin{knot}[clip width=4,clip radius=2pt]
		\strand [blue, thick] (west) to [out=90,in=-180] (north) to [out=0,in=90] (east) to [out=-90,in=0] (south) to [out=-180,in=-90] (west);
		\strand [thick] (cf) to (trd);
		\strand [thick] (cf2) to (trd2);
		\strand [thick] (true2) to [out=90, in=-90] (trd2eu);
		\strand [thick] (trd1eu) to [out=-70,in=70] (true);
		\strand [thick] (tru2u) to (c2);
		\strand [thick] (tru1u) to (c1);
		\strand [thick] (trd2wu) to [out=-110,in=110] (truw2);
		\strand [thick] (truw) to (trd1wu);
		\flipcrossings{5,7,2,4}
		\end{knot}
		\end{tikzpicture}
	}
	=
	\hbox{
		\begin{tikzpicture}[baseline=(current  bounding  box.center)]
		
		\node (cf) at (-0.7,0.5){\strut$c\otimes_s d$};
		\node (cf2) at (0.7,0.5){\strut$c'\otimes_s d'$};
		
		\node (tr) at (-0.7,1.6) {\strut$\bigtriangledown$};
		\coordinate (trd) at (-0.7,1.5);
		\coordinate (truw) at (-0.8,1.7);
		\coordinate (true) at (-0.6,1.7);
		
		\node (tr2) at (0.7,1.6) {\strut$\bigtriangledown$};
		\coordinate (trd2) at (0.7,1.5);
		\coordinate (truw2) at (0.6,1.7);
		\coordinate (true2) at (0.8,1.7);

		\node (tr1u) at (-0.7,3.7) {\strut$\bigtriangleup$};
		\coordinate (trd1wu) at (-0.8,3.62);
		\coordinate (trd1eu) at (-0.6,3.62);
		\coordinate (tru1u) at (-0.7,3.85);
		\node (c1) at (-0.7,5) {\strut$c\otimes_s d$};
		
		\node (tr2u) at (0.7,3.7) {\strut$\bigtriangleup$};
		\coordinate (trd2wu) at (0.6,3.62);
		\coordinate (trd2eu) at (0.8,3.62);
		\coordinate (tru2u) at (0.7,3.85);
		\node (c2) at (0.7,5) {\strut$c'\otimes_s d'$};

		\coordinate (west) at (-1.1,2.7);
		\coordinate (north) at (0,3.1);
		\coordinate (east) at (1.2,2.7);
		\coordinate (south) at (0,2.3);
		
		\begin{knot}[clip width=4,clip radius=2pt]
		\strand [blue, thick] (west) to [out=90,in=-180] (north) to [out=0,in=90] (east) to [out=-90,in=0] (south) to [out=-180,in=-90] (west);
		\strand [thick] (cf) to (trd);
		\strand [thick] (cf2) to (trd2);
		\strand [thick] (true2) to [out=90, in=-90] (trd2eu);
		\strand [thick] (trd1eu) to [out=-70,in=70] (true);
		\strand [thick] (tru2u) to (c2);
		\strand [thick] (tru1u) to (c1);
		\strand [thick] (trd2wu) to [out=-110,in=110] (truw2);
		\strand [thick] (truw) to (trd1wu);
		\flipcrossings{1,3,2,4}
		\end{knot}
		\end{tikzpicture}
	},
	$$
	where in the last step we used Lemma \ref{DCslicing} to interchange the crossings between the ring and the strands between the inclusion and projection. In the last diagram, the ring evaluates to 1, and the diagram evaluates to the identity on $(c\otimes_s d)\otimes_c (c'\otimes_sd')$, see Equation \eqref{DCinclprojprops}. To show that the inclusion and projection are morphisms in $\dcentcat{A}$, ie. that they commute with the half-braidings, we compute, for $a\in\cat{A}\subset\dcentcat{A}$:
	\begin{eqnarray*}
		\hbox{
			\begin{tikzpicture}[baseline=(current  bounding  box.center)]
			
			\node (cf) at (-0.7,-1.5){\strut$c\otimes_s d$};
			\node (cf2) at (0.7,-1.5){\strut$c'\otimes_s d'$};
			
			\node (tr) at (-0.7,0.1) {\strut$\bigtriangledown$};
			\coordinate (trd) at (-0.7,0);
			\coordinate (truw) at (-0.8,00.2);
			\coordinate (true) at (-0.6,0.2);
			
			\node (tr2) at (0.7,0.1) {\strut$\bigtriangledown$};
			\coordinate (trd2) at (0.7,0);
			\coordinate (truw2) at (0.6,00.2);
			\coordinate (true2) at (0.8,0.2);

			\node (tr1) at (0,1.2) {\strut$\bigtriangleup$};
			\coordinate (trd1ww) at (-0.13,1.12);
			\coordinate (trd1e) at (0.1,1.12);
			\coordinate (trd1w) at (-0.1,1.12);
			\coordinate (trd1ee) at (0.13,1.12);
			\coordinate (tru1) at (0,1.35);
			\node (c1) at (0,2) {\strut$(c\otimes_c c')\otimes_s(d\otimes_cd')$};
			
			\node (a) at (-1.5,-1.5){\strut$a$};
			\coordinate (lc) at (-1.5,-1.25);
			\coordinate (rc) at (1.6,0.1);
			\coordinate (ao) at (1.6,2);
			
			\begin{knot}[clip width=4]
			\strand [thick] (cf) to (trd);
			\strand [thick] (cf2) to (trd2);
			\strand [thick] (true) to [out=70, in=-70] (trd1e);
			\strand [thick] (true2) to [out=70, in=-70] (trd1ee);
			\strand [thick] (tru1) to (c1);
			\strand [thick] (truw) to [out=110,in=-110] (trd1ww);
			\strand [thick] (a) to [out=90,in=-90] (lc) to [out=90,in=-90] (rc) to [out=90,in=-90] (ao);
			\end{knot}
			\draw [thick] (truw2) to [out=110,in=-110] (trd1w);
			\end{tikzpicture}
		}
		=&
		\hbox{
			\begin{tikzpicture}[baseline=(current  bounding  box.center)]
			
			\node (cf) at (-0.7,-0.5){\strut$c\otimes_s d$};
			\node (cf2) at (0.7,-0.5){\strut$c'\otimes_s d'$};
			
			\node (tr) at (-0.7,0.1) {\strut$\bigtriangledown$};
			\coordinate (trd) at (-0.7,0);
			\coordinate (truw) at (-0.8,00.2);
			\coordinate (true) at (-0.6,0.2);
			
			\node (tr2) at (0.7,0.1) {\strut$\bigtriangledown$};
			\coordinate (trd2) at (0.7,0);
			\coordinate (truw2) at (0.6,00.2);
			\coordinate (true2) at (0.8,0.2);

			\node (tr1) at (0,2.2) {\strut$\bigtriangleup$};
			\coordinate (trd1ww) at (-0.13,2.12);
			\coordinate (trd1e) at (0.1,2.12);
			\coordinate (trd1w) at (-0.1,2.12);
			\coordinate (trd1ee) at (0.13,2.12);
			\coordinate (tru1) at (0,2.35);
			\node (c1) at (0,3) {\strut$(c\otimes_c c')\otimes_s(d\otimes_cd')$};
			
			\node (a) at (-1.5,-0.5){\strut$a$};
			\coordinate (lc) at (-1.5,0);
			\coordinate (rc) at (1.6,2);
			\coordinate (ao) at (1.6,3);
			
			\begin{knot}[clip width=4,clip radius=3pt]
			\strand [thick] (cf) to (trd);
			\strand [thick] (cf2) to (trd2);
			\strand [thick] (true) to [out=90, in=-100] (trd1e);
			\strand [thick] (true2) to [out=70, in=-70] (trd1ee);
			\strand [thick] (tru1) to (c1);
			\strand [thick] (truw) to [out=110,in=-110] (trd1ww);
			\strand [thick] (truw2) to [out=90,in=-80] (trd1w);
			\strand [thick] (a) to [out=90,in=-90] (lc) to [out=90,in=-90] (rc) to [out=90,in=-90] (ao);
			\flipcrossings{3,4}
			\end{knot}
			\end{tikzpicture}
		}
		=
		\hbox{
			\begin{tikzpicture}[baseline=(current  bounding  box.center)]
			
			\node (cf) at (-0.7,-0.5){\strut$c\otimes_s d$};
			\node (cf2) at (0.7,-0.5){\strut$c'\otimes_s d'$};
			
			\node (tr) at (-0.7,0.1) {\strut$\bigtriangledown$};
			\coordinate (trd) at (-0.7,0);
			\coordinate (truw) at (-0.8,00.2);
			\coordinate (true) at (-0.6,0.2);
			
			\node (tr2) at (0.7,0.1) {\strut$\bigtriangledown$};
			\coordinate (trd2) at (0.7,0);
			\coordinate (truw2) at (0.6,00.2);
			\coordinate (true2) at (0.8,0.2);

			\node (tr1) at (0,2.2) {\strut$\bigtriangleup$};
			\coordinate (trd1ww) at (-0.13,2.12);
			\coordinate (trd1e) at (0.1,2.12);
			\coordinate (trd1w) at (-0.1,2.12);
			\coordinate (trd1ee) at (0.13,2.12);
			\coordinate (tru1) at (0,2.35);
			\node (c1) at (0,3) {\strut$(c\otimes_c c')\otimes_s(d\otimes_cd')$};
			
			\node (a) at (-1.5,-0.5){\strut$a$};
			\coordinate (lc) at (-1.5,0.8);
			\coordinate (rc) at (1.6,2.5);
			\coordinate (ao) at (1.6,3);
			
			\coordinate (mide) at (0.25,1);
			\coordinate (midw) at (-0.25,1);
			
			\begin{knot}[clip width=4,clip radius =3pt]
			\strand [thick] (cf) to (trd);
			\strand [thick] (cf2) to (trd2);
			\strand [thick] (true) to [out=70,in=-90] (mide) to [out=90,in=-90] (trd1e);
			\strand [thick] (true2) to [out=70, in=-70] (trd1ee);
			\strand [thick] (tru1) to (c1);
			\strand [thick] (truw) to [out=110,in=-110] (trd1ww);
			\strand [thick] (a) to [out=90,in=-90] (lc) to [out=90,in=-90] (rc) to [out=90,in=-90] (ao);
			\strand [thick] (midw) to [out=90,in=-90] (trd1w);
			\flipcrossings{3}
			\end{knot}
			\draw [thick] (truw2) to [out=110,in=-90] (midw);
			\end{tikzpicture}
		}\\
		=&
		\hbox{
			\begin{tikzpicture}[baseline=(current  bounding  box.center)]
			
			\node (cf) at (-0.7,0.5){\strut$c\otimes_s d$};
			\node (cf2) at (0.7,0.5){\strut$c'\otimes_s d'$};
			
			\node (tr) at (-0.7,1.1) {\strut$\bigtriangledown$};
			\coordinate (trd) at (-0.7,1);
			\coordinate (truw) at (-0.8,1.2);
			\coordinate (true) at (-0.6,1.2);
			
			\node (tr2) at (0.7,1.1) {\strut$\bigtriangledown$};
			\coordinate (trd2) at (0.7,1);
			\coordinate (truw2) at (0.6,1.2);
			\coordinate (true2) at (0.8,1.2);

			\node (tr1) at (0,2.2) {\strut$\bigtriangleup$};
			\coordinate (trd1ww) at (-0.13,2.12);
			\coordinate (trd1e) at (0.1,2.12);
			\coordinate (trd1w) at (-0.1,2.12);
			\coordinate (trd1ee) at (0.13,2.12);
			\coordinate (tru1) at (0,2.35);
			\node (c1) at (0,4) {\strut$(c\otimes_c c')\otimes_s(d\otimes_cd')$};
			
			\node (a) at (-1.5,0.5){\strut$a$};
			\coordinate (lc) at (-1.5,2);
			\coordinate (rc) at (1.6,3.75);
			\coordinate (ao) at (1.6,4);
			
			\begin{knot}[clip width=4,clip radius =3pt]
			\strand [thick] (cf) to (trd);
			\strand [thick] (cf2) to (trd2);
			\strand [thick] (true) to [out=70, in=-70] (trd1e);
			\strand [thick] (true2) to [out=70, in=-70] (trd1ee);
			\strand [thick] (tru1) to (c1);
			\strand [thick] (truw) to [out=110,in=-110] (trd1ww);
			\strand [thick] (a) to [out=90,in=-90] (lc) to [out=90,in=-90] (rc) to [out=90,in=-90] (ao);
			\end{knot}
			\draw [thick] (truw2) to [out=110,in=-110] (trd1w);
			\end{tikzpicture}
		},
	\end{eqnarray*}
	where we made repeated use of slicing, and use of Equation \eqref{DCcrossinginteraction} in the second equality.
\end{proof}

\begin{lem}
	The morphisms from Equation \eqref{DClaxproj} combine to give natural transformations.
\end{lem}

\begin{proof}
	Let $f\colon c_1\rar c_2$, $f'\colon c_1'\rar c_2'$, $g\colon d_1\rar d_2$ and $g'\colon d_1'\rar d_2'$ be morphisms in $\dcentcat{A}$. Using the definition of $\otimes_s$ on morphisms, Equation \eqref{DCsymtensmorph}, we compute:
	$$
	
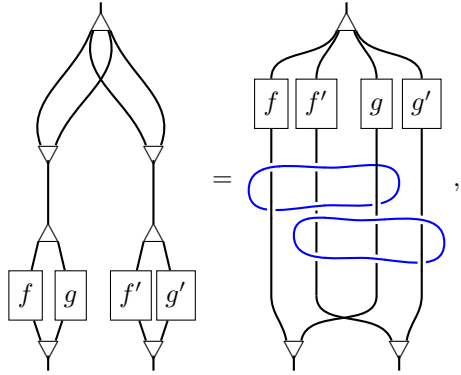
\hbox{
		\begin{tikzpicture}[baseline=(current  bounding  box.center)]
		
		\coordinate (csd) at (-0.7,-2.8);
		\coordinate (cpsdp) at (0.7,-2.8);
		
		\node (trll) at (-0.7,-2.5) {\strut$\bigtriangledown$};
		\coordinate (trdll) at (-0.7,-2.6);
		\coordinate (truwll) at (-0.8,-2.4);
		\coordinate (truell) at (-0.6,-2.4);
		
		\node (tr2lr) at (0.7,-2.5) {\strut$\bigtriangledown$};
		\coordinate (trd2lr) at (0.7,-2.6);
		\coordinate (truw2lr) at (0.6,-2.4);
		\coordinate (true2lr) at (0.8,-2.4);

		\node (f) at (-1,-1.8)[draw,minimum height=18pt]{\strut$f$};
		\node (g) at (-0.4,-1.8)[draw,minimum height=18pt]{\strut$g$};
		\node (fp) at (0.4,-1.8)[draw,minimum height=18pt]{\strut$f'$};
		\node (gp) at (1,-1.8)[draw,minimum height=18pt]{\strut$g'$};
		
		\node (tr1ll) at (-0.7,-1) {\strut$\bigtriangleup$};
		\coordinate (trd1llw) at (-0.83,-1.08);
		\coordinate (trd1lle) at (-0.57,-1.08);
		\coordinate (trllu1) at (-0.7,-0.85);
		
		\node (tr1lr) at (0.7,-1) {\strut$\bigtriangleup$};
		\coordinate (trd1lre) at (0.83,-1.08);
		\coordinate (trd1lrw) at (0.57,-1.08);
		\coordinate (trlru1) at (0.7,-0.85);

		\node (tr) at (-0.7,0.1) {\strut$\bigtriangledown$};
		\coordinate (trd) at (-0.7,0);
		\coordinate (truw) at (-0.8,00.2);
		\coordinate (true) at (-0.6,0.2);
		
		\node (tr2) at (0.7,0.1) {\strut$\bigtriangledown$};
		\coordinate (trd2) at (0.7,0);
		\coordinate (truw2) at (0.6,00.2);
		\coordinate (true2) at (0.8,0.2);

		\node (tr1) at (0,1.8) {\strut$\bigtriangleup$};
		\coordinate (trd1ww) at (-0.13,1.72);
		\coordinate (trd1e) at (0.1,1.72);
		\coordinate (trd1w) at (-0.1,1.72);
		\coordinate (trd1ee) at (0.13,1.72);
		\coordinate (tru1) at (0,1.95);
		\coordinate (c1) at (0,2.1);

		\begin{knot}[clip width=4,clip radius =3pt]
		\strand [thick] (csd) to (trdll);
		\strand [thick] (cpsdp) to (trd2lr);
		\strand [thick] (truwll) to (f);
		\strand [thick] (truell) to (g);
		\strand [thick] (truw2lr) to (fp);
		\strand [thick] (true2lr) to (gp);
		\strand [thick] (f) to (trd1llw);
		\strand [thick] (g) to (trd1lle);
		\strand [thick] (fp) to (trd1lrw);
		\strand [thick] (gp) to (trd1lre);			
		\strand [thick] (trllu1) to (trd);
		\strand [thick] (trlru1) to (trd2);
		\strand [thick] (true) to [out=70, in=-70] (trd1e);
		\strand [thick] (true2) to [out=70, in=-70] (trd1ee);
		\strand [thick] (tru1) to (c1);
		\strand [thick] (truw) to [out=110,in=-110] (trd1ww);
		\end{knot}
		\draw [thick] (truw2) to [out=110,in=-110] (trd1w);
		\end{tikzpicture}
	}
	=
	\hbox{
		\begin{tikzpicture}[baseline=(current  bounding  box.center)]
		
		\coordinate (csd) at (-0.7,-2.8);
		\coordinate (cpsdp) at (0.7,-2.8);
		
		\node (trll) at (-0.7,-2.5) {\strut$\bigtriangledown$};
		\coordinate (trdll) at (-0.7,-2.6);
		\coordinate (truwll) at (-0.8,-2.4);
		\coordinate (truell) at (-0.6,-2.4);
		
		\node (tr2lr) at (0.7,-2.5) {\strut$\bigtriangledown$};
		\coordinate (trd2lr) at (0.7,-2.6);
		\coordinate (truw2lr) at (0.6,-2.4);
		\coordinate (true2lr) at (0.8,-2.4);
		
		\coordinate (lf) at (-1,-1.8);
		\coordinate (lg) at (0.4,-1.8);
		\coordinate (lfp) at (-0.4,-1.8);
		\coordinate (lgp) at (1,-1.8);

		\node (f) at (-1,0.75)[draw,minimum height=18pt]{\strut$f$};
		\node (g) at (0.4,0.75)[draw,minimum height=18pt]{\strut$g$};
		\node (fp) at (-0.4,0.75)[draw,minimum height=18pt]{\strut$f'$};
		\node (gp) at (1,0.75)[draw,minimum height=18pt]{\strut$g'$};

		\node (tr1) at (0,1.8) {\strut$\bigtriangleup$};
		\coordinate (trd1ww) at (-0.13,1.72);
		\coordinate (trd1e) at (0.1,1.72);
		\coordinate (trd1w) at (-0.1,1.72);
		\coordinate (trd1ee) at (0.13,1.72);
		\coordinate (tru1) at (0,1.95);
		\coordinate (c1) at (0,2.1);

		\coordinate (west) at (-1.3,-0.4);
		\coordinate (north) at (-0.2,-0.1);
		\coordinate (east) at (0.7,-0.3);
		\coordinate (south) at (-0.2,-0.6);
		
		\coordinate (west1) at (-0.7,-1.0);
		\coordinate (north1) at (0.2,-0.8);
		\coordinate (east1) at (1.3,-1.1);
		\coordinate (south1) at (0.2,-1.3);

		\begin{knot}[clip width=4,clip radius =3pt]
		\strand [blue, thick] (west) to [out=90,in=-180] (north) to [out=0,in=90] (east) to [out=-90,in=0](south) to [out=-180,in=-90] (west);
		\strand [blue, thick] (west1) to [out=90,in=-180] (north1) to [out=0,in=90] (east1) to[out=-90,in=0](south1) to [out=-180,in=-90] (west1);
		\strand [thick] (csd) to (trdll);
		\strand [thick] (cpsdp) to (trd2lr);
		\strand [thick] (truwll) to [out=100,in=-90] (lf) to (f);
		\strand [thick] (truell) to [out=80,in=-90](lg) to (g);
		\strand [thick] (true2lr) to [out=80,in=-90] (lgp) to (gp);
		\strand [thick] (g) to [out=90, in=-70] (trd1e);
		\strand [thick] (gp) to [out=90, in=-70] (trd1ee);
		\strand [thick] (tru1) to (c1);
		\strand [thick] (f) to [out=90,in=-110] (trd1ww);
		\strand [thick] (fp) to [out=90,in=-110] (trd1w);
		\strand [thick] (lfp) to (fp);
		\flipcrossings{2,3,12,9}
		\end{knot}
		\draw [thick] (truw2lr) to [in=-90,out=110] (lfp);
		\end{tikzpicture}
	},
	$$
	where we replaced projections followed by inclusions by the idempotent from Lemma \ref{DCidempotent}, and used the naturality of the symmetry and braiding, as well as Equation \eqref{DCcrossinginteraction}, to move the rings down and the morphisms up. We can now use cloaking (Lemma \ref{DCcloaking}) for the bottom strand of the top ring with the bottom ring to get:
	$$
	
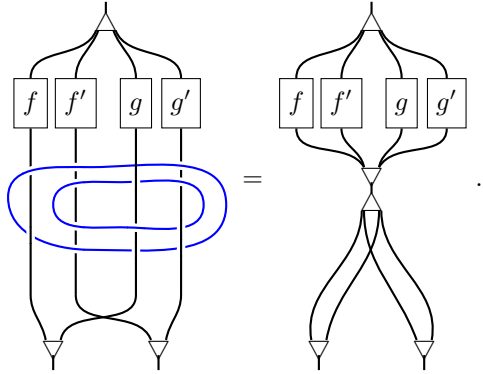
\hbox{
		\begin{tikzpicture}[baseline=(current  bounding  box.center)]
		
		\coordinate (csd) at (-0.7,-2.8);
		\coordinate (cpsdp) at (0.7,-2.8);
		
		\node (trll) at (-0.7,-2.5) {\strut$\bigtriangledown$};
		\coordinate (trdll) at (-0.7,-2.6);
		\coordinate (truwll) at (-0.8,-2.4);
		\coordinate (truell) at (-0.6,-2.4);
		
		\node (tr2lr) at (0.7,-2.5) {\strut$\bigtriangledown$};
		\coordinate (trd2lr) at (0.7,-2.6);
		\coordinate (truw2lr) at (0.6,-2.4);
		\coordinate (true2lr) at (0.8,-2.4);
		
		\coordinate (lf) at (-1,-1.8);
		\coordinate (lg) at (0.4,-1.8);
		\coordinate (lfp) at (-0.4,-1.8);
		\coordinate (lgp) at (1,-1.8);

		\node (f) at (-1,0.75)[draw,minimum height=18pt]{\strut$f$};
		\node (g) at (0.4,0.75)[draw,minimum height=18pt]{\strut$g$};
		\node (fp) at (-0.4,0.75)[draw,minimum height=18pt]{\strut$f'$};
		\node (gp) at (1,0.75)[draw,minimum height=18pt]{\strut$g'$};

		\node (tr1) at (0,1.8) {\strut$\bigtriangleup$};
		\coordinate (trd1ww) at (-0.13,1.72);
		\coordinate (trd1e) at (0.1,1.72);
		\coordinate (trd1w) at (-0.1,1.72);
		\coordinate (trd1ee) at (0.13,1.72);
		\coordinate (tru1) at (0,1.95);
		\coordinate (c1) at (0,2.1);
		
		\coordinate (west) at (-1.3,-0.5);
		\coordinate (north) at (-0.2,-0.1);
		\coordinate (east) at (1.6,-0.6);
		\coordinate (south) at (-0.2,-1.2);
		
		\coordinate (west1) at (-0.7,-0.6);
		\coordinate (north1) at (0.2,-0.3);
		\coordinate (east1) at (1.3,-0.6);
		\coordinate (south1) at (0.2,-0.9);

		\begin{knot}[clip width=4,clip radius =3pt]
		\strand [blue, thick] (west) to [out=90,in=-180] (north) to [out=0,in=90] (east) to [out=-90,in=0](south) to [out=-180,in=-90] (west);
		\strand [blue, thick] (west1) to [out=90,in=-180] (north1) to [out=0,in=90] (east1) to[out=-90,in=0](south1) to [out=-180,in=-90] (west1);
		\strand [thick] (csd) to (trdll);
		\strand [thick] (cpsdp) to (trd2lr);
		\strand [thick] (truwll) to [out=100,in=-90] (lf) to (f);
		\strand [thick] (truell) to [out=80,in=-90](lg) to (g);
		\strand [thick] (true2lr) to [out=80,in=-90] (lgp) to (gp);
		\strand [thick] (g) to [out=90, in=-70] (trd1e);
		\strand [thick] (gp) to [out=90, in=-70] (trd1ee);
		\strand [thick] (tru1) to (c1);
		\strand [thick] (f) to [out=90,in=-110] (trd1ww);
		\strand [thick] (fp) to [out=90,in=-110] (trd1w);
		\strand [thick] (lfp) to (fp);
		\flipcrossings{2,3,14,11,8,5}
		\end{knot}
		\draw [thick] (truw2lr) to [in=-90,out=110] (lfp);
		\end{tikzpicture}
	}
	=
	\hbox{
		\begin{tikzpicture}[baseline=(current  bounding  box.center)]
		
		\coordinate (csd) at (-0.7,-2.8);
		\coordinate (cpsdp) at (0.7,-2.8);
		
		\node (trll) at (-0.7,-2.5) {\strut$\bigtriangledown$};
		\coordinate (trdll) at (-0.7,-2.6);
		\coordinate (truwll) at (-0.8,-2.4);
		\coordinate (truell) at (-0.6,-2.4);
		
		\node (tr2lr) at (0.7,-2.5) {\strut$\bigtriangledown$};
		\coordinate (trd2lr) at (0.7,-2.6);
		\coordinate (truw2lr) at (0.6,-2.4);
		\coordinate (true2lr) at (0.8,-2.4);

		\node (tr1d) at (0,-0.6) {\strut$\bigtriangleup$};
		\coordinate (lf) at (-0.13,-0.68);
		\coordinate (lg) at (0.1,-0.68);
		\coordinate (lfp) at (-0.1,-0.68);
		\coordinate (lgp) at (0.13,-0.68);
		\coordinate (tru1d) at (0,-0.45);

		\node (trmu) at (0,-0.2) {\strut$\bigtriangledown$};
		\coordinate (trdmu) at (0,-0.3);
		\coordinate (truwmu) at (-0.1,-0.1);
		\coordinate (truemu) at (0.1,-0.1);

		\node (f) at (-1,0.75)[draw,minimum height=18pt]{\strut$f$};
		\node (g) at (0.4,0.75)[draw,minimum height=18pt]{\strut$g$};
		\node (fp) at (-0.4,0.75)[draw,minimum height=18pt]{\strut$f'$};
		\node (gp) at (1,0.75)[draw,minimum height=18pt]{\strut$g'$};

		\node (tr1) at (0,1.8) {\strut$\bigtriangleup$};
		\coordinate (trd1ww) at (-0.13,1.72);
		\coordinate (trd1e) at (0.1,1.72);
		\coordinate (trd1w) at (-0.1,1.72);
		\coordinate (trd1ee) at (0.13,1.72);
		\coordinate (tru1) at (0,1.95);
		\coordinate (c1) at (0,2.1);

		\begin{knot}[clip width=4,clip radius =3pt]
		\strand [thick] (csd) to (trdll);
		\strand [thick] (cpsdp) to (trd2lr);
		\strand [thick] (truwll) to [out=100,in=-90] (lf);
		\strand [thick] (truwmu) to [out=120,in=-90] (f);
		\strand [thick] (truwmu) to [out=100,in=-90] (fp);
		\strand [thick] (truemu) to [out=80,in=-90] (g);
		\strand [thick] (truemu) to [out=60,in=-90] (gp);
		\strand [thick] (truell) to [out=80,in=-90](lg);
		\strand [thick] (true2lr) to [out=80,in=-90] (lgp);
		\strand [thick] (g) to [out=90, in=-70] (trd1e);
		\strand [thick] (gp) to [out=90, in=-70] (trd1ee);
		\strand [thick] (tru1) to (c1);
		\strand [thick] (f) to [out=90,in=-110] (trd1ww);
		\strand [thick] (fp) to [out=90,in=-110] (trd1w);
		\strand [thick] (tru1d) to (trdmu);
		\end{knot}
		\draw [thick] (truw2lr) to [in=-90,out=110] (lfp);
		\end{tikzpicture}
	}.
	$$
	The last equality follows from the fact that the rings are transparent to each other. This we means we can bring the smaller ring down using Equation \eqref{DCcrossinginteraction}, and cancel it with the inclusion, cf. Equation \eqref{DCinclprojprops}. The larger ring then gives the inclusion-projection composite in the middle of the last diagram, using Equation \eqref{DCinclprojprops} again. 
	
	The proof of naturality for the other map in Equation \eqref{DClaxproj} is obtained by reading the diagrams top to bottom.
\end{proof}

With this Lemma in hand, we can define $\eta$ and $\zeta$ to be the natural transformations with components $\eta_{c,c',d,d'}$ and $\zeta_{c,d,c',d'}$ defined in Equation \eqref{DClaxproj}.

\subsubsection{Unit Compatiblity}
We will now produce the required morphisms $u_0,u_1,u_2$ and $v_0,v_1,v_2$ that relate the units for the two tensor products on $\dcentcat{A}$. 

We start with the following observation:
\begin{lem}
	The following are morphisms in $\dcentcat{A}$
	\begin{eqnarray*}
		\hbox{
			\begin{tikzpicture}[baseline=(current  bounding  box.center)]
			
			\coordinate (outw) at (-0.04,3.8);
			\coordinate (oute) at (0.04,3.8);
			\coordinate (mid) at (0,3);
			
			\begin{knot}[clip width=4, clip radius = 3pt]
			\strand [thick, blue] (outw) to [out=-90,in=180] (mid) to [out=0,in=-90] (oute);
			\end{knot}
			\end{tikzpicture}
		}&:=&(v_0\colon  \mathbb{I}_c\xrightarrow{\oplus\frac{t_i}{D}\bar{\tn{coev}}_i}\bigoplus_{i\in\SA} i\otimes_c i^*)
		\\
		\hbox{
			\begin{tikzpicture}[baseline=(current  bounding  box.center)]
			
			\coordinate (inw) at (-0.04,2.2);
			\coordinate (ine) at (0.04,2.2);
			\coordinate (mid) at (0,3);
			
			\begin{knot}[clip width=4, clip radius = 3pt]
			\strand [thick, blue] (inw) to [out=90,in=180] (mid) to [out=0,in=90] (ine);
			\end{knot}
			\end{tikzpicture}
		}&:=&(u_0\colon  \mathbb{I}_s \xrightarrow{\oplus\tn{ev}} \mathbb{I}_c).
	\end{eqnarray*} 
	These morphism exhibit the unit $\mathbb{I}_c$ for $\otimes_c$ as a subobject of the unit $\mathbb{I}_s$ for $\otimes_s$ with inclusion $v_0$ and projection $u_0$.
\end{lem}
\begin{proof}
	The fact that these maps constitute a inclusion and projection pair is clear. We still need to show that these morphisms are morphisms in $\dcentcat{A}$, i.e. that they commute with the braiding. Recall that the braiding for $\mathbb{I}_s$ is defined in Equation \eqref{DCunithalfbraid}. We compute, using Lemma \ref{DCotherdirectsum} in the last step:
	$$
	\hbox{
		\begin{tikzpicture}[baseline=(current  bounding  box.center)]
		
		\coordinate (outw) at (-0.05,3.5);
		\coordinate (oute) at (0.05,3.5);
		\coordinate (olw) at (-0.05,2.5);
		\coordinate (ole) at (0.05,2.5);
		\coordinate (mid) at (0,2);
		
		\node (ain) at (-0.5,1.8){\strut$a$};
		\coordinate (lc) at (-0.5,2.5);
		\coordinate (rc) at (0.5, 3.5);
		\coordinate (aout) at (0.5,3.5);
		
		\begin{knot}[clip width=4, clip radius = 3pt]
		\strand [thick, blue] (outw) to [out=-90,in=90] (olw) to [out=-90,in=180] (mid) to [out=0,in=-90] (ole) to [out=90,in=-90] (oute);
		\strand [thick] (ain) to [out=90,in=-90] (lc) to [out=90,in=-90] (rc) to [out=90,in=-90] (aout);
		\end{knot}
		\end{tikzpicture}
	}
	=\sum\limits_{i,j\in\cat{O}(\cat{A})}\sum\limits_{\phi\in B(aj,i)}\frac{d_j}{D}
	\hbox{
		\begin{tikzpicture}[baseline=(current  bounding  box.center)]
		
		\node (outw) at (-0.5,4){\strut$i$};
		\node (oute) at (0.05,4){\strut$i^*$};
		\node (olw) at (-0.5,2.5)[draw]{\strut$\phi$};
		\node (ole) at (0.5,3)[draw]{\strut$\phi^t$};
		\coordinate (mid) at (0,2);
		
		\node (ain) at (-0.5,1.2){\strut$a$};
		\coordinate (lc) at (0,2.9);
		\coordinate (rc) at (1, 3.3);
		\node (aout) at (0.5,4){\strut$a$};
		
		\node (jl) at (0.7,2){\strut$j^*$};
		
		\begin{knot}[clip width=4, clip radius = 3pt]
		\strand [thick] (outw) to [out=-90,in=90] (olw);
		\strand [thick] (olw) to [out=-80,in=180] (mid) to [out=0,in=-90] (rc) to [out=90,in=80] (ole);
		\strand [thick] (ain) to [out=90,in=-100] (olw);
		\strand [thick] (ole) to [out=100,in=-90] (aout);
		\strand [thick] (oute) to [out=-90,in=90] (lc) to [out=-90,in=-90] (ole);
		\end{knot}
		\end{tikzpicture}
	}
	=
	\hbox{
		\begin{tikzpicture}[baseline=(current  bounding  box.center)]
		
		\coordinate (outw) at (-0.05,3.5);
		\coordinate (oute) at (0.05,3.5);
		\coordinate (olw) at (-0.05,3.2);
		\coordinate (ole) at (0.05,3.2);
		\coordinate (mid) at (0,2.7);
		
		\node (ain) at (-0.5,1.8){\strut$a$};
		\coordinate (lc) at (-0.5,2.3);
		\coordinate (rc) at (0.5, 2.8);
		\coordinate (aout) at (0.5,3.5);
		
		\begin{knot}[clip width=4, clip radius = 3pt]
		\strand [thick, blue] (outw) to [out=-90,in=90] (olw) to [out=-90,in=180] (mid) to [out=0,in=-90] (ole) to [out=90,in=-90] (oute);
		\strand [thick] (ain) to [out=90,in=-90] (lc) to [out=90,in=-90] (rc) to [out=90,in=-90] (aout);
		\end{knot}
		\end{tikzpicture}
	}.
	$$
	A similar, but simpler, argument shows that $u_0$ commutes with braiding.
\end{proof}

To produce the morphisms $u_1$ and $v_2$, we note that:
\begin{lem}
	The objects $\mathbb{I}_c\otimes_s\mathbb{I}_c$ and $\mathbb{I}_c$ are canonically isomorphic.
\end{lem}
\begin{proof}
We observe that on $\cat{A}\subset \dcentcat{A}$, the symmetric tensor product and the convolution tensor product agree. 
\end{proof}

We set $u_1\colon  \mathbb{I}_c\otimes_s\mathbb{I}_c \xleftrightarrow{\cong} \mathbb{I}_c\colon v_2$.

For $u_2$ and $v_1$, we use the following:

\begin{lem}\label{DCu2v1def}
	The object $\mathbb{I}_s$ is a subobject of $\mathbb{I}_s\otimes_c\mathbb{I}_s$, with inclusion and projection given by
	$$
	u_2= 
	\hbox{
		\begin{tikzpicture}[baseline=(current  bounding  box.center)]
		
		\coordinate (inw) at (-0.05,0);
		\coordinate (ine) at (0.05,0);
		
		\coordinate (outlw) at (-0.35,1);
		\coordinate (outle) at (-0.25,1);
		
		\coordinate (outrw) at (0.25,1);
		\coordinate (outre) at (0.35,1);	
		
		\coordinate (mid) at (0,0.6);

		\begin{knot}[clip width=4,clip radius =3pt]
		\strand [thick,blue] (inw) to [out=90,in=-90] (outlw);
		\strand [thick,blue] (ine) to [out=90,in=-90] (outre);
		\strand [thick,blue] (outle) to [out=-90,in=180] (mid) to [out=0,in=-90] (outrw); 
		\end{knot}
		\end{tikzpicture}
	}
	=\sum\limits_{i\in\cat{O}(\cat{A})}\frac{t_i}{D}
	\hbox{
		\begin{tikzpicture}[baseline=(current  bounding  box.center)]
		
		\node (inw) at (-0.1,0){\strut$i$};
		\node (ine) at (0.1,0){\strut$i^*$};
		
		\node (outlw) at (-0.4,1.5){\strut$i$};
		\node (outle) at (-0.2,1.5){\strut$i^*$};
		
		\node (outrw) at (0.2,1.5){\strut$i$};
		\node (outre) at (0.4,1.5){\strut$i^*$};	
		
		\coordinate (mid) at (0,0.6);

		\begin{knot}[clip width=4,clip radius =3pt]
		\strand [thick] (inw) to [out=90,in=-90] (outlw);
		\strand [thick] (ine) to [out=90,in=-90] (outre);
		\strand [thick] (outle) to [out=-90,in=180] (mid) to [out=0,in=-90] (outrw); 
		\end{knot}
		\end{tikzpicture}
	}
	\mbox{ and }
	v_1= \hbox{
		\begin{tikzpicture}[baseline=(current  bounding  box.center)]
		
		\coordinate (inw) at (-0.05,1);
		\coordinate (ine) at (0.05,1);
		
		\coordinate (outlw) at (-0.35,0);
		\coordinate (outle) at (-0.25,0);
		
		\coordinate (outrw) at (0.25,0);
		\coordinate (outre) at (0.35,0);	
		
		\coordinate (mid) at (0,0.4);

		\begin{knot}[clip width=4,clip radius =3pt]
		\strand [thick,blue] (inw) to [out=-90,in=90] (outlw);
		\strand [thick,blue] (ine) to [out=-90,in=90] (outre);
		\strand [thick,blue] (outle) to [out=90,in=180] (mid) to [out=0,in=90] (outrw);
		\end{knot}
		\end{tikzpicture}
	}
	=\sum\limits_{i,j\in\cat{O}(\cat{A})} \delta_{i,j^*}
	\hbox{
		\begin{tikzpicture}[baseline=(current  bounding  box.center)]
		
		\node (inw) at (-0.1,1){\strut$i$};
		\node (ine) at (0.1,1){\strut$i^*$};
		
		\node (outlw) at (-0.4,-0.5){\strut$i$};
		\node (outle) at (-0.2,-0.5){\strut$i^*$};
		
		\node (outrw) at (0.2,-0.5){\strut$j$};
		\node (outre) at (0.4,-0.5){\strut$j^*$};	
		
		\coordinate (mid) at (0,0.4);

		\begin{knot}[clip width=4,clip radius =3pt]
		\strand [thick] (inw) to [out=-90,in=90] (outlw);
		\strand [thick] (ine) to [out=-90,in=90] (outre);
		\strand [thick] (outle) to [out=90,in=180] (mid) to [out=0,in=90] (outrw);
		\end{knot}
		\end{tikzpicture}
	},
	$$
	repectively. $\delta_{i,j^*}$ denotes the Kronecker delta symbol that is 1 when $i=j^*$ and zero otherwise. In particular, $u_2$ and $v_1$ are morphisms in $\dcentcat{A}$.
\end{lem}

\begin{proof}
	It is clear that $u_2$ and $v_1$ constitute an inclusion-projection pair, composing along of $\mathbb{I}_s\otimes_c \mathbb{I}_s$ gives $\sum_{i} d_i \frac{d_i}{D}=1$ times the identity on $\mathbb{I}_s$. We still need to establish they are indeed morphisms in $\dcentcat{A}$. That is, we need to show that 
	$$
	\hbox{
		\begin{tikzpicture}[baseline=(current  bounding  box.center)]

		\coordinate (inw) at (-0.05,0);
		\coordinate (ine) at (0.05,0);
		
		\coordinate (outlw) at (-0.35,1);
		\coordinate (outle) at (-0.25,1);
		
		\coordinate (outrw) at (0.25,1);
		\coordinate (outre) at (0.35,1);	
		
		\coordinate (mid) at (0,0.6);	
		
		\coordinate (ulw) at (-0.35,2);
		\coordinate (ule) at (-0.25,2);
		
		\coordinate (urw) at (0.25,2);
		\coordinate (ure) at (0.35,2);	
		
		\coordinate (a) at (-0.5,0);
		\coordinate (lc) at (-0.5,0.6);
		\coordinate (rc) at (0.5,1.7);
		\coordinate (ao) at (0.5,2);
		
		\begin{knot}[clip width=4,clip radius =3pt]
		\strand [thick,blue] (inw) to [out=90,in=-90] (outlw) to (ulw);
		\strand [thick,blue] (ine) to [out=90,in=-90] (outre) to (ure);
		\strand [thick,blue] (ule) to (outle) to [out=-90,in=180] (mid);
		\strand[thick,blue] (mid) to [out=0,in=-90] (outrw) to (urw);
		\strand [thick] (a) to[out=90,in=-90] (lc) to [out=90,in=-90] (rc) to [out=90,in=-90] (ao);
		\end{knot}
		\end{tikzpicture}
	}
	=
	\hbox{
		\begin{tikzpicture}[baseline=(current  bounding  box.center)]
		
		\coordinate (lw) at (-0.05,-1);
		\coordinate (le) at (0.05,-1);
		
		\coordinate (inw) at (-0.05,0);
		\coordinate (ine) at (0.05,0);
		
		\coordinate (outlw) at (-0.35,1);
		\coordinate (outle) at (-0.25,1);
		
		\coordinate (outrw) at (0.25,1);
		\coordinate (outre) at (0.35,1);	
		
		\coordinate (mid) at (0,0.6);

		\coordinate (a) at (-0.5,-1);
		\coordinate (rc) at (0.5,0.3);
		\coordinate (ao) at (0.5,1);
		
		\begin{knot}[clip width=4,clip radius =3pt]
		\strand [thick,blue] (lw) to (inw) to [out=90,in=-90] (outlw);
		\strand [thick,blue] (le) to (ine) to [out=90,in=-90] (outre);
		\strand [thick,blue] (outle) to [out=-90,in=180] (mid) to [out=0,in=-90] (outrw);
		\strand [thick] (a) to[out=90,in=-90] (rc) to [out=90,in=-90] (ao);
		\end{knot}
		\end{tikzpicture}
	}
	\tn{ and }
	\hbox{
		\begin{tikzpicture}[baseline=(current  bounding  box.center)]
		
		\coordinate (inw) at (-0.05,1);
		\coordinate (ine) at (0.05,1);
		
		\coordinate (outlw) at (-0.35,0);
		\coordinate (outle) at (-0.25,0);
		
		\coordinate (outrw) at (0.25,0);
		\coordinate (outre) at (0.35,0);	
		
		\coordinate (mid) at (0,0.4);	
		
		\coordinate (ulw) at (-0.35,-1);
		\coordinate (ule) at (-0.25,-1);
		
		\coordinate (urw) at (0.25,-1);
		\coordinate (ure) at (0.35,-1);	
		
		\coordinate (a) at (-0.5,-1);
		\coordinate (lc) at (-0.5,-0.6);
		\coordinate (rc) at (0.5,0.2);
		\coordinate (ao) at (0.5,1);
		
		\begin{knot}[clip width=4,clip radius =3pt]
		\strand [thick,blue] (inw) to [out=-90,in=90] (outlw) to (ulw);
		\strand [thick,blue] (ine) to [out=-90,in=90] (outre) to (ure);
		\strand [thick,blue] (ule) to (outle) to [out=90,in=180] (mid);
		\strand [thick,blue] (mid) to [out=0,in=90] (outrw) to (urw);
		\strand [thick] (a) to [out=90,in=-90] (lc) to [out=90,in=-90] (rc) to [out=90,in=-90] (ao);
		\end{knot}
		\end{tikzpicture}
	}
	=
	\hbox{
		\begin{tikzpicture}[baseline=(current  bounding  box.center)]
		
		\coordinate (iw) at (-0.05,2);
		\coordinate (ie) at (0.05,2);

		\coordinate (inw) at (-0.05,1);
		\coordinate (ine) at (0.05,1);
		
		\coordinate (outlw) at (-0.35,0);
		\coordinate (outle) at (-0.25,0);
		
		\coordinate (outrw) at (0.25,0);
		\coordinate (outre) at (0.35,0);	
		
		\coordinate (mid) at (0,0.4);

		\coordinate (a) at (-0.5,0);
		\coordinate (lc) at (-0.5,0.7);
		\coordinate (rc) at (0.5,1.5);
		\coordinate (ao) at (0.5,2);
		
		\begin{knot}[clip width=4,clip radius =3pt]
		\strand [thick,blue] (iw) to (inw) to [out=-90,in=90] (outlw);
		\strand [thick,blue] (ie) to (ine) to [out=-90,in=90] (outre);
		\strand [thick,blue] (outle) to [out=90,in=180] (mid) to [out=0,in=90] (outrw);
		\strand [thick] (a) to [out=90,in=-90] (lc) to [out=90,in=-90] (rc) to [out=90,in=-90] (ao);
		\end{knot}
		\end{tikzpicture}
	}.
	$$
	Unpacking the definition (Equation \eqref{DCunithalfbraid}) of the half-braiding for $\mathbb{I}_s$, we see that we get for $a\in\cat{A}$:
	\begin{equation}\label{DCunpack}
	\hbox{
		\begin{tikzpicture}[baseline=(current  bounding  box.center)]

		\coordinate (inw) at (-0.05,0);
		\coordinate (ine) at (0.05,0);
		
		\coordinate (outlw) at (-0.35,1);
		\coordinate (outle) at (-0.25,1);
		
		\coordinate (outrw) at (0.25,1);
		\coordinate (outre) at (0.35,1);	
		
		\coordinate (mid) at (0,0.6);	
		
		\coordinate (ulw) at (-0.35,2);
		\coordinate (ule) at (-0.25,2);
		
		\coordinate (urw) at (0.25,2);
		\coordinate (ure) at (0.35,2);	
		
		\coordinate (a) at (-0.5,0);
		\coordinate (lc) at (-0.5,0.6);
		\coordinate (rc) at (0.5,1.7);
		\coordinate (ao) at (0.5,2);
		
		\begin{knot}[clip width=4,clip radius =3pt]
		\strand [thick,blue] (inw) to [out=90,in=-90] (outlw) to (ulw);
		\strand [thick,blue] (ine) to [out=90,in=-90] (outre) to (ure);
		\strand [thick,blue] (ule) to (outle) to [out=-90,in=180] (mid);
		\strand[thick,blue] (mid) to [out=0,in=-90] (outrw) to (urw);
		\strand [thick] (a) to[out=90,in=-90] (lc) to [out=90,in=-90] (rc) to [out=90,in=-90] (ao);
		\end{knot}
		\end{tikzpicture}
	}
	=
	\sum\limits_{j,k\in\cat{O}(\cat{A})}\sum\limits_{\phi,\phi'\in B(ak,j)}\frac{t_k}{D}
	\hbox{
		\begin{tikzpicture}[baseline=(current  bounding  box.center)]
		
		\node (inw) at (-0.5,-1){\strut$k$};
		\node (ine) at (0.05,-1){\strut$k^*$};
		
		\node (outw) at (-1,3){\strut$j$};
		\node (oute) at (1,3){\strut$j^*$};
		
		\node (iw) at (-0.5,3){\strut$j^*$};
		\node (ie) at (0.5,3){\strut$j$};
		
		\node (phid) at (-0.5,1)[draw, minimum width=10pt]{\strut$\phi^*$};
		\node (phi) at (0.5,1.5)[draw,minimum width=10pt]{\strut$\phi'$};
		
		\coordinate (lc) at (-1.5,0.7);
		\coordinate (rc) at (-0.3,2.2);
		
		\coordinate (mid) at (-0.2,0.2);
		
		\node (ain) at (-1.5,-1){\strut$a$};
		\node (aout) at (1.5,3){\strut$a$};
		
		\node (phin) at (-1,0.3)[draw, minimum width=10pt]{\strut$\phi$};
		\node (phio) at (1,2.2)[draw, minimum width=10pt]{\strut$\phi^{\prime*}$};
		
		\node (al) at (0,1){\strut$a$};
		\node (kl) at (-0.5,0){\strut$k$};
		
		\coordinate (rcc) at (1.5,1.8);
		
		\begin{knot}[clip width=4,clip radius =3pt]
		\strand [thick] (iw) to [out=-90,in=90] (phid);
		\strand [thick] (ie) to [out=-90,in=90] (phi);
		\strand [thick] (phid) to [out=-70,in=-110] (phi);
		\strand [thick] (phid) to [out=-110,in=180] (mid) to [out=0,in=-80] (phi);
		\strand [thick] (ine) to [out=80,in=-90] (phio);
		\strand [thick] (phio) to [out=-80,in=-90] (rcc) to [out=90,in=-90] (aout);
		\strand [thick] (phio) to [out=90,in=-90] (oute);
		\strand [thick] (inw) to [out=90,in=-80] (phin);
		\strand [thick] (ain) to [out=90,in=-100] (phin);
		\strand [thick] (phin) to [out=90,in=-90] (outw);
		\end{knot}
		\end{tikzpicture}
	}.
	\end{equation}
	We can manipulate the middle part of the summands to see:
	\begin{equation*}
	\hbox{
		\begin{tikzpicture}[baseline=(current  bounding  box.center)]
		
		\node (iw) at (-0.5,2.5){\strut$j^*$};
		\node (ie) at (0.5,2.5){\strut$j$};
		
		\node (phid) at (-0.5,1)[draw, minimum width=10pt]{\strut$\phi^*$};
		\node (phi) at (0.5,1.5)[draw,minimum width=10pt]{\strut$\phi'$};
		
		\coordinate (lc) at (-1.5,0.7);
		\coordinate (rc) at (-0.3,2.2);
		
		\coordinate (mid) at (0,0);
		
		\node (al) at (0,1){\strut$a$};
		\node (kl) at (0,-0.2){\strut$k$};
		
		\begin{knot}[clip width=4,clip radius =3pt]
		\strand [thick] (iw) to [out=-90,in=90] (phid);
		\strand [thick] (ie) to [out=-90,in=90] (phi);
		\strand [thick] (phid) to [out=-70,in=-110] (phi);
		\strand [thick] (phid) to [out=-110,in=180] (mid) to [out=0,in=-70] (phi);
		\end{knot}
		\end{tikzpicture}
	}
	=
	\hbox{
		\begin{tikzpicture}[baseline=(current  bounding  box.center)]
		
		\node (iw) at (-0.5,2.8){\strut$j^*$};
		\node (ie) at (0.5,2.8){\strut$j$};
		
		\node (phid) at (-0.5,0.8)[draw, minimum width=10pt]{\strut$\phi^t$};
		\node (phi) at (0.5,1.8)[draw,minimum width=10pt]{\strut$\phi'$};
		
		\coordinate (lc) at (-1,0.5);
		\coordinate (rc) at (-0.3,2.2);
		
		\coordinate (mid) at (0.5,0.7);
		
		\begin{knot}[clip width=4,clip radius =3pt]
		\strand [thick] (iw) to [out=-90,in=90] (lc) to [out=-90,in=-90] (phid);
		\strand [thick] (ie) to [out=-90,in=90] (phi);
		\strand [thick] (phid) to [out=110,in=-110] (phi);
		\strand [thick] (phid) to [out=70,in=180] (mid) to [out=0,in=-70] (phi);
		\end{knot}
		\end{tikzpicture}
	}
	=
	\delta_{\phi,\phi'}
	\hbox{
		\begin{tikzpicture}[baseline=(current  bounding  box.center)]
		
		\node (iw) at (-0.5,2.8){\strut$j^*$};
		\node (ie) at (0.5,2.8){\strut$j$};
		
		\coordinate (mid) at (0,1.7);
		
		\begin{knot}[clip width=4,clip radius =3pt]
		\strand [thick] (iw) to [out=-90,in=180] (mid) to [out=0,in=-90] (ie);
		\end{knot}
		\end{tikzpicture}
	},
	\end{equation*}
	where $\phi,\phi'\in B(ak,j)$. Plugging this into Equation \eqref{DCunpack}, we get:
	\begin{equation*}
	\sum\limits_{j,k\in\cat{O}(\cat{A})}\sum\limits_{\phi,\in B(ak,j)}\frac{t_k}{D}
	\hbox{
		\begin{tikzpicture}[baseline=(current  bounding  box.center)]
		
		\node (inw) at (-0.5,0){\strut$k$};
		\node (ine) at (0.05,0){\strut$k^*$};
		
		\node (outw) at (-1,3){\strut$j$};
		\node (oute) at (1,3){\strut$j^*$};
		
		\node (iw) at (-0.5,3){\strut$j^*$};
		\node (ie) at (0.5,3){\strut$j$};

		\coordinate (lc) at (-1.5,0.7);
		\coordinate (rc) at (-0.3,2.2);
		
		\coordinate (mid) at (-0.2,2.2);
		
		\node (ain) at (-1.5,0){\strut$a$};
		\node (aout) at (1.5,3){\strut$a$};
		
		\node (phin) at (-0.5,1.5)[draw, minimum width=10pt]{\strut$\phi$};
		\node (phio) at (0.5,1.5)[draw, minimum width=10pt]{\strut$\phi^{*}$};

		\coordinate (rcc) at (1,1);
		
		\begin{knot}[clip width=4,clip radius =3pt]
		\strand [thick] (iw) to [out=-90,in=180] (mid) to [out=0,in=-90] (ie);
		\strand [thick] (ine) to [out=90,in=-100] (phio);
		\strand [thick] (phio) to [out=-80,in=-90] (rcc) to [out=90,in=-90] (aout);
		\strand [thick] (phio) to [out=90,in=-90] (oute);
		\strand [thick] (inw) to [out=90,in=-80] (phin);
		\strand [thick] (ain) to [out=90,in=-100] (phin);
		\strand [thick] (phin) to [out=90,in=-90] (outw);
		\end{knot}
		\end{tikzpicture}
	}
	=
	\hbox{
		\begin{tikzpicture}[baseline=(current  bounding  box.center)]
		
		\coordinate (lw) at (-0.05,-1);
		\coordinate (le) at (0.05,-1);
		
		\coordinate (inw) at (-0.05,0);
		\coordinate (ine) at (0.05,0);
		
		\coordinate (outlw) at (-0.35,1);
		\coordinate (outle) at (-0.25,1);
		
		\coordinate (outrw) at (0.25,1);
		\coordinate (outre) at (0.35,1);	
		
		\coordinate (mid) at (0,0.6);

		\coordinate (a) at (-0.5,-1);
		\coordinate (rc) at (0.5,0.3);
		\coordinate (ao) at (0.5,1);
		
		\begin{knot}[clip width=4,clip radius =3pt]
		\strand [thick,blue] (lw) to (inw) to [out=90,in=-90] (outlw);
		\strand [thick,blue] (le) to (ine) to [out=90,in=-90] (outre);
		\strand [thick,blue] (outle) to [out=-90,in=180] (mid) to [out=0,in=-90] (outrw);
		\strand [thick] (a) to[out=90,in=-90] (rc) to [out=90,in=-90] (ao);
		\end{knot}
		\end{tikzpicture}
	}.
	\end{equation*}
	Similarly, we have:
	$$
	\hbox{
		\begin{tikzpicture}[baseline=(current  bounding  box.center)]

		\coordinate (inw) at (-0.05,2);
		\coordinate (ine) at (0.05,2);
		
		\coordinate (outlw) at (-0.35,1);
		\coordinate (outle) at (-0.25,1);
		
		\coordinate (outrw) at (0.25,1);
		\coordinate (outre) at (0.35,1);	
		
		\coordinate (mid) at (0,1.4);	
		
		\coordinate (ulw) at (-0.35,0);
		\coordinate (ule) at (-0.25,0);
		
		\coordinate (urw) at (0.25,0);
		\coordinate (ure) at (0.35,0);	
		
		\coordinate (a) at (-0.5,0);
		\coordinate (lc) at (-0.5,0.5);
		\coordinate (rc) at (0.5,1.5);
		\coordinate (ao) at (0.5,2);

		\begin{knot}[clip width=4,clip radius =3pt]
		\strand [thick,blue] (inw) to [out=-90,in=90] (outlw) to (ulw);
		\strand [thick,blue] (ine) to [out=-90,in=90] (outre) to (ure);
		\strand [thick,blue] (ule) to (outle) to [out=90,in=180] (mid);
		\strand[thick,blue] (mid) to [out=0,in=90] (outrw) to (urw);
		\strand [thick] (a) to[out=90,in=-90] (lc) to [out=90,in=-90] (rc) to [out=90,in=-90] (ao);
		\end{knot}
		\end{tikzpicture}
	}
	=
	\sum\limits_{j,k\in\cat{O}(\cat{A})}\sum\limits_{\phi,\phi'\in B(ak,j)}
	\hbox{
		\begin{tikzpicture}[baseline=(current  bounding  box.center)]
		
		\node (inw) at (-0.5,3){\strut$k$};
		\node (ine) at (0.05,3){\strut$k^*$};
		
		\node (outw) at (-1,-1){\strut$j$};
		\node (oute) at (1,-1){\strut$j^*$};
		
		\node (iw) at (-0.5,-1){\strut$j^*$};
		\node (ie) at (0.5,-1){\strut$j$};
		
		\node (phid) at (-0.5,1)[draw, minimum width=10pt]{\strut$\phi^*$};
		\node (phi) at (0.5,0.5)[draw,minimum width=10pt]{\strut$\phi'$};
		
		\coordinate (lc) at (-1.5,1.3);
		\coordinate (rc) at (-0.3,-0.2);
		
		\coordinate (mid) at (-0.2,1.8);
		
		\node (ain) at (-1.5,-1){\strut$a$};
		\node (aout) at (1.5,3){\strut$a$};
		
		\node (phin) at (-1,1.7)[draw, minimum width=10pt]{\strut$\phi$};
		\node (phio) at (1,-0.2)[draw, minimum width=10pt]{\strut$\phi^{\prime*}$};
		
		\node (al) at (0.15,1){\strut$a$};
		\node (kl) at (-0.4,2){\strut$k$};
		
		\coordinate (amidup) at (0,1.5);
		\coordinate (amidlow) at (0,0);

		\begin{knot}[clip width=4,clip radius =3pt]
		\strand [thick] (iw) to [out=90,in=-90] (phid);
		\strand [thick] (ie) to [out=90,in=-90] (phi);
		\strand [thick] (phid) to [out=70,in=90] (amidup) to [out=-90,in=90](amidlow) to [out=-90, in=-110] (phi);
		\strand [thick] (phid) to [out=110,in=180] (mid) to [out=0,in=80] (phi);
		\strand [thick] (ine) to [out=-80,in=90] (phio);
		\strand [thick] (phio) to [out=60,in=-90] (aout);
		\strand [thick] (phio) to [out=-90,in=90] (oute);
		\strand [thick] (inw) to [out=-90,in=80] (phin);
		\strand [thick] (ain) to [out=90,in=-100] (phin);
		\strand [thick] (phin) to [out=-90,in=90] (outw);
		\end{knot}
		\end{tikzpicture}
	}.
	$$
	We can again examine the middle part of this diagram to see:
	$$
	\hbox{
		\begin{tikzpicture}[baseline=(current  bounding  box.center)]
		
		\node (iw) at (-0.5,0){\strut$j^*$};
		\node (ie) at (0.5,0){\strut$j$};
		
		\node (phid) at (-0.5,1)[draw, minimum width=10pt]{\strut$\phi^*$};
		\node (phi) at (0.5,1.5)[draw,minimum width=10pt]{\strut$\phi'$};
		
		\coordinate (mid) at (0,2);
		
		\begin{knot}[clip width=4,clip radius =3pt]
		\strand [thick] (iw) to (iw) to [out=90,in=-90] (phid);
		\strand [thick] (ie) to [out=90,in=-90] (phi);
		\strand [thick] (phid) to [out=110,in=180] (mid) to [out=0,in=90] (phi);
		\strand [thick] (phid) to [out=70,in=-110] (phi);
		\end{knot}
		\end{tikzpicture}
	}
	=
	\hbox{
		\begin{tikzpicture}[baseline=(current  bounding  box.center)]
		
		\node (iw) at (-0.5,0){\strut$j^*$};
		\node (ie) at (0.5,0){\strut$j$};
		
		\node (phid) at (-0.5,1)[draw, minimum width=10pt]{\strut$\phi^t$};
		\node (phi) at (0.5,2)[draw,minimum width=10pt]{\strut$\phi'$};
		
		\coordinate (lc) at (-1,0.9);
		\coordinate (rc) at (0.1,1.1);
		
		\coordinate (mid) at (0,2.5);
		
		\begin{knot}[clip width=4,clip radius =3pt]
		\strand [thick] (iw) to [out=90,in=-90] (rc) to [out=90,in=70] (phid);
		\strand [thick] (ie) to [out=90,in=-90] (phi);
		\strand [thick] (phid) to [out=-90,in=-90] (lc) to [out=90,in=180] (mid) to [out=0,in=90] (phi);
		\strand [thick] (phid) to [out=110,in=-110] (phi);
		\end{knot}
		\end{tikzpicture}
	}=
	\hbox{
		\begin{tikzpicture}[baseline=(current  bounding  box.center)]
		
		\node (iw) at (-0.3,-0.5){\strut$j^*$};
		\node (ie) at (0.5,-0.5){\strut$j$};
		
		\node (phid) at (-1,2)[draw, minimum width=10pt]{\strut$\phi^t$};
		\node (phi) at (-1,1)[draw,minimum width=10pt]{\strut$\phi'$};
		
		\coordinate (lc) at (-1.5,0.7);
		\coordinate (rc) at (-0.3,2.2);
		
		\coordinate (twistmid) at (-1,1.5);
		
		\coordinate (mid) at (-1.5,2.3);
		
		\begin{knot}[clip width=4,clip radius =3pt]
		\strand [thick] (iw) to [out=90,in=-90] (rc) to [out=90,in=70] (phid);
		\strand [thick] (ie) to [out=90,in=-60] (phi);
		\strand [thick] (phid) to [out=-90,in=90] (phi);
		\strand [thick] (phid) to [out=110,in=90] (mid) to [out=-90,in=90] (lc) to [out=-90,in=-110] (phi);
		\end{knot}
		\fill (twistmid) circle[radius=2pt];
		\fill (lc) circle[radius=2pt];
		\end{tikzpicture},
	}
	$$
	where in the last equality we have moved $\phi^t$ up and $\phi'$ along the back. We can now use Equation \eqref{DCtwisttrick} to move the twists to the $j$ strand. Then, after pre-composing with the coevaluation $\mathbb{I}_c\rar j j^*$, we can view the last morphism as an endomorphism of the simple object $j$, so it is completely determined by its trace. This trace is computed by pre-composing the morphisms in the last diagram the morphism $\mathbb{I}\rar j^*j$ that is obtained by composing the coevaluation for $j^*$ and $j^{**}$ with the pivotal structure. After we cancel the self-intersection with the twist in the diagram thus obtained, we get that this computes as $\delta_{\phi,\phi'}$. This means that this morphism evaluates to $\delta_{\phi,\phi'}$ times the evaluation on $j^*j$. Plugging this in yields the desired relation.	
\end{proof}

\subsection{Coherence}
This section is devoted to proving that the morphisms defined in the previous section satisfy the coherence conditions from Definition \ref{DClax2fold}. This will establish Theorem \ref{DCdc2foldmon}.

\subsubsection{Unitor Coherence}
\begin{lem}
	The morphisms $u_0$ and $v_0$ satisfy the coherence diagrams from Definition \ref{DClax2fold}\eqref{DCunitor1}, where $1=c$, $2=s$ and $1=s$, $2=c$ respectively.
\end{lem}

\begin{proof}
	For $u_0$, both routes through the diagram in \ref{DClax2fold}\eqref{DCunitor1} evaluate to $u_0$ directly, so there is nothing to prove. 
	
	For $v_0$, it is more convenient to compare $\mathbb{I}_c\xrightarrow{\lambda_s^{-1}}\mathbb{I}_c\otimes_s\mathbb{I}_s\xrightarrow{u_0\otimes_s \tn{id}}\mathbb{I}_s\otimes_s\mathbb{I}_s\xrightarrow{\lambda_s}\mathbb{I}_s$ to $v_0$. To do this, observe that, in string diagrams, this composite computes as:
	$$
	
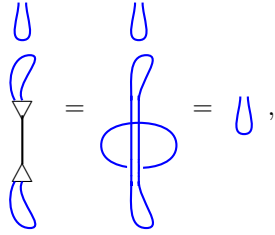
\hbox{
		\begin{tikzpicture}[baseline=(current  bounding  box.center)]
		
		\node (tr1n) at (0,1.2) {\strut$\bigtriangleup$};
		\coordinate (trd) at (0,1.35);
		\coordinate (truw) at (-0.1,1.12);
		\coordinate (true) at (0.1,1.12);
		\coordinate (trum) at (0,1.12);
		
		\node (tr1) at (0,2.1) {\strut$\bigtriangledown$};		
		\coordinate (tr1d) at (0,2);
		\coordinate (tr1uw) at (-0.1,2.2);
		\coordinate (tr1ue) at (0.1,2.2);
		\coordinate (tr1um) at (0,2.2);
		
		\coordinate (ctrl1) at (0.1,2.8);		
		
		\coordinate (ctrl) at (0.1,0.5);
		
		\coordinate (outw) at (-0.05,3.5);
		\coordinate (oute) at (0.05,3.5);
		\coordinate (mid) at (0,3);
		
		\begin{knot}[clip width=4, clip radius = 3pt]
		\strand [thick] (tr1d) to (trd);
		\strand [thick,blue] (trum) to [out=-120, in=0] (ctrl) to [out=180,in=-120] (truw);
		\strand [thick,blue] (tr1um) to [out=120, in=0] (ctrl1) to [out=180,in=120] (tr1uw);
		\strand [thick, blue] (outw) to [out=-90,in=180] (mid) to [out=0,in=-90] (oute);
		\end{knot}
		\end{tikzpicture}
	}=
	\hbox{
		\begin{tikzpicture}[baseline=(current  bounding  box.center)]
		
		\coordinate (west) at (-0.5,1.6);
		\coordinate (north) at (0,1.9);
		\coordinate (east) at (0.5,1.6);
		\coordinate (south) at (0,1.3);

		\coordinate (truw) at (-0.1,1.12);
		\coordinate (trum) at (0,1.12);
		
		\coordinate (tr1uw) at (-0.1,2.2);
		\coordinate (tr1um) at (0,2.2);
		
		\coordinate (ctrl1) at (0.1,2.8);		
		
		\coordinate (ctrl) at (0.1,0.5);
		
		\coordinate (outw) at (-0.05,3.5);
		\coordinate (oute) at (0.05,3.5);
		\coordinate (mid) at (0,3);
		
		\begin{knot}[clip width=4, clip radius = 3pt]
		\strand [thick,blue] (tr1um) to [out=90,in=-90] (trum) to [out=-90, in=0] (ctrl) to [out=180,in=-90] (truw) to [out=-90,in=90] (tr1uw);
		\strand [thick,blue] (tr1um) to [out=90, in=0] (ctrl1) to [out=180,in=90] (tr1uw);
		\strand [thick, blue] (outw) to [out=-90,in=180] (mid) to [out=0,in=-90] (oute);
		\strand [blue, thick] (west)
		to [out=90,in=-180] (north)
		to [out=0,in=90] (east)
		to [out=-90,in=0] (south)
		to [out=-180,in=-90] (west);
		\flipcrossings{}
		\end{knot}
		\end{tikzpicture}
	}=
	\hbox{
		\begin{tikzpicture}[baseline=(current  bounding  box.center)]
		
		\coordinate (outw) at (-0.05,3.5);
		\coordinate (oute) at (0.05,3.5);
		\coordinate (mid) at (0,3);
		
		\begin{knot}[clip width=4, clip radius = 3pt]
		\strand [thick, blue] (outw) to [out=-90,in=180] (mid) to [out=0,in=-90] (oute);
		\end{knot}
		\end{tikzpicture}
	},
	$$
	where we have applied Equation \eqref{DCinclprojprops} followed by snapping (Lemma \ref{DCsnapping}) and the fact that a loose ring evaluates to 1.
\end{proof}

\begin{lem}
	The morphisms $\eta,u_2$ make the diagrams from Definition \ref{DClax2fold}\eqref{DCunitor4} commute for $\otimes_1=\otimes_c$ and $\otimes_2=\otimes_c$. Analogously, the morphisms $\zeta$ and $v_1$ make the diagrams from Definitions \ref{DClax2fold}\eqref{DCunitor3} commute for $\otimes_1=\otimes_s$ and $\otimes_2=\otimes_c$.	
\end{lem}

\begin{proof}
	Consider the anti-clockwise composite in the diagram from Definition \ref{DClax2fold}\eqref{DCunitor4}:
	\begin{align*}
	d \otimes_c d' \xrightarrow{\lambda_s^{-1}}& \mathbb{I}_s \otimes_s (d \otimes_c d') \xrightarrow{u_2\otimes_s \tn{id}} (\mathbb{I}_s\otimes_c\mathbb{I}_s)\otimes_s(d\otimes_c d)\\
	& \xrightarrow{\eta_{\mathbb{I}_s,\mathbb{I}_s,d,d'}} (\mathbb{I}_s\otimes_sd)\otimes_c (\mathbb{I}_s \otimes_s d') \xrightarrow{\lambda_s\otimes_s\lambda_s}d\otimes_c d'.
	\end{align*}
	In terms of string diagrams, replacing inclusions followed by projections by the idempotent using Equation \eqref{DCinclprojprops} right away, this becomes:
	$$
	\hbox{
		\begin{tikzpicture}[baseline=(current  bounding  box.center)]
		
		\node (din) at (0.5, -0.6){\strut$d$};
		\node (dpin) at (1,-0.6){\strut$d'$};
		
		\coordinate (ctrl1) at (1.2,0.5);
		
		\coordinate (iw1) at (-0.05,1);
		\coordinate (ie1) at (0.05,1);
		
		\coordinate (iw2) at (-0.05,1.6);
		\coordinate (ie2) at (0.05,1.6);
		
		\coordinate (mid) at (0,2.1);
		
		\coordinate (ilw1) at (-0.35,2.6);
		\coordinate (ile1) at (-0.25,2.6);
		\coordinate (irw1) at (0.25,2.6);
		\coordinate (ire1) at (0.35,2.6);
		
		\coordinate (ilw1a) at (-0.35,3);
		\coordinate (ile1a) at (-0.25,3);
		\coordinate (irw1a) at (0.25,3);
		\coordinate (ire1a) at (0.35,3);
		\coordinate (dl) at (0.5,3.25);

		\coordinate (du) at (0,3.8);
		\coordinate (irw2) at (0.65,3.8);
		\coordinate (ire2) at (0.75,3.8);
		
		\coordinate (ilw3) at (-0.35,5);
		\coordinate (ile3) at (-0.25,5);
		\coordinate (irw3) at (0.65,5);
		\coordinate (ire3) at (0.75,5);
		
		\coordinate (ucw) at (0.2,5.5);
		\coordinate (uce) at (1.2,5.5);
		
		\coordinate (dout) at (0,6.3);
		\coordinate (dpout) at (1,6.3);
		
		\coordinate (west1) at (-0.5,1.3);
		\coordinate (north1) at (0.5,1.6);
		\coordinate (east1) at (1.5,1.3);
		\coordinate (south1) at (0.5,1.0);
		
		\coordinate (west2) at (-0.6,2.6);
		\coordinate (north2) at (0.5,2.75);
		\coordinate (east2) at (1.3,2.6);
		\coordinate (south2) at (0.5,2.4);
		
		\coordinate (west3) at (-0.6,4.6);
		\coordinate (north3) at (-0.2,4.9);
		\coordinate (east3) at (0.2,4.6);
		\coordinate (south3) at (-0.2,4.3);
		
		\coordinate (west4) at (0.4,4.6);
		\coordinate (north4) at (1,4.9);
		\coordinate (east4) at (1.3,4.6);
		\coordinate (south4) at (1,4.3);

		\begin{knot}[clip width=4,clip radius =3.5pt]
		\strand [thick] (dpin) to [out=90,in=-90] (dpout);
		\strand [thick] (din) to [out=90,in=-90] (dl);
		\strand [thick] (dl) to [out=90,in=-90] (du) to [out=90,in=-90] (dout);
		\strand [thick,blue] (mid) to [out=0, in=-90] (irw1) to [out=90,in=-90] (irw1a) to [out=90,in=-90] (irw2) to [out=90,in=-90] (irw3) to [out=90,in=90] (uce) to [out=-90,in=90] (ire3) to [out=-90,in=90] (ire2) to [out=-90,in=90] (ire1a) to [out=-90,in=90] (ire1) to [out=-90,in=90] (ie2) to [out=-90,in=90] (ie1) to [out=-90,in=90] (ctrl1) to [out=-90,in=-90] (iw1) to [out=90,in=-90] (iw2) to [out=90,in=-90] (ilw1) to [out=90,in=-90] (ilw1a) to [out=90,in=-90] (ilw3) to [out=90,in=90] (ucw) to [out=-90,in=90] (ile3) to [out=-90,in=90] (ile1a) to [out=-90,in=90] (ile1) to [out=-90,in=180] (mid);
		\strand [blue, thick] (west1) to [out=90,in=-180] (north1) to [out=0,in=90] (east1) to [out=-90,in=0] (south1) to [out=-180,in=-90] (west1);
		\strand [blue, thick] (west2) to [out=90,in=-180] (north2) to [out=0,in=90] (east2) to [out=-90,in=0] (south2) to [out=-180,in=-90] (west2);
		\strand [blue, thick] (west3) to [out=90,in=-180] (north3) to [out=0,in=90] (east3) to [out=-90,in=0] (south3) to [out=-180,in=-90] (west3);
		\strand [blue, thick] (west4) to [out=90,in=-180] (north4) to [out=0,in=90] (east4) to [out=-90,in=0] (south4) to [out=-180,in=-90] (west4);
		\flipcrossings{2,19,32,34,9,38,36,29,25,28,7,15,21,5,13,11,3,18,22,23}
		\end{knot}
		\end{tikzpicture}
	}
	=
	\hbox{
		\begin{tikzpicture}[baseline=(current  bounding  box.center)]
		
		\node (din) at (0.5, -0.6){\strut$d$};
		\node (dpin) at (1,-0.6){\strut$d'$};
		
		\coordinate (ctrl1) at (1.2,0.5);
		
		\coordinate (iw1) at (-0.05,1);
		\coordinate (ie1) at (0.05,1);
		
		\coordinate (iw2) at (-0.05,1.6);
		\coordinate (ie2) at (0.05,1.6);
		
		\coordinate (mid) at (0,2.1);
		
		\coordinate (ilw1) at (-0.35,2.6);
		\coordinate (ile1) at (-0.25,2.6);
		\coordinate (irw1) at (0.25,2.6);
		\coordinate (ire1) at (0.35,2.6);
		
		\coordinate (ilw1a) at (-0.35,3);
		\coordinate (ile1a) at (-0.25,3);
		\coordinate (irw1a) at (0.25,3);
		\coordinate (ire1a) at (0.35,3);
		\coordinate (dl) at (0.5,3.25);

		\coordinate (du) at (0,3.8);
		\coordinate (irw2) at (0.65,3.8);
		\coordinate (ire2) at (0.75,3.8);
		
		\coordinate (ilw3) at (-0.35,5);
		\coordinate (ile3) at (-0.25,5);
		\coordinate (irw3) at (0.65,5);
		\coordinate (ire3) at (0.75,5);
		
		\coordinate (ucw) at (0.2,5.5);
		\coordinate (uce) at (1.2,5.5);
		
		\coordinate (dout) at (0,6.3);
		\coordinate (dpout) at (1,6.3);

		\begin{knot}[clip width=4,clip radius =3pt]
		\strand [thick] (dpin) to [out=90,in=-90] (dpout);
		\strand [thick] (din) to [out=90,in=-90] (dl);
		\strand [thick] (dl) to [out=90,in=-90] (du) to [out=90,in=-90] (dout);
		\strand [thick,blue] (mid) to [out=0, in=-90] (irw1) to [out=90,in=-90] (irw1a) to [out=90,in=-90] (irw2) to [out=90,in=-90] (irw3) to [out=90,in=90] (uce) to [out=-90,in=90] (ire3) to [out=-90,in=90] (ire2) to [out=-90,in=90] (ire1a) to [out=-90,in=90] (ire1) to [out=-90,in=90] (ie2) to [out=-90,in=90] (ie1) to [out=-90,in=90] (ctrl1) to [out=-90,in=-90] (iw1) to [out=90,in=-90] (iw2) to [out=90,in=-90] (ilw1) to [out=90,in=-90] (ilw1a) to [out=90,in=-90] (ilw3) to [out=90,in=90] (ucw) to [out=-90,in=90] (ile3) to [out=-90,in=90] (ile1a) to [out=-90,in=90] (ile1) to [out=-90,in=180] (mid);
		\flipcrossings{8,2,5,3}
		\end{knot}
		\end{tikzpicture}
	}
	=
	\hbox{
		\begin{tikzpicture}[baseline=(current  bounding  box.center)]
		
		\node (din) at (0.5, -0.5){\strut$d$};
		\node (dpin) at (1,-0.5){\strut$d'$};
		
		\coordinate (dout) at (0.5,6.3);
		\coordinate (dpout) at (1,6.3);

		\begin{knot}[clip width=4,clip radius =3pt]
		\strand [thick] (dpin) to [out=90,in=-90] (dpout);
		\strand [thick] (din) to [out=90,in=-90] (dout);
		\end{knot}
		\end{tikzpicture}
	}.
	$$
	To see the first equality, first observe that snapping (Lemma \ref{DCsnapping}) for the top right ring results in a free ring and the winding in the top right of the second diagram. The free ring evaluates to 1, leaving us with the winding. Now observe that we can use a similar trick for all the other rings, where we do have to move the middle ring down first.
	The second equality is unwinding the loop, using that overcrossing for the loop is the symmetry in $\cat{A}$, hence the same as an unresolved crossing (see Equation \eqref{DCtransparancy}). Reading the diagrams top to bottom yields a proof for the case of $\zeta$ and $v_1$.
\end{proof}

\subsubsection{Associator Coherence}
We now proceed to check coherence between the associators and the (op)lax structure.

\begin{lem}
	The morphisms $u_1$ and $u_2$ satisfy the coherence diagrams from Definition \ref{DClax2fold}\eqref{DCunitor2} for $\otimes_1=\otimes_c$ and $\otimes_2=\otimes_s$. Furthermore, the morphisms $v_1$ and $v_2$ satisfy the coherence diagrams from Definition \ref{DClax2fold}\eqref{DCunitor2} for $\otimes_1=\otimes_s$ and $\otimes_2=\otimes_c$.
\end{lem}
\begin{proof}
	For $u_1$ and $v_1$, there is nothing to prove. For $u_2$, we have to check that:
	\begin{center}
		\begin{tikzcd}
			(\mathbb{I}_1\otimes_2 \mathbb{I}_1)\otimes_2 \mathbb{I}_1 \arrow[rr,"\alpha_2"]& &\mathbb{I}_1\otimes_2(\mathbb{I}_1\otimes_2 \mathbb{I}_1)\\
			\mathbb{I}_1 \otimes_2 \mathbb{I}_1\arrow[u,"u_2\otimes_2\tn{id}"] \arrow[r, "u_1"] & \mathbb{I}_1& \mathbb{I}_1 \otimes_2 \mathbb{I}_1\arrow[l,"u_1"]\arrow[u,"\tn{id}\otimes_2 u_2"]
		\end{tikzcd}
	\end{center}
	commutes. In terms of string diagrams, this becomes:
	$$
	
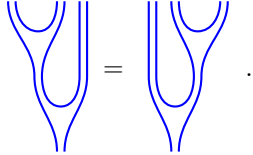
\hbox{
		\begin{tikzpicture}[baseline=(current  bounding  box.center)]

		\coordinate (inw) at (-0.05,0);
		\coordinate (ine) at (0.05,0);
		
		\coordinate (outlw) at (-0.35,1);
		\coordinate (outle) at (-0.25,1);
		
		\coordinate (outrw) at (0.25,1);
		\coordinate (outre) at (0.35,1);	
		
		\coordinate (ooww) at (-0.7,2);
		\coordinate (oowe) at (-0.6,2);
		
		\coordinate (oomw) at (-0.05,2);
		\coordinate (oome) at (0.05,2);
		
		\coordinate (ooew) at (0.25,2);
		\coordinate (ooee) at (0.35,2);
		
		\coordinate (mid) at (0,0.6);	
		\coordinate (mu) at (-0.3,1.6);
		
		\begin{knot}[clip width=4,clip radius =3pt]
		\strand [thick,blue] (inw) to [out=90,in=-90] (outlw) to [out=90,in=-90] (ooww);
		\strand [thick,blue] (ine) to [out=90,in=-90] (outre) to [out=90,in=-90] (ooee);
		\strand [thick,blue] (oome) to [out=-90,in=90] (outle) to [out=-90,in=180] (mid) to [out=0,in=-90] (outrw) to [out=90,in=-90] (ooew);
		\strand [thick,blue] (oowe) to [out=-90,in=180]  (mu) to [out=0,in=-90] (oomw);
		\end{knot}
		\end{tikzpicture}
	}
	=\reflectbox{
		\hbox{
			\begin{tikzpicture}[baseline=(current  bounding  box.center)]

			\coordinate (inw) at (-0.05,0);
			\coordinate (ine) at (0.05,0);
			
			\coordinate (outlw) at (-0.35,1);
			\coordinate (outle) at (-0.25,1);
			
			\coordinate (outrw) at (0.25,1);
			\coordinate (outre) at (0.35,1);	
			
			\coordinate (ooww) at (-0.7,2);
			\coordinate (oowe) at (-0.6,2);
			
			\coordinate (oomw) at (-0.05,2);
			\coordinate (oome) at (0.05,2);
			
			\coordinate (ooew) at (0.25,2);
			\coordinate (ooee) at (0.35,2);
			
			\coordinate (mid) at (0,0.6);	
			\coordinate (mu) at (-0.3,1.6);
			
			\begin{knot}[clip width=4,clip radius =3pt]
			\strand [thick,blue] (inw) to [out=90,in=-90] (outlw) to [out=90,in=-90] (ooww);
			\strand [thick,blue] (ine) to [out=90,in=-90] (outre) to [out=90,in=-90] (ooee);
			\strand [thick,blue] (oome) to [out=-90,in=90] (outle) to [out=-90,in=180] (mid) to [out=0,in=-90] (outrw) to [out=90,in=-90] (ooew);
			\strand [thick,blue] (oowe) to [out=-90,in=180]  (mu) to [out=0,in=-90] (oomw);
			\end{knot}
			\end{tikzpicture}
		}
	}.
	$$
	The proof for $v_2$ proceeds similarly, remembering that the associators for $\otimes_s$ are induced from the associators of $\cat{A}$.
\end{proof}

\begin{lem}
	The morphisms $\eta,u_1$ make the diagrams from Definitions \ref{DClax2fold}\eqref{DCunitor3} commute for $\otimes_1=\otimes_c$ and $\otimes_2=\otimes_s$. Analogously, the morphisms $\zeta$ and $v_2$ make the diagrams from Definitions \ref{DClax2fold}\eqref{DCunitor3} commute for $\otimes_1=\otimes_s$ and $\otimes_2=\otimes_c$.
\end{lem}

\begin{proof}
	As $u_1$ and $v_1$ are the isomorphisms between $\mathbb{I}_c$ and $\mathbb{I}_c\otimes_s\mathbb{I}_c$, there is nothing to prove.
\end{proof}

\begin{lem}\label{DCcohef}
	The natural transformation $\eta$ makes the diagrams from Definitions \ref{DClax2fold}\eqref{DCtensor1} commute for $\otimes_1=\otimes_c$ and $\otimes_2=\otimes_s$. Analogously, the morphisms $\zeta$ make the diagrams from Definitions \ref{DClax2fold}\eqref{DCtensor2} commute for $\otimes_1=\otimes_s$ and $\otimes_2=\otimes_c$.	
\end{lem}

\begin{proof}
	For the first case, we compute the anti-clockwise composite from the top-right corner to the bottom-right corner:
	$$
	
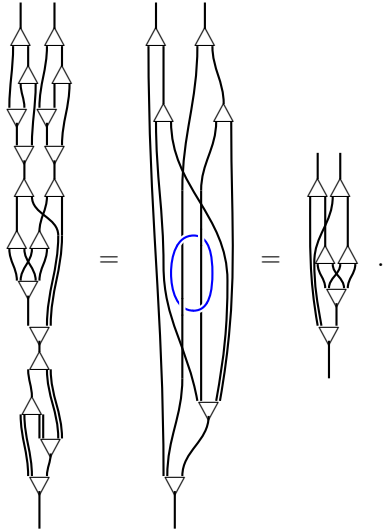
\hbox{
		\begin{tikzpicture}[baseline=(current  bounding  box.center)]
		
		\coordinate (in) at (0,0);
		
		\node (tr1) at (0,0.6) {\strut$\bigtriangledown$};
		\coordinate (tr1d) at (0,0.5);
		\coordinate (tr1uww) at (-0.15,0.7);
		\coordinate (tr1uw) at (-0.1,0.7);
		\coordinate (tr1ue) at (0.1,0.7);

		\node (tr2) at (0.15,1.1) {\strut$\bigtriangledown$};
		\coordinate (tr2d) at (0.15,1);
		\coordinate (tr2uww) at (0,1.2);
		\coordinate (tr2uw) at (0.05,1.2);
		\coordinate (tr2ue) at (0.25,1.2);
		\coordinate (tr2uee) at (0.3,1.2);

		\node (tr3) at (-0.1,1.6) {\strut$\bigtriangleup$};
		\coordinate (tr3dww) at (-0.25,1.52);
		\coordinate (tr3dw) at (-0.2,1.52);
		\coordinate (tr3de) at (-0.0,1.52);
		\coordinate (tr3dee) at (0.05,1.52);
		\coordinate (tr3u) at (-0.1,1.75);
		
		\node (tr4) at (0,2.2) {\strut$\bigtriangleup$};
		\coordinate (tr4dw) at (-0.15,2.12);
		\coordinate (tr4de) at (0.1,2.12);
		\coordinate (tr4dee) at (0.15,2.12);
		\coordinate (tr4u) at (0,2.35);
		
		\node (tr5) at (0,2.6) {\strut$\bigtriangledown$};
		\coordinate (tr5d) at (0,2.5);
		\coordinate (tr5uw) at (-0.15,2.7);
		\coordinate (tr5ue) at (0.1,2.7);
		\coordinate (tr5uee) at (0.15,2.7);
		
		\node (tr6) at (-0.15,3.2) {\strut$\bigtriangledown$};
		\coordinate (tr6d) at (-0.15,3.1);
		\coordinate (tr6uww) at (-0.3,3.3);
		\coordinate (tr6uw) at (-0.25,3.3);
		\coordinate (tr6ue) at (-0.05,3.3);
		\coordinate (tr6uee) at (0,3.3);
		
		\node (trl7) at (-0.3,3.8) {\strut$\bigtriangleup$};
		\coordinate (trl7dw) at (-0.4,3.72);
		\coordinate (trl7de) at (-0.2,3.72);
		\coordinate (trl7u) at (-0.3,3.95);
		
		\node (trr7) at (0,3.8) {\strut$\bigtriangleup$};
		\coordinate (trr7dw) at (-0.1,3.72);
		\coordinate (trr7de) at (0.1,3.72);
		\coordinate (trr7u) at (0,3.95);
		
		\coordinate (mw) at (0.25,3.9);
		\coordinate (me) at (0.3,3.9);
		
		\node (trl8) at (-0.2,4.5) {\strut$\bigtriangleup$};
		\coordinate (trl8dw) at (-0.3,4.42);
		\coordinate (trl8de) at (-0.1,4.42);
		\coordinate (trl8u) at (-0.2,4.65);
		
		\node (trr8) at (0.2,4.5) {\strut$\bigtriangleup$};
		\coordinate (trr8dw) at (0.1,4.42);
		\coordinate (trr8de) at (0.3,4.42);
		\coordinate (trr8u) at (0.2,4.65);
		
		\node (trl9) at (-0.2,5) {\strut$\bigtriangledown$};
		\coordinate (trl9d) at (-0.2,4.9);
		\coordinate (trl9uw) at (-0.3,5.1);
		\coordinate (trl9ue) at (-0.1,5.1);
		
		\node (trr9) at (0.2,5) {\strut$\bigtriangledown$};
		\coordinate (trr9d) at (0.2,4.9);
		\coordinate (trr9uw) at (0.1,5.1);
		\coordinate (trr9ue) at (0.3,5.1);
		
		\node (trl10) at (-0.3,5.5) {\strut$\bigtriangledown$};
		\coordinate (trl10d) at (-0.3,5.4);
		\coordinate (trl10uw) at (-0.4,5.6);
		\coordinate (trl10ue) at (-0.2,5.6);
		
		\node (trr10) at (0.1,5.5) {\strut$\bigtriangledown$};
		\coordinate (trr10d) at (0.1,5.4);
		\coordinate (trr10uw) at (0,5.6);
		\coordinate (trr10ue) at (0.2,5.6);
		
		\node (trl11) at (-0.15,6) {\strut$\bigtriangleup$};
		\coordinate (trl11dw) at (-0.25,5.92);
		\coordinate (trl11de) at (-0.05,5.92);
		\coordinate (trl11u) at (-0.15,6.15);
		
		\node (trr11) at (0.3,6) {\strut$\bigtriangleup$};
		\coordinate (trr11dw) at (0.2,5.92);
		\coordinate (trr11de) at (0.4,5.92);
		\coordinate (trr11u) at (0.3,6.15);
		
		\node (trl12) at (-0.25,6.5) {\strut$\bigtriangleup$};
		\coordinate (trl12dw) at (-0.35,6.42);
		\coordinate (trl12de) at (-0.15,6.42);
		\coordinate (trl12u) at (-0.25,6.65);
		
		\node (trr12) at (0.2,6.5) {\strut$\bigtriangleup$};
		\coordinate (trr12dw) at (0.1,6.42);
		\coordinate (trr12de) at (0.3,6.42);
		\coordinate (trr12u) at (0.2,6.65);
		
		\coordinate (outw) at (-0.25, 7) ;
		\coordinate (oute) at (0.2, 7) ;				
		
		\begin{knot}[clip width=4,clip radius =3pt]
		\strand [thick] (in) to (tr1d);
		\strand [thick] (tr1uww) to [out=90,in=-90] (tr3dww);
		\strand [thick] (tr1uw) to [out=90,in=-90] (tr3dw);
		\strand [thick] (tr1ue) to [out=90,in=-90] (tr2d);
		\strand [thick] (tr2uww) to [out=90,in=-90] (tr3de);
		\strand [thick] (tr2uw) to [out=90,in=-90] (tr3dee);
		\strand [thick] (tr2ue) to [out=90,in=-90] (tr4de);
		\strand [thick] (tr2uee) to [out=90,in=-90] (tr4dee);
		\strand [thick] (tr3u) to [out=90,in=-90] (tr4dw);
		\strand [thick] (tr4u) to (tr5d);
		\strand [thick] (tr5uw) to [out=90,in=-90] (tr6d);
		\strand [thick] (tr6uww) to [out=90,in=-90] (trl7dw);
		\strand [thick] (tr6uw) to [out=90,in=-90] (trr7dw);
		\strand [thick] (tr6ue) to [out=90,in=-90] (trl7de);
		\strand [thick] (tr6uee) to [out=90,in=-90] (trr7de);
		\strand [thick] (tr5ue) to [out=90,in=-90] (mw);
		\strand [thick] (tr5uee) to [out=90,in=-90] (me);
		\strand [thick] (trl7u) to [out=90,in=-90] (trl8dw);
		\strand [thick] (trr7u) to [out=90,in=-90] (trr8dw);
		\strand [thick] (mw) to [out=90,in=-90] (trl8de);
		\strand [thick] (me) to [out=90,in=-90] (trr8de);
		\strand [thick] (trl8u) to [out=90,in=-90] (trl9d);
		\strand [thick] (trr8u) to [out=90,in=-90] (trr9d);
		\strand [thick] (trl9uw) to [out=90,in=-90] (trl10d);
		\strand [thick] (trr9uw) to [out=90,in=-90] (trr10d);
		\strand [thick] (trl10ue) to [out=90,in=-90] (trl11dw);
		\strand [thick] (trl9ue) to [out=90,in=-90] (trl11de);
		\strand [thick] (trr10ue) to [out=90,in=-90] (trr11dw);
		\strand [thick] (trr9ue) to [out=90,in=-90] (trr11de);
		\strand [thick] (trl10uw) to [out=90,in=-90] (trl12dw);
		\strand [thick] (trr10uw) to [out=90,in=-90] (trr12dw);
		\strand [thick] (trl11u) to [out=90,in=-90] (trl12de);
		\strand [thick] (trr11u) to [out=90,in=-90] (trr12de);
		\strand [thick] (trl12u) to (outw);
		\strand [thick] (trr12u) to (oute);
		\end{knot}
		\end{tikzpicture}
	}
	=
	\hbox{
		\begin{tikzpicture}[baseline=(current  bounding  box.center)]
		
		\coordinate (in) at (0,0);
		
		\node (tr1) at (0,0.6) {\strut$\bigtriangledown$};
		\coordinate (tr1d) at (0,0.5);
		\coordinate (tr1uww) at (-0.15,0.7);
		\coordinate (tr1uw) at (-0.1,0.7);
		\coordinate (tr1ue) at (0.1,0.7);

		\node (tr2) at (0.45,1.6) {\strut$\bigtriangledown$};
		\coordinate (tr2d) at (0.45,1.5);
		\coordinate (tr2uww) at (0.3,1.7);
		\coordinate (tr2uw) at (0.35,1.7);
		\coordinate (tr2ue) at (0.55,1.7);
		\coordinate (tr2uee) at (0.6,1.7);
		
		\coordinate (low) at (0.1, 2);
		
		\coordinate (midw) at (-0.15,3.5);
		\coordinate (midee) at (0.7,3.3);
		
		\coordinate (upw) at (0.1,4.5);
		\coordinate (up) at (0.35,4.5);
		
		\node (trl11) at (-0.15,5.5) {\strut$\bigtriangleup$};
		\coordinate (trl11dw) at (-0.25,5.42);
		\coordinate (trl11de) at (-0.05,5.42);
		\coordinate (trl11u) at (-0.15,5.65);
		
		\node (trr11) at (0.65,5.5) {\strut$\bigtriangleup$};
		\coordinate (trr11dw) at (0.55,5.42);
		\coordinate (trr11de) at (0.75,5.42);
		\coordinate (trr11u) at (0.65,5.65);
		
		\node (trl12) at (-0.25,6.5) {\strut$\bigtriangleup$};
		\coordinate (trl12dw) at (-0.35,6.42);
		\coordinate (trl12de) at (-0.15,6.42);
		\coordinate (trl12u) at (-0.25,6.65);
		
		\node (trr12) at (0.4,6.5) {\strut$\bigtriangleup$};
		\coordinate (trr12dw) at (0.3,6.42);
		\coordinate (trr12de) at (0.5,6.42);
		\coordinate (trr12u) at (0.4,6.65);
		
		\coordinate (outw) at (-0.25, 7) ;
		\coordinate (oute) at (0.4, 7) ;		
		
		\coordinate (west1) at (-0.05,3.4);
		\coordinate (north1) at (0.25,3.9);
		\coordinate (east1) at (0.5,3.4);
		\coordinate (south1) at (0.25,2.9);

		\begin{knot}[clip width=3,clip radius =3pt]
		\strand [thick] (in) to (tr1d);
		\strand [thick] (tr1uww) to [out=90,in=-90] (trl12dw);
		\strand [thick] (tr1uw) to [out=90,in=-90] (low);
		\strand [thick] (low) to [out=90,in=-90] (upw);
		\strand [thick] (upw) to [out=90,in=-90](trr12dw);
		\strand [thick] (tr1ue) to [out=90,in=-90] (tr2d);
		\strand [thick] (tr2uww) to [out=100,in=-90] (midw) to [out=90,in=-90] (trl11dw);
		\strand [thick] (tr2uw) to [out=90,in=-90] (up);
		\strand [thick] (up) to [out=90,in=-90] (trr11dw);
		\strand [thick] (tr2ue) to [out=80,in=-90] (midee) to [out=90,in=-90] (trl11de);
		\strand [thick] (tr2uee) to [out=80,in=-90] (trr11de);
		\strand [thick] (trl11u) to [out=90,in=-90] (trl12de);
		\strand [thick] (trr11u) to [out=90,in=-90] (trr12de);
		\strand [thick] (trl12u) to (outw);
		\strand [thick] (trr12u) to (oute);
		\strand [blue, thick] (west1) to [out=90,in=-180] (north1) to [out=0,in=90] (east1) to [out=-90,in=0] (south1) to [out=-180,in=-90] (west1);
		\flipcrossings{1,4}
		\end{knot}
		\end{tikzpicture}
	}
	=
	\hbox{
		\begin{tikzpicture}[baseline=(current  bounding  box.center)]
		
		\coordinate (in) at (0,0);
		
		\node (tr1) at (0,0.6) {\strut$\bigtriangledown$};
		\coordinate (tr1d) at (0,0.5);
		\coordinate (tr1uww) at (-0.15,0.7);
		\coordinate (tr1uw) at (-0.1,0.7);
		\coordinate (tr1ue) at (0.1,0.7);

		\node (tr2) at (0.1,1.1) {\strut$\bigtriangledown$};
		\coordinate (tr2d) at (0.1,1);
		\coordinate (tr2uww) at (-0.05,1.2);
		\coordinate (tr2uw) at (0.0,1.2);
		\coordinate (tr2ue) at (0.2,1.2);
		\coordinate (tr2uee) at (0.25,1.2);
		
		\coordinate (lc) at (-0.195,1.55);
		
		\node (trl11) at (-0.05,1.6) {\strut$\bigtriangleup$};
		\coordinate (trl11dw) at (-0.15,1.52);
		\coordinate (trl11de) at (0.05,1.52);
		\coordinate (trl11u) at (-0.05,1.75);
		
		\node (trr11) at (0.25,1.6) {\strut$\bigtriangleup$};
		\coordinate (trr11dw) at (0.15,1.52);
		\coordinate (trr11de) at (0.35,1.52);
		\coordinate (trr11u) at (0.25,1.75);
		
		\node (trl12) at (-0.15,2.5) {\strut$\bigtriangleup$};
		\coordinate (trl12dw) at (-0.25,2.42);
		\coordinate (trl12de) at (-0.05,2.42);
		\coordinate (trl12u) at (-0.15,2.65);
		
		\node (trr12) at (0.15,2.5) {\strut$\bigtriangleup$};
		\coordinate (trr12dw) at (0.05,2.42);
		\coordinate (trr12de) at (0.25,2.42);
		\coordinate (trr12u) at (0.15,2.65);
		
		\coordinate (outw) at (-0.15, 3) ;
		\coordinate (oute) at (0.15, 3) ;

		\begin{knot}[clip width=0,clip radius =3pt]
		\strand [thick] (in) to (tr1d);
		\strand [thick] (tr1uww) to [out=100,in=-90] (trl12dw);
		\strand [thick] (tr1uw) to [out=100,in=-90] (lc) to [out=90,in=-90](trr12dw);
		\strand [thick] (tr1ue) to [out=90,in=-90] (tr2d);
		\strand [thick] (tr2uww) to [out=90,in=-90] (trl11dw);
		\strand [thick] (tr2uw) to [out=90,in=-90] (trr11dw);
		\strand [thick] (tr2ue) to [out=90,in=-90] (trl11de);
		\strand [thick] (tr2uee) to [out=90,in=-90] (trr11de);
		\strand [thick] (trl11u) to [out=90,in=-90] (trl12de);
		\strand [thick] (trr11u) to [out=90,in=-90] (trr12de);
		\strand [thick] (trl12u) to (outw);
		\strand [thick] (trr12u) to (oute);
		\end{knot}
		\end{tikzpicture}
	}.
	$$
	In the first step, we used the relation from Equation \eqref{DCinclprojprops} to replace projection-inclusion pairs by rings, and subsequently used slicing (Lemma \ref{DCslicing}) to bring these rings to a position where we could use:
	$$
	
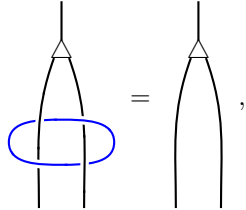
\hbox{
		\begin{tikzpicture}[baseline=(current  bounding  box.center)]

		\node (truw) at (-0.3,0){};
		\node (true) at (0.3,0){};
		
		\coordinate (lc) at (-0.3,1.8);
		\coordinate (rc) at (0.3,1.8);
		
		\node (tr1) at (0,2.2) {\strut$\bigtriangleup$};
		\coordinate (trd1w) at (-0.1,2.12);
		\coordinate (trd1e) at (0.1,2.12);
		\coordinate (tru1) at (0,2.35);
		\node (c1) at (0,3) {};
		
		\coordinate (west1) at (-0.7,1.0);
		\coordinate (north1) at (0.0,1.3);
		\coordinate (east1) at (0.7,1.0);
		\coordinate (south1) at (0.0,0.7);

		\begin{knot}[clip width=4]
		\strand [thick] (true) to [out=90, in=-70] (trd1e);
		\strand [thick] (tru1) to (c1);
		\strand [thick] (truw) to [out=90,in=-110] (trd1w);
		\strand [blue, thick] (west1) to [out=90,in=-180] (north1) to [out=0,in=90] (east1) to [out=-90,in=0] (south1) to [out=-180,in=-90] (west1);
		\flipcrossings{1,4}
		\end{knot}
		\end{tikzpicture}
	}
	=
	\hbox{
		\begin{tikzpicture}[baseline=(current  bounding  box.center)]

		\node (truw) at (-0.3,0){};
		\node (true) at (0.3,0){};
		
		\coordinate (lc) at (-0.3,1.8);
		\coordinate (rc) at (0.3,1.8);
		
		\node (tr1) at (0,2.2) {\strut$\bigtriangleup$};
		\coordinate (trd1w) at (-0.1,2.12);
		\coordinate (trd1e) at (0.1,2.12);
		\coordinate (tru1) at (0,2.35);
		\node (c1) at (0,3) {};

		\begin{knot}[clip width=4]
		\strand [thick] (true) to [out=90, in=-70] (trd1e);
		\strand [thick] (tru1) to (c1);
		\strand [thick] (truw) to [out=90,in=-110] (trd1w);
		\end{knot}
		\end{tikzpicture}
	},
	$$
	which is an easy consequence of Equation \ref{DCinclprojprops}. This left the ring in the middle of the second diagram. To rid ourselves of this, we used the relation between the braiding in $\dcentcat{A}$ and the symmetry in $\cat{A}$ from Equation \eqref{DCcrossinginteraction} and slicing to cancel the ring with the top right projection. The third diagram is just the composite on the right hand side of the coherence diagram \ref{DClax2fold}\eqref{DCtensor1}. 
	Reading the diagrams in this proof top to bottom yields a proof of the commutativity of the diagram from Definition \ref{DClax2fold}\eqref{DCtensor2}.
\end{proof}

\begin{lem}
	The natural transformation $\eta$ makes the diagrams from Definitions \ref{DClax2fold}\eqref{DCtensor2} commute for $\otimes_1=\otimes_c$ and $\otimes_2=\otimes_s$. Analogously, the morphisms $\zeta$ make the diagrams from Definitions \ref{DClax2fold}\eqref{DCtensor1} commute for $\otimes_1=\otimes_s$ and $\otimes_2=\otimes_c$.	
\end{lem}

\begin{proof}
	As we are suppressing the associators in the string diagrams, we see that we have, in terms of string diagrams for the left side composite in the diagram in Definition \ref{DClax2fold}\eqref{DCtensor2}:
	$$
	\hbox{
		\begin{tikzpicture}[baseline=(current  bounding  box.center)]
		
		\coordinate (inw) at (0,0);
		\coordinate (inm) at (0.5,0);
		\coordinate (ine) at (1,0);

		\node (trw1) at (0,0.6) {\strut$\bigtriangledown$};
		\coordinate (trw1d) at (0,0.5);
		\coordinate (trw1uw) at (-0.1,0.7);
		\coordinate (trw1ue) at (0.1,0.7);
		
		\node (trm1) at (0.5,0.6) {\strut$\bigtriangledown$};
		\coordinate (trm1d) at (0.5,0.5);
		\coordinate (trm1uw) at (0.4,0.7);
		\coordinate (trm1ue) at (0.6,0.7);
		
		\node (tre1) at (1,0.6) {\strut$\bigtriangledown$};
		\coordinate (tre1d) at (1,0.5);
		\coordinate (tre1uw) at (0.9,0.7);
		\coordinate (tre1ue) at (1.1,0.7);

		\node (tr2) at (0.25,1.6) {\strut$\bigtriangleup$};
		\coordinate (tr2dw) at (0.15,1.52);
		\coordinate (tr2de) at (0.35,1.52);
		\coordinate (tr2dww) at (0.1,1.52);
		\coordinate (tr2dee) at (0.4,1.52);
		\coordinate (tr2u) at (0.25,1.75);

		\node (tr3) at (0.25,2) {\strut$\bigtriangledown$};
		\coordinate (tr3d) at (0.25,1.9);
		\coordinate (tr3uw) at (0.15,2.1);
		\coordinate (tr3ue) at (0.35,2.1);
		\coordinate (tr3uww) at (0.1,2.1);
		\coordinate (tr3uee) at (0.4,2.1);

		\node (tr4) at (0.5,3) {\strut$\bigtriangleup$};
		\coordinate (tr4dw) at (0.45,2.92);
		\coordinate (tr4de) at (0.55,2.92);
		\coordinate (tr4dww) at (0.4,2.92);
		\coordinate (tr4dee) at (0.6,2.92);
		\coordinate (tr4dwww) at (0.35,2.92);
		\coordinate (tr4deee) at (0.65,2.92);
		\coordinate (tr4u) at (0.5,3.15);
		
		\coordinate (out) at (0.5, 3.5) ;

		\begin{knot}[clip width=0,clip radius =3pt]
		\strand [thick] (inw) to (trw1d);
		\strand [thick] (inm) to (trm1d);
		\strand [thick] (ine) to (tre1d);
		\strand [thick] (trw1uw) to [out=90,in=-90] (tr2dww);
		\strand [thick] (trm1uw) to [out=90,in=-90](tr2dw);
		\strand [thick] (trw1ue) to [out=90,in=-90] (tr2de);
		\strand [thick] (trm1ue) to [out=90,in=-90] (tr2dee);
		\strand [thick] (tr2u) to [out=90,in=-90] (tr3d);
		\strand [thick] (tre1ue) to [out=90,in=-90] (tr4deee);
		\strand [thick] (tre1uw) to [out=90,in=-90] (tr4dw);
		\strand [thick] (tr3uww) to [out=90,in=-90] (tr4dwww);
		\strand [thick] (tr3uw) to [out=90,in=-90] (tr4dww);
		\strand [thick] (tr3ue) to [out=90,in=-90] (tr4de);
		\strand [thick] (tr3uee) to [out=90,in=-90] (tr4dee);
		\strand [thick] (tr4u) to (out);
		\end{knot}
		\end{tikzpicture}
	}
	=
	\hbox{
		\begin{tikzpicture}[baseline=(current  bounding  box.center)]
		
		\coordinate (inw) at (0,0);
		\coordinate (inm) at (0.5,0);
		\coordinate (ine) at (1,0);

		\node (trw1) at (0,0.6) {\strut$\bigtriangledown$};
		\coordinate (trw1d) at (0,0.5);
		\coordinate (trw1uw) at (-0.1,0.7);
		\coordinate (trw1ue) at (0.1,0.7);
		
		\node (trm1) at (0.5,0.6) {\strut$\bigtriangledown$};
		\coordinate (trm1d) at (0.5,0.5);
		\coordinate (trm1uw) at (0.4,0.7);
		\coordinate (trm1ue) at (0.6,0.7);
		
		\node (tre1) at (1,0.6) {\strut$\bigtriangledown$};
		\coordinate (tre1d) at (1,0.5);
		\coordinate (tre1uw) at (0.9,0.7);
		\coordinate (tre1ue) at (1.1,0.7);

		\coordinate (tr2dw) at (0.15,1.22);
		\coordinate (tr2de) at (0.35,1.22);
		\coordinate (tr2dww) at (0.1,1.22);
		\coordinate (tr2dee) at (0.4,1.22);

		\coordinate (tr3uw) at (0.15,2.1);
		\coordinate (tr3ue) at (0.35,2.1);
		\coordinate (tr3uww) at (0.1,2.1);
		\coordinate (tr3uee) at (0.4,2.1);

		\node (tr4) at (0.5,3) {\strut$\bigtriangleup$};
		\coordinate (tr4dw) at (0.45,2.92);
		\coordinate (tr4de) at (0.55,2.92);
		\coordinate (tr4dww) at (0.4,2.92);
		\coordinate (tr4dee) at (0.6,2.92);
		\coordinate (tr4dwww) at (0.35,2.92);
		\coordinate (tr4deee) at (0.65,2.92);
		\coordinate (tr4u) at (0.5,3.15);
		
		\coordinate (out) at (0.5, 3.5) ;
		
		\coordinate (west1) at (-0.15,1.7);
		\coordinate (north1) at (0.3,2);
		\coordinate (east1) at (0.6,1.7);
		\coordinate (south1) at (0.3,1.4);

		\begin{knot}[clip width=3,clip radius =3pt]
		\strand [thick] (inw) to (trw1d);
		\strand [thick] (inm) to (trm1d);
		\strand [thick] (ine) to (tre1d);
		\strand [thick] (trw1uw) to [out=90,in=-90] (tr2dww) to (tr3uww)to [out=90,in=-90] (tr4dwww);
		\strand [thick] (trm1uw) to [out=90,in=-90](tr2dw) to (tr3uw) to [out=90,in=-90] (tr4dww);
		\strand [thick] (trw1ue) to [out=90,in=-90] (tr2de) to (tr3ue) to [out=90,in=-90] (tr4de);
		\strand [thick] (trm1ue) to [out=90,in=-90] (tr2dee) to (tr3uee) to [out=90,in=-90] (tr4dee);
		\strand [thick] (tre1ue) to [out=90,in=-90] (tr4deee);
		\strand [thick] (tre1uw) to [out=90,in=-90] (tr4dw);
		\strand [thick] (tr4u) to (out);
		\strand [blue, thick] (west1) to [out=90,in=-180] (north1) to [out=0,in=90] (east1) to [out=-90,in=0] (south1) to [out=-180,in=-90] (west1);
		\flipcrossings{1,3,6,8}
		\end{knot}
		\end{tikzpicture}
	}
	=
	\hbox{
		\begin{tikzpicture}[baseline=(current  bounding  box.center)]
		
		\coordinate (inw) at (0,0);
		\coordinate (inm) at (0.5,0);
		\coordinate (ine) at (1,0);

		\node (trw1) at (0,0.6) {\strut$\bigtriangledown$};
		\coordinate (trw1d) at (0,0.5);
		\coordinate (trw1uw) at (-0.1,0.7);
		\coordinate (trw1ue) at (0.1,0.7);
		
		\node (trm1) at (0.5,0.6) {\strut$\bigtriangledown$};
		\coordinate (trm1d) at (0.5,0.5);
		\coordinate (trm1uw) at (0.4,0.7);
		\coordinate (trm1ue) at (0.6,0.7);
		
		\node (tre1) at (1,0.6) {\strut$\bigtriangledown$};
		\coordinate (tre1d) at (1,0.5);
		\coordinate (tre1uw) at (0.9,0.7);
		\coordinate (tre1ue) at (1.1,0.7);

		\coordinate (tr2dw) at (0.15,2.1);
		\coordinate (tr2de) at (0.35,2.1);
		\coordinate (tr2dww) at (0.1,2.1);
		\coordinate (tr2dee) at (0.4,2.1);

		\coordinate (tr3uw) at (0.15,2.3);
		\coordinate (tr3ue) at (0.35,2.3);
		\coordinate (tr3uww) at (0.1,2.3);
		\coordinate (tr3uee) at (0.4,2.3);

		\node (tr4) at (0.5,3) {\strut$\bigtriangleup$};
		\coordinate (tr4dw) at (0.45,2.92);
		\coordinate (tr4de) at (0.55,2.92);
		\coordinate (tr4dww) at (0.4,2.92);
		\coordinate (tr4dee) at (0.6,2.92);
		\coordinate (tr4dwww) at (0.35,2.92);
		\coordinate (tr4deee) at (0.65,2.92);
		\coordinate (tr4u) at (0.5,3.15);
		
		\coordinate (out) at (0.5, 3.5) ;
		
		\coordinate (west1) at (-0.4,1.4);
		\coordinate (north1) at (0.2,1.6);
		\coordinate (east1) at (0.7,1.4);
		\coordinate (south1) at (0.2,1.2);

		\begin{knot}[clip width=3,clip radius =3pt]
		\strand [thick] (inw) to (trw1d);
		\strand [thick] (inm) to (trm1d);
		\strand [thick] (ine) to (tre1d);
		\strand [thick] (trw1uw) to [out=90,in=-90] (tr2dww) to (tr3uww)to [out=90,in=-90] (tr4dwww);
		\strand [thick] (trm1uw) to [out=90,in=-90](tr2de);
		\strand [thick] (tr2de) to (tr3uw) to [out=90,in=-90] (tr4dww);
		\strand [thick] (trw1ue) to [out=90,in=-90] (tr2dw) to (tr3ue) to [out=90,in=-90] (tr4de);
		\strand [thick] (trm1ue) to [out=90,in=-90] (tr2dee) to (tr3uee) to [out=90,in=-90] (tr4dee);
		\strand [thick] (tre1ue) to [out=90,in=-90] (tr4deee);
		\strand [thick] (tre1uw) to [out=90,in=-90] (tr4dw);
		\strand [thick] (tr4u) to (out);
		\strand [blue, thick] (west1) to [out=90,in=-180] (north1) to [out=0,in=90] (east1) to [out=-90,in=0] (south1) to [out=-180,in=-90] (west1);
		\flipcrossings{1,3,8,6}
		\end{knot}
		\end{tikzpicture}
	}
	=
	\hbox{
		\begin{tikzpicture}[baseline=(current  bounding  box.center)]
		
		\coordinate (inw) at (0,0);
		\coordinate (inm) at (0.5,0);
		\coordinate (ine) at (1,0);

		\node (trw1) at (0,0.6) {\strut$\bigtriangledown$};
		\coordinate (trw1d) at (0,0.5);
		\coordinate (trw1uw) at (-0.1,0.7);
		\coordinate (trw1ue) at (0.1,0.7);
		
		\node (trm1) at (0.5,0.6) {\strut$\bigtriangledown$};
		\coordinate (trm1d) at (0.5,0.5);
		\coordinate (trm1uw) at (0.4,0.7);
		\coordinate (trm1ue) at (0.6,0.7);
		
		\node (tre1) at (1,0.6) {\strut$\bigtriangledown$};
		\coordinate (tre1d) at (1,0.5);
		\coordinate (tre1uw) at (0.9,0.7);
		\coordinate (tre1ue) at (1.1,0.7);

		\coordinate (tr2dw) at (0.15,2.1);
		\coordinate (tr2de) at (0.35,2.1);
		\coordinate (tr2dww) at (0.1,2.1);
		\coordinate (tr2dee) at (0.4,2.1);

		\coordinate (tr3uw) at (0.15,2.3);
		\coordinate (tr3ue) at (0.35,2.3);
		\coordinate (tr3uww) at (0.1,2.3);
		\coordinate (tr3uee) at (0.4,2.3);

		\node (tr4) at (0.5,3) {\strut$\bigtriangleup$};
		\coordinate (tr4dw) at (0.45,2.92);
		\coordinate (tr4de) at (0.55,2.92);
		\coordinate (tr4dww) at (0.4,2.92);
		\coordinate (tr4dee) at (0.6,2.92);
		\coordinate (tr4dwww) at (0.35,2.92);
		\coordinate (tr4deee) at (0.65,2.92);
		\coordinate (tr4u) at (0.5,3.15);
		
		\coordinate (out) at (0.5, 3.5) ;
		
		\coordinate (west1) at (-0.4,1.4);
		\coordinate (north1) at (0.2,1.6);
		\coordinate (east1) at (0.7,1.4);
		\coordinate (south1) at (0.2,1.2);

		\begin{knot}[clip width=0,clip radius =3pt]
		\strand [thick] (inw) to (trw1d);
		\strand [thick] (inm) to (trm1d);
		\strand [thick] (ine) to (tre1d);
		\strand [thick] (trw1uw) to [out=90,in=-90] (tr4dwww);
		\strand [thick] (trm1uw) to [out=90,in=-90] (tr4dww);
		\strand [thick] (trw1ue) to [out=90,in=-90] (tr4de);
		\strand [thick] (trm1ue) to [out=90,in=-90] (tr4dee);
		\strand [thick] (tre1ue) to [out=90,in=-90] (tr4deee);
		\strand [thick] (tre1uw) to [out=90,in=-90] (tr4dw);
		\strand [thick] (tr4u) to (out);
		\end{knot}
		\end{tikzpicture}
	}.
	$$
	Here the first equality is replacing the inclusion-projection pair by the idempotent (Equation \eqref{DCinclprojprops}), the second uses the interaction between the braiding of $\dcentcat{A}$ and the symmetry of $\cat{A}$ (Equation \eqref{DCcrossinginteraction}). For the final step, we use slicing (Lemma \ref{DCslicing}) for the two leftmost inclusions on the bottom, this frees the ring so that it subsequently evaluates to 1.
	
	Similar arguments also reduce the right side of this coherence diagram to the rightmost string diagram.
	
	For the case involving $\zeta$, we read the diagrams top to bottom.
\end{proof}

This finishes proving that $\dcentcat{A}$ can be viewed as a bilax 2-fold monoidal category in as advertised in Theorem \ref{DCdc2foldmon}. 

\subsubsection{Braiding Coherence}
To prove Theorem \ref{DCdc2foldmon}, we still need to prove that the compatibility morphisms are compatible with the braiding.

\begin{lem}
	The morphism $v_2$ makes the diagram from Definition \ref{DCbraideddef}\eqref{DCu2braid} commute, where $\otimes_1=\otimes_s$ and $\otimes_2=\otimes_c$. Analogously, the morphism $u_1$ makes the corresponding diagram from Definition \ref{DCbraideddef}\eqref{DCu2braid} for $\beta_2$ commute, where $\otimes_1=\otimes_c$ and $\otimes_2=\otimes_s$.
\end{lem}

\begin{proof}
	As the maps involved are canonical isomorphisms coming from the unitors, the diagram \ref{DCbraideddef}\eqref{DCu2braid} is automatically commutative.
\end{proof}

\begin{lem}
	The morphism $u_2$ makes the diagram from Definition \ref{DCbraideddef}\eqref{DCu2braid} commute, where $\otimes_1=\otimes_c$ and $\otimes_2=\otimes_s$. Analogously, the morphism $v_1$ makes the corresponding diagram from Definition \ref{DCbraideddef}\eqref{DCu2braid} for $\beta_2$ commute, where $\otimes_1=\otimes_s$ and $\otimes_2=\otimes_c$.
\end{lem}

\begin{proof}
	We need to show that:
	$$
	
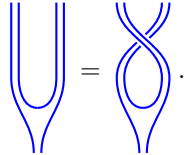
\hbox{
		\begin{tikzpicture}[baseline=(current  bounding  box.center)]

		\coordinate (inw) at (-0.05,0);
		\coordinate (ine) at (0.05,0);
		
		\coordinate (outlw) at (-0.35,1);
		\coordinate (outle) at (-0.25,1);
		
		\coordinate (outrw) at (0.25,1);
		\coordinate (outre) at (0.35,1);	
		
		\coordinate (mid) at (0,0.6);	
		
		\coordinate (ulw) at (-0.35,2);
		\coordinate (ule) at (-0.25,2);
		
		\coordinate (urw) at (0.25,2);
		\coordinate (ure) at (0.35,2);

		\begin{knot}[clip width=4,clip radius =3pt]
		\strand [thick,blue] (inw) to [out=90,in=-90] (outlw) to (ulw);
		\strand [thick,blue] (ine) to [out=90,in=-90] (outre) to (ure);
		\strand [thick,blue] (ule) to (outle) to [out=-90,in=180] (mid);
		\strand[thick,blue] (mid) to [out=0,in=-90] (outrw) to (urw);
		\end{knot}
		\end{tikzpicture}
	}
	=
	\hbox{
		\begin{tikzpicture}[baseline=(current  bounding  box.center)]
		
		\coordinate (inw) at (-0.05,0);
		\coordinate (ine) at (0.05,0);
		
		\coordinate (outlw) at (-0.35,1);
		\coordinate (outle) at (-0.25,1);
		
		\coordinate (outrw) at (0.25,1);
		\coordinate (outre) at (0.35,1);	
		
		\coordinate (mid) at (0,0.6);	
		
		\coordinate (ulw) at (-0.35,2);
		\coordinate (ule) at (-0.25,2);
		
		\coordinate (urw) at (0.25,2);
		\coordinate (ure) at (0.35,2);

		\begin{knot}[clip width=3,clip radius =3pt]
		\strand [thick,blue] (inw) to [out=90,in=-90] (outlw) to [out=90,in=-90] (urw);
		\strand [thick,blue] (ine) to [out=90,in=-90] (outre) to [out=90,in=-90] (ule);
		\strand [thick,blue] (ure) to [out=-90,in=90] (outle) to [out=-90,in=180] (mid);
		\strand[thick,blue] (mid) to [out=0,in=-90] (outrw) to [out=90,in=-90] (ulw);
		\flipcrossings{1,2,4}
		\end{knot}
		\end{tikzpicture}
	}.
	$$
	Using the definition of the half-braiding on $\mathbb{I}_s$ from Equation \eqref{DCunithalfbraid}, we see that the right hand side equals:
	\begin{equation}\label{DCsummandinproof}
	\sum\limits_{i\in\cat{O}(\cat{A})}\sum\limits_{\phi\in B(ii^*i,i)}
	\hbox{
		\begin{tikzpicture}[baseline=(current  bounding  box.center)]
		
		\node (inw) at (-0.1,-0.5){\strut$i$};
		\node (ine) at (0.1,-0.5){\strut$i^*$};
		
		\node (phi) at (-0.3,1.5)[draw]{\strut$\phi$};
		\node (phid) at (0.3,1.5)[draw]{\strut$\phi^*$};
		
		\coordinate (mid) at (0,0.6);	
		
		\node (ulw) at (-0.5,3){\strut$i$};
		\node (ule) at (-0.3,3){\strut$i^*$};
		
		\node (urw) at (0.3,3){\strut$i$};
		\node (ure) at (0.5,3){\strut$i^*$};	
		
		\node (jl) at (0.25,0.9){\strut$i$};
		
		\begin{knot}[clip width=3,clip radius =3pt]
		\strand [thick] (inw) to [out=90,in=-130] (phi);
		\strand [thick] (ine) to [out=90,in=-50] (phid);
		\strand [thick] (urw) to [out=-90,in=70] (phid); 
		\strand [thick] (phi) to [out=-110,in=180] (mid);
		\strand [thick] (mid) to [out=0,in=-80] (phi);
		\strand [thick] (phi) to [out=90,in=-90] (ulw);
		\strand [thick] (phid) to [out=90,in=-90] (ule);
		\strand [thick] (phid) to [out=60, in=-90] (ure);
		\end{knot}
		\end{tikzpicture}
	},
	\end{equation}
	where we have already taken into account restrictions on the possible combinations of labelling of the strands that can occur: the two leftmost strands coming into $\phi$ need to agree with the rightmost two coming out of $\phi^*$, and these in turn must be labelled by a pair of dual objects. Further, as the leftmost strand going from bottom to top is a morphism between simple objects, its incoming and outgoing labels must be the same. 
	
	Examining $\phi\in\Hom(ii^*i,i)\cong \Hom(ii^*,ii^*)$, we see we can write it as:
	$$
	\hbox{
		\begin{tikzpicture}[baseline=(current  bounding  box.center)]
		
		\node (inw) at (-0.1,-1){\strut$i$};
		\node (inm) at (0.1,-1){\strut$i^*$};
		\node (ine) at (0.3,-1){\strut$i$};		
		
		\node (phi) at (0.1,0.5)[draw]{\strut$\phi$};

		\node (u) at (0.1,2.5){\strut$i$};
		
		\begin{knot}[clip width=3,clip radius =3pt]
		\strand [thick] (inw) to [out=90,in=-120] (phi);
		\strand [thick] (inm) to [out=90,in=-90] (phi);
		\strand [thick] (ine) to [out=90,in=-60] (phi);
		\strand [thick] (phi) to [out=90, in=-90] (u);
		\end{knot}
		\end{tikzpicture}
	}
	=
	\hbox{
		\begin{tikzpicture}[baseline=(current  bounding  box.center)]
		
		\node (inw) at (-0.1,-1){\strut$i$};
		\node (inm) at (0.1,-1){\strut$i^*$};
		\node (ine) at (0.3,-1){\strut$i$};		
		
		\node (psi) at (0.1,0.5)[draw]{\strut$\psi$};
		\node (psit) at (0.1,1.5)[draw]{\strut$\psi^t$};
		
		\coordinate (c) at (0.6,1.7);
		
		\node (u) at (0.1,2.5){\strut$i$};
		
		\node (l) at (0.2, 1){\strut$l$};
		
		\begin{knot}[clip width=3,clip radius =3pt]
		\strand [thick] (inw) to [out=90,in=-100] (psi);
		\strand [thick] (inm) to [out=90,in=-80] (psi);
		\strand [thick] (ine) to [out=90,in=-90] (c) to [out=90,in=80] (psit);
		\strand [thick] (psi) to (psit);
		\strand [thick] (psit) to [out=100, in=-90] (u);
		\end{knot}
		\end{tikzpicture}
	},
	$$
	for some $l\in\SA$ and $\psi\in\Hom(ii^*,l)$. Therefore, picking a basis for $\Hom(ii^*,l)$ for each $l\in\SA$ gives a basis for $\Hom(ii^*,ii^*)\cong \Hom(ii^*i,i)$. Rescaling if necessary we can arrange
	$$
	\hbox{
		\begin{tikzpicture}[baseline=(current  bounding  box.center)]
		
		\node (inw) at (-0.1,-1){\strut$i$};
		
		\node (phi) at (0.1,0.5)[draw]{\strut$\phi$};
		
		\coordinate (lc) at (0.2,-0.5);
		
		\node (u) at (0.1,2.5){\strut$i$};
		
		\begin{knot}[clip width=3,clip radius =3pt]
		\strand [thick] (inw) to [out=90,in=-120] (phi);
		\strand [thick] (phi) to [out=-60,in=0] (lc) to [out=180,in=-90] (phi);
		\strand [thick] (phi) to [out=90, in=-90] (u);
		\end{knot}
		\end{tikzpicture}
	}
	=
	\hbox{
		\begin{tikzpicture}[baseline=(current  bounding  box.center)]
		
		\node (inw) at (-0.1,-1){\strut$i$};
		
		\node (psi) at (0.1,0.5)[draw]{\strut$\psi$};
		\node (psit) at (0.1,1.5)[draw]{\strut$\psi^t$};
		
		\coordinate (c) at (0.6,1.7);
		
		\coordinate (lc) at (0.6,0.2);
		
		\node (u) at (0.1,2.5){\strut$i$};
		
		\node (l) at (0.2, 1){\strut$l$};
		
		\begin{knot}[clip width=3,clip radius =3pt]
		\strand [thick] (inw) to [out=90,in=-100] (psi);
		\strand [thick] (psi) to [out=-60,in=-90] (lc) to [out=90,in=-90] (c) to [out=90,in=80] (psit);
		\strand [thick] (psi) to (psit);
		\strand [thick] (psit) to [out=100, in=-90] (u);
		\end{knot}
		\end{tikzpicture}
	}
	=
	\hbox{
		\begin{tikzpicture}[baseline=(current  bounding  box.center)]
		
		\node (inw) at (0.1,-1){\strut$i$};
		
		\node (u) at (0.1,2.5){\strut$i$};

		\begin{knot}[clip width=3,clip radius =3pt]
		\strand [thick] (inw) to [out=90,in=-90] (u);
		\end{knot}
		\end{tikzpicture}
	}.
	$$
	We now claim that the transposes for these $\phi$ are given by:
	$$
	\hbox{
		\begin{tikzpicture}[baseline=(current  bounding  box.center)]
		
		\node (in) at (0.1,-1){\strut$i$};
		
		\node (phi) at (0.1,0.5)[draw]{\strut$\phi^t$};

		\node (ow) at (-0.1,2){\strut$i$};
		\node (om) at (0.1,2){\strut$i^*$};
		\node (oe) at (0.3,2){\strut$i$};

		\begin{knot}[clip width=3,clip radius =3pt]
		\strand [thick] (ow) to [out=-90,in=120] (phi);
		\strand [thick] (om) to [out=-90,in=90] (phi);
		\strand [thick] (oe) to [out=-90,in=60] (phi);
		\strand [thick] (phi) to [out=-90, in=90] (in);
		\end{knot}
		\end{tikzpicture}
	}
	=
	\hbox{
		\begin{tikzpicture}[baseline=(current  bounding  box.center)]
		
		\node (in) at (-0.1,-1){\strut$i$};
		
		\node (psi) at (0.1,0)[draw]{\strut$\psi$};
		\node (psit) at (0.1,1)[draw]{\strut$\psi^t$};
		
		\coordinate (c) at (0.5,-0.2);
		\coordinate (lc) at (0.5,1.1);
		
		\node (ow) at (-0.1,2){\strut$i$};
		\node (om) at (0.1,2){\strut$i^*$};
		\node (oe) at (0.3,2){\strut$i$};

		\node (l) at (0.2, 0.5){\strut$l$};
		
		\begin{knot}[clip width=3,clip radius =3pt]
		\strand [thick] (ow) to [out=-90,in=100] (psit);
		\strand [thick] (om) to [out=-90,in=80] (psit);
		\strand [thick] (oe) to [out=-90,in=90] (lc) to [out=-90,in=90] (c) to [out=-90,in=-80] (psi);
		\strand [thick] (psi) to (psit);
		\strand [thick] (psi) to [out=-100, in=90] (in);
		\end{knot}
		\end{tikzpicture}
	}.
	$$
	
	To see this, we compute the composite:

	$$
	\hbox{
		\begin{tikzpicture}[baseline=(current  bounding  box.center)]
		
		\coordinate (in) at (0.1,-2);
		
		\node (phit) at (0.1,-1)[draw]{\strut$\phi^t$};
		\node (phi) at (0.1,2)[draw]{\strut$\phi'$};

		\coordinate (o) at (0.1,3);

		\begin{knot}[clip width=3,clip radius =3pt]
		\strand [thick] (o) to [out=-90,in=90] (phi);
		\strand [thick] (phit) to [out=90,in=-90] (phi);
		\strand [thick] (phit) to [out=100,in=-100] (phi);
		\strand [thick] (phit) to [out=80,in=-80] (phi);
		\strand [thick] (phit) to [out=-90, in=90] (in);
		\end{knot}
		\end{tikzpicture}
	}
	=
	\hbox{
		\begin{tikzpicture}[baseline=(current  bounding  box.center)]
		
		\coordinate (in) at (-0.1,-2);
		
		\node (psi) at (0.1,-1)[draw]{\strut$\psi$};
		\node (psit) at (0.1,0)[draw]{\strut$\psi^t$};
		
		\node (psip) at (0.1,1)[draw]{\strut$\psi'$};
		\node (psitp) at (0.1,2)[draw]{\strut$(\psi')^t$};

		\coordinate (c) at (0.7,-1.2);
		\coordinate (lc) at (0.7,2.2);
		
		\coordinate (o) at (-0.1,3);

		\coordinate (l) at (0.2, 0.5);
		
		\begin{knot}[clip width=3,clip radius =3pt]
		\strand [thick] (psip) to [out=-100,in=100] (psit);
		\strand [thick] (psip) to [out=-80,in=80] (psit);
		\strand [thick] (psip) to [out=90,in=-90] (psitp);
		\strand [thick] (psitp) to [out=70,in=90] (lc) to [out=-90,in=90] (c) to [out=-90,in=-70] (psi);
		\strand [thick] (psi) to (psit);
		\strand [thick] (psi) to [out=-100, in=90] (in);
		\strand [thick] (psitp) to [out=100,in=-90] (o);
		\end{knot}
		\end{tikzpicture}
	}
	=\delta_{\psi,\psi'}
	\hbox{
		\begin{tikzpicture}[baseline=(current  bounding  box.center)]
		
		\coordinate (in) at (-0.1,-2);
		
		\node (psi) at (0.1,-1)[draw]{\strut$\psi$};
		\node (psitp) at (0.1,2)[draw]{\strut$(\psi')^t$};

		\coordinate (c) at (0.7,-1.2);
		\coordinate (lc) at (0.7,2.2);
		
		\coordinate (o) at (-0.1,3);

		\coordinate (l) at (0.2, 0.5);
		
		\begin{knot}[clip width=3,clip radius =3pt]
		\strand [thick] (psi) to [out=90,in=-90] (psitp);
		\strand [thick] (psitp) to [out=70,in=90] (lc) to [out=-90,in=90] (c) to [out=-90,in=-70] (psi);
		\strand [thick] (psi) to [out=-100, in=90] (in);
		\strand [thick] (psitp) to [out=100,in=-90] (o);
		\end{knot}
		\end{tikzpicture}
	}
	=\delta_{\psi,\psi'}
	\hbox{
		\begin{tikzpicture}[baseline=(current  bounding  box.center)]
		
		\coordinate (in) at (-0.1,-2);

		\coordinate (o) at (-0.1,3);

		\draw [thick] (in) to [out=90,in=-90] (o);
		
		\end{tikzpicture}
	}.
	$$
	Putting this together, we see that the sum in Equation \eqref{DCsummandinproof} becomes:
	$$
	\sum\limits_{i\in\cat{O}(\cat{A})}\sum\limits_{l\in\cat{O}(\cat{A})}\sum\limits_{\psi\in B(i^*i,l)}
	\hbox{
		\begin{tikzpicture}[baseline=(current  bounding  box.center)]
		
		\node (inw) at (-0.1,-1.5){\strut$i$};
		\node (ine) at (0.1,-1.5){\strut$i^*$};

		\node (psi) at (0.5,1)[draw]{\strut$\psi$};
		\node (psit) at (0.5,2)[draw]{\strut$\psi^t$};
		
		\node (ulw) at (-0.5,3){\strut$i$};
		\node (ule) at (-0.3,3){\strut$i^*$};
		
		\node (urw) at (0.3,3){\strut$i$};
		\node (ure) at (0.5,3){\strut$i^*$};	
		
		\node (jl) at (0.6,1.5){\strut$l$};
		
		\begin{knot}[clip width=3,clip radius =3pt]
		\strand [thick] (inw) to [out=90,in=-90] (ulw);
		\strand [thick] (ine) to [out=90,in=-50] (psi);
		\strand [thick] (urw) to [out=-90,in=100] (psit); 
		\strand [thick] (psi) to [out=-100,in=-90] (ule);
		\strand [thick] (psit) to [out=80, in=-90] (ure);
		\strand [thick] (psi) to (psit);
		\end{knot}
		\end{tikzpicture}
	}
	=\sum\limits_{i\in\cat{O}(\cat{A})}
	\hbox{
		\begin{tikzpicture}[baseline=(current  bounding  box.center)]
		
		\node (inw) at (-0.1,-1.5){\strut$i$};
		\node (ine) at (0.1,-1.5){\strut$i^*$};

		\coordinate (psi) at (0,1.5);

		\node (ulw) at (-0.5,3){\strut$i$};
		\node (ule) at (-0.3,3){\strut$i^*$};
		
		\node (urw) at (0.3,3){\strut$i$};
		\node (ure) at (0.5,3){\strut$i^*$};

		\begin{knot}[clip width=3,clip radius =3pt]
		\strand [thick] (inw) to [out=90,in=-90] (ulw);
		\strand [thick] (ine) to [out=90,in=-90] (ure);
		\strand [thick] (urw) to [out=-90,in=0] (psi) to [out=-180,in=-90] (ule);
		\end{knot}
		\end{tikzpicture}
	},
	$$
	and this is what we wanted to show. For the proof of the other case, we read the diagrams top to bottom and see that two twists cancel.
\end{proof}

\begin{lem}
	The morphism $\eta$ makes the corresponding diagram from Definition \ref{DCbraideddef}\eqref{DCetabraid} for $\beta_2$ commute, where $\otimes_1=\otimes_c$ and $\otimes_2=\otimes_s$. Analogously, the morphism $\zeta$ makes the diagram from Definition \ref{DCbraideddef}\eqref{DCetabraid} commute, where $\otimes_1=\otimes_s$ and $\otimes_2=\otimes_c$.
\end{lem}

\begin{proof}
	For the first statement, the top route in the diagram computes as:
	$$
	
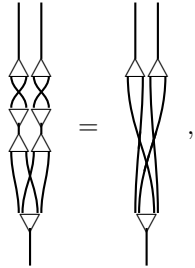
\hbox{
		\begin{tikzpicture}[baseline=(current  bounding  box.center)]
		
		\coordinate (in) at (0,0);		
		
		\node (tr1) at (0,0.6) {\strut$\bigtriangledown$};
		\coordinate (tr1d) at (0,0.5);
		\coordinate (tr1uw) at (-0.1,0.7);
		\coordinate (tr1ue) at (0.1,0.7);
		\coordinate (tr1uww) at (-0.15,0.7);
		\coordinate (tr1uee) at (0.15,0.7);

		\node (trw2) at (-0.15,1.6) {\strut$\bigtriangleup$};
		\coordinate (trw2dw) at (-0.25,1.52);
		\coordinate (trw2de) at (-0.05,1.52);
		\coordinate (trw2u) at (-0.15,1.75);
		
		\node (tre2) at (0.15,1.6) {\strut$\bigtriangleup$};
		\coordinate (tre2dw) at (0.05,1.52);
		\coordinate (tre2de) at (0.25,1.52);
		\coordinate (tre2u) at (0.15,1.75);

		\node (trw3) at (-0.15,2) {\strut$\bigtriangledown$};
		\coordinate (trw3d) at (-0.15,1.9);
		\coordinate (trw3uw) at (-0.25,2.1);
		\coordinate (trw3ue) at (-0.05,2.1);
		
		\node (tre3) at (0.15,2) {\strut$\bigtriangledown$};
		\coordinate (tre3d) at (0.15,1.9);
		\coordinate (tre3uw) at (0.05,2.1);
		\coordinate (tre3ue) at (0.25,2.1);

		\node (trw4) at (-0.15,2.6) {\strut$\bigtriangleup$};
		\coordinate (trw4dw) at (-0.25,2.52);
		\coordinate (trw4de) at (-0.05,2.52);
		\coordinate (trw4u) at (-0.15,2.75);
		
		\node (tre4) at (0.15,2.6) {\strut$\bigtriangleup$};
		\coordinate (tre4dw) at (0.05,2.52);
		\coordinate (tre4de) at (0.25,2.52);
		\coordinate (tre4u) at (0.15,2.75);

		\coordinate (outw) at (-0.15, 3.5) ;
		\coordinate (oute) at (0.15, 3.5) ;

		\begin{knot}[clip width=0,clip radius =3pt]
		\strand [thick] (in) to (tr1d);
		\strand [thick] (tr1uww) to [out=90,in=-90] (trw2dw);
		\strand [thick] (tr1uw) to [out=90,in=-90](tre2dw);
		\strand [thick] (tr1ue) to [out=90,in=-90] (trw2de);
		\strand [thick] (tr1uee) to [out=90,in=-90] (tre2de);
		\strand [thick] (trw2u) to [out=90,in=-90] (trw3d);
		\strand [thick] (tre2u) to [out=90,in=-90] (tre3d);
		\strand [thick] (trw3uw) to [out=90,in=-90] (trw4de);
		\strand [thick] (trw3ue) to [out=90,in=-90] (trw4dw);
		\strand [thick] (tre3uw) to [out=90,in=-90] (tre4de);
		\strand [thick] (tre3ue) to [out=90,in=-90] (tre4dw);
		\strand [thick] (trw4u) to (outw);
		\strand [thick] (tre4u) to (oute);
		\end{knot}
		\end{tikzpicture}
	}
	=
	\hbox{
		\begin{tikzpicture}[baseline=(current  bounding  box.center)]
		
		\coordinate (in) at (0,0);		
		
		\node (tr1) at (0,0.6) {\strut$\bigtriangledown$};
		\coordinate (tr1d) at (0,0.5);
		\coordinate (tr1uw) at (-0.1,0.7);
		\coordinate (tr1ue) at (0.1,0.7);
		\coordinate (tr1uww) at (-0.15,0.7);
		\coordinate (tr1uee) at (0.15,0.7);

		\node (trw4) at (-0.15,2.6) {\strut$\bigtriangleup$};
		\coordinate (trw4dw) at (-0.25,2.52);
		\coordinate (trw4de) at (-0.05,2.52);
		\coordinate (trw4u) at (-0.15,2.75);
		
		\node (tre4) at (0.15,2.6) {\strut$\bigtriangleup$};
		\coordinate (tre4dw) at (0.05,2.52);
		\coordinate (tre4de) at (0.25,2.52);
		\coordinate (tre4u) at (0.15,2.75);

		\coordinate (outw) at (-0.15, 3.5) ;
		\coordinate (oute) at (0.15, 3.5) ;

		\begin{knot}[clip width=0,clip radius =3pt]
		\strand [thick] (in) to (tr1d);
		\strand [thick] (tr1uww) to [out=90,in=-90] (trw4de);
		\strand [thick] (tr1uw) to [out=90,in=-90](tre4de);
		\strand [thick] (tr1ue) to [out=90,in=-90] (trw4dw);
		\strand [thick] (tr1uee) to [out=90,in=-90] (tre4dw);
		\strand [thick] (trw4u) to (outw);
		\strand [thick] (tre4u) to (oute);
		\end{knot}
		\end{tikzpicture}
	},
	$$
	where we immediately cancelled the rings coming from the projection after inclusion (like in the proof of Lemma \ref{DCcohef}). Similarly, the bottom route computes as:
	$$
	
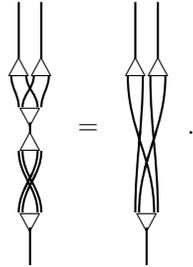
\hbox{
		\begin{tikzpicture}[baseline=(current  bounding  box.center)]
		
		\coordinate (in) at (0,0);		
		
		\node (tr1) at (0,0.6) {\strut$\bigtriangledown$};
		\coordinate (tr1d) at (0,0.5);
		\coordinate (tr1uw) at (-0.1,0.7);
		\coordinate (tr1ue) at (0.1,0.7);
		\coordinate (tr1uww) at (-0.15,0.7);
		\coordinate (tr1uee) at (0.15,0.7);

		\node (tr2) at (-0,1.6) {\strut$\bigtriangleup$};
		\coordinate (tr2dw) at (-0.1,1.52);
		\coordinate (tr2de) at (0.1,1.52);
		\coordinate (tr2u) at (-0.0,1.75);
		\coordinate (tr2dww) at (-0.15,1.52);
		\coordinate (tr2dee) at (0.15,1.52);

		\node (tr3) at (0,2) {\strut$\bigtriangledown$};
		\coordinate (tr3d) at (0,1.9);
		\coordinate (tr3uw) at (-0.1,2.1);
		\coordinate (tr3ue) at (0.1,2.1);
		\coordinate (tr3uww) at (-0.15,2.1);
		\coordinate (tr3uee) at (0.15,2.1);

		\node (trw4) at (-0.15,2.6) {\strut$\bigtriangleup$};
		\coordinate (trw4dw) at (-0.25,2.52);
		\coordinate (trw4de) at (-0.05,2.52);
		\coordinate (trw4u) at (-0.15,2.75);
		
		\node (tre4) at (0.15,2.6) {\strut$\bigtriangleup$};
		\coordinate (tre4dw) at (0.05,2.52);
		\coordinate (tre4de) at (0.25,2.52);
		\coordinate (tre4u) at (0.15,2.75);

		\coordinate (outw) at (-0.15, 3.5) ;
		\coordinate (oute) at (0.15, 3.5) ;

		\begin{knot}[clip width=0,clip radius =3pt]
		\strand [thick] (in) to (tr1d);
		\strand [thick] (tr1uww) to [out=90,in=-90] (tr2de);
		\strand [thick] (tr1uw) to [out=90,in=-90](tr2dee);
		\strand [thick] (tr1ue) to [out=90,in=-90] (tr2dww);
		\strand [thick] (tr1uee) to [out=90,in=-90] (tr2dw);
		\strand [thick] (tr2u) to [out=90,in=-90] (tr3d);
		\strand [thick] (tr3uww) to [out=90,in=-90] (trw4dw);
		\strand [thick] (tr3ue) to [out=90,in=-90] (trw4de);
		\strand [thick] (tr3uw) to [out=90,in=-90] (tre4dw);
		\strand [thick] (tr3uee) to [out=90,in=-90] (tre4de);
		\strand [thick] (trw4u) to (outw);
		\strand [thick] (tre4u) to (oute);
		\end{knot}
		\end{tikzpicture}
	}
	=
	\hbox{
		\begin{tikzpicture}[baseline=(current  bounding  box.center)]
		
		\coordinate (in) at (0,0);		
		
		\node (tr1) at (0,0.6) {\strut$\bigtriangledown$};
		\coordinate (tr1d) at (0,0.5);
		\coordinate (tr1uw) at (-0.1,0.7);
		\coordinate (tr1ue) at (0.1,0.7);
		\coordinate (tr1uww) at (-0.15,0.7);
		\coordinate (tr1uee) at (0.15,0.7);

		\node (trw4) at (-0.15,2.6) {\strut$\bigtriangleup$};
		\coordinate (trw4dw) at (-0.25,2.52);
		\coordinate (trw4de) at (-0.05,2.52);
		\coordinate (trw4u) at (-0.15,2.75);
		
		\node (tre4) at (0.15,2.6) {\strut$\bigtriangleup$};
		\coordinate (tre4dw) at (0.05,2.52);
		\coordinate (tre4de) at (0.25,2.52);
		\coordinate (tre4u) at (0.15,2.75);

		\coordinate (outw) at (-0.15, 3.5) ;
		\coordinate (oute) at (0.15, 3.5) ;

		\begin{knot}[clip width=0,clip radius =3pt]
		\strand [thick] (in) to (tr1d);
		\strand [thick] (tr1uww) to [out=90,in=-90] (trw4de);
		\strand [thick] (tr1uw) to [out=90,in=-90](tre4de);
		\strand [thick] (tr1ue) to [out=90,in=-90] (trw4dw);
		\strand [thick] (tr1uee) to [out=90,in=-90] (tre4dw);
		\strand [thick] (trw4u) to (outw);
		\strand [thick] (tre4u) to (oute);
		\end{knot}
		\end{tikzpicture}
	}.
	$$
	Reading the diagrams top to bottom yields a proof for the other assertion in the lemma.
\end{proof}

\begin{lem}
	The morphism $\eta$ makes the diagram from Definition \ref{DCbraideddef}\eqref{DCetabraid} commute, where $\otimes_1=\otimes_c$ and $\otimes_2=\otimes_s$. Analogously, the morphism $\zeta$ makes the corresponding diagram from Definition \ref{DCbraideddef}\eqref{DCetabraid} for $\beta_2$ commute, where $\otimes_1=\otimes_s$ and $\otimes_2=\otimes_c$.
\end{lem}

\begin{proof}
	The top route computes as:
	$$
	
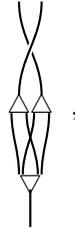
\hbox{
		\begin{tikzpicture}[baseline=(current  bounding  box.center)]
		
		\coordinate (in) at (0,0);		
		
		\node (tr1) at (0,0.6) {\strut$\bigtriangledown$};
		\coordinate (tr1d) at (0,0.5);
		\coordinate (tr1uw) at (-0.1,0.7);
		\coordinate (tr1ue) at (0.1,0.7);
		\coordinate (tr1uww) at (-0.15,0.7);
		\coordinate (tr1uee) at (0.15,0.7);

		\node (trw4) at (-0.15,1.6) {\strut$\bigtriangleup$};
		\coordinate (trw4dw) at (-0.25,1.52);
		\coordinate (trw4de) at (-0.05,1.52);
		\coordinate (trw4u) at (-0.15,1.75);
		
		\node (tre4) at (0.15,1.6) {\strut$\bigtriangleup$};
		\coordinate (tre4dw) at (0.05,1.52);
		\coordinate (tre4de) at (0.25,1.52);
		\coordinate (tre4u) at (0.15,1.75);

		\coordinate (outw) at (-0.15, 3) ;
		\coordinate (oute) at (0.15, 3) ;

		\begin{knot}[clip width=3,clip radius =5pt]
		\strand [thick] (in) to (tr1d);
		\strand [thick] (tr1uww) to [out=90,in=-90] (trw4dw);
		\strand [thick] (tr1uw) to [out=90,in=-90](tre4dw);
		\strand [thick] (tr1ue) to [out=90,in=-90] (trw4de);
		\strand [thick] (tr1uee) to [out=90,in=-90] (tre4de);
		\strand [thick] (trw4u) to [out=90,in=-90] (oute);
		\strand [thick] (tre4u) to [out=90,in=-90] (outw);
		\end{knot}
		\end{tikzpicture}
	},
	$$
	our goal is to show that the bottom route in the diagram is the same. For this composite we have that:
	$$
	
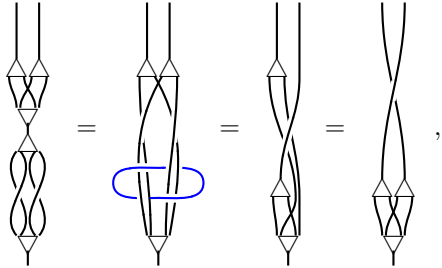
\hbox{
		\begin{tikzpicture}[baseline=(current  bounding  box.center)]
		
		\coordinate (in) at (0,0);		
		
		\node (tr1) at (0,0.3) {\strut$\bigtriangledown$};
		\coordinate (tr1d) at (0,0.2);
		\coordinate (tr1uw) at (-0.1,0.4);
		\coordinate (tr1ue) at (0.1,0.4);
		\coordinate (tr1uww) at (-0.15,0.4);
		\coordinate (tr1uee) at (0.15,0.4);

		\node (tr2) at (0,1.6) {\strut$\bigtriangleup$};
		\coordinate (tr2dw) at (-0.1,1.52);
		\coordinate (tr2de) at (0.1,1.52);
		\coordinate (tr2u) at (-0.0,1.75);
		\coordinate (tr2dww) at (-0.15,1.52);
		\coordinate (tr2dee) at (0.15,1.52);

		\node (tr3) at (0,2) {\strut$\bigtriangledown$};
		\coordinate (tr3d) at (0,1.9);
		\coordinate (tr3uw) at (-0.1,2.1);
		\coordinate (tr3ue) at (0.1,2.1);
		\coordinate (tr3uww) at (-0.15,2.1);
		\coordinate (tr3uee) at (0.15,2.1);

		\node (trw4) at (-0.15,2.6) {\strut$\bigtriangleup$};
		\coordinate (trw4dw) at (-0.25,2.52);
		\coordinate (trw4de) at (-0.05,2.52);
		\coordinate (trw4u) at (-0.15,2.75);
		
		\node (tre4) at (0.15,2.6) {\strut$\bigtriangleup$};
		\coordinate (tre4dw) at (0.05,2.52);
		\coordinate (tre4de) at (0.25,2.52);
		\coordinate (tre4u) at (0.15,2.75);

		\coordinate (outw) at (-0.15, 3.5) ;
		\coordinate (oute) at (0.15, 3.5) ;

		\begin{knot}[clip width=4,clip radius =3pt]
		\strand [thick] (in) to (tr1d);
		\strand [thick] (tr1uww) to [out=130,in=-50] (tr2dw);
		\strand [thick] (tr1uw) to [out=50,in=-130](tr2dww);
		\strand [thick] (tr1ue) to [out=130,in=-50] (tr2dee);
		\strand [thick] (tr1uee) to [out=50,in=-130] (tr2de);
		\strand [thick] (tr2u) to [out=90,in=-90] (tr3d);
		\strand [thick] (tr3uww) to [out=90,in=-90] (trw4dw);
		\strand [thick] (tr3ue) to [out=90,in=-90] (trw4de);
		\strand [thick] (tr3uw) to [out=90,in=-90] (tre4dw);
		\strand [thick] (tr3uee) to [out=90,in=-90] (tre4de);
		\strand [thick] (trw4u) to (outw);
		\strand [thick] (tre4u) to (oute);
		\end{knot}
		\end{tikzpicture}
	}
	=
	\hbox{
		\begin{tikzpicture}[baseline=(current  bounding  box.center)]
		
		\coordinate (in) at (0,0);		
		
		\node (tr1) at (0,0.3) {\strut$\bigtriangledown$};
		\coordinate (tr1d) at (0,0.2);
		\coordinate (tr1uw) at (-0.1,0.4);
		\coordinate (tr1ue) at (0.1,0.4);
		\coordinate (tr1uww) at (-0.15,0.4);
		\coordinate (tr1uee) at (0.15,0.4);

		\node (trw4) at (-0.15,2.6) {\strut$\bigtriangleup$};
		\coordinate (trw4dw) at (-0.25,2.52);
		\coordinate (trw4de) at (-0.05,2.52);
		\coordinate (trw4u) at (-0.15,2.75);
		
		\node (tre4) at (0.15,2.6) {\strut$\bigtriangleup$};
		\coordinate (tre4dw) at (0.05,2.52);
		\coordinate (tre4de) at (0.25,2.52);
		\coordinate (tre4u) at (0.15,2.75);

		\coordinate (outw) at (-0.15, 3.5) ;
		\coordinate (oute) at (0.15, 3.5) ;
		
		\coordinate (west1) at (-0.6,1.1);
		\coordinate (north1) at (0.0,1.3);
		\coordinate (east1) at (0.6,1.1);
		\coordinate (south1) at (0.0,0.9);

		\begin{knot}[clip width=4,clip radius =6pt]
		\strand [thick] (in) to (tr1d);
		\strand [thick] (tr1uww) to [out=100,in=-120] (tre4dw);
		\strand [thick] (tr1uw) to [out=90,in=-90](trw4dw);
		\strand [thick] (tr1ue) to [out=90,in=-90] (tre4de);
		\strand [thick] (tr1uee) to [out=80,in=-60] (trw4de);
		\strand [thick] (trw4u) to (outw);
		\strand [thick] (tre4u) to (oute);
		\strand [blue, thick] (west1) to [out=90,in=-180] (north1) to [out=0,in=90] (east1) to [out=-90,in=0] (south1) to [out=-180,in=-90] (west1);
		\flipcrossings{2,4,8,10}
		\end{knot}
		\end{tikzpicture}
	}
	=
	\hbox{
		\begin{tikzpicture}[baseline=(current  bounding  box.center)]
		
		\coordinate (in) at (0,0);		
		
		\node (tr1) at (0,0.3) {\strut$\bigtriangledown$};
		\coordinate (tr1d) at (0,0.2);
		\coordinate (tr1uw) at (-0.1,0.4);
		\coordinate (tr1ue) at (0.1,0.4);
		\coordinate (tr1uww) at (-0.15,0.4);
		\coordinate (tr1uee) at (0.15,0.4);

		\coordinate (cw) at (0.1,0.9);
		\coordinate (ce) at (0.15,0.9);
		
		\coordinate (cw2) at (0.1,0.9);
		\coordinate (ce2) at (0.15,0.9);

		\node (trw4) at (-0.15,2.6) {\strut$\bigtriangleup$};
		\coordinate (trw4dw) at (-0.25,2.52);
		\coordinate (trw4de) at (-0.05,2.52);
		\coordinate (trw4u) at (-0.15,2.75);
		
		\node (tre4) at (-0.1,1) {\strut$\bigtriangleup$};
		\coordinate (tre4dw) at (-0.2,0.92);
		\coordinate (tre4de) at (0.0,0.92);
		\coordinate (tre4u) at (-0.1,1.15);
		
		\coordinate (rc) at (0.15,2.5);
		
		\coordinate (outw) at (-0.15, 3.5);
		\coordinate (oute) at (0.15, 3.5);

		\begin{knot}[clip width=4,clip radius =4pt]
		\strand [thick] (in) to (tr1d);
		\strand [thick] (tr1uww) to [out=90,in=-90] (tre4dw);
		\strand [thick] (tr1uw) to [out=80,in=-90] (cw) to [out=90,in=-90] (cw2) to [out=90,in=-90] (trw4dw);
		\strand [thick] (tr1ue) to [out=90,in=-90] (tre4de);
		\strand [thick] (tr1uee) to [out=90,in=-90] (ce) to [out=90,in=-90] (ce2) to [out=90,in=-90] (trw4de);
		\strand [thick] (trw4u) to [out=90,in=-90] (outw);
		\strand [thick] (tre4u) to [out=90,in=-90] (rc) to [out=90,in=-90] (oute);
		\flipcrossings{1,2}
		\end{knot}
		\end{tikzpicture}
	}
	=
	\hbox{
		\begin{tikzpicture}[baseline=(current  bounding  box.center)]
		
		\coordinate (in) at (0,0);		
		
		\node (tr1) at (0,0.3) {\strut$\bigtriangledown$};
		\coordinate (tr1d) at (0,0.2);
		\coordinate (tr1uw) at (-0.1,0.4);
		\coordinate (tr1ue) at (0.1,0.4);
		\coordinate (tr1uww) at (-0.15,0.4);
		\coordinate (tr1uee) at (0.15,0.4);

		\node (trw4) at (0.15,1) {\strut$\bigtriangleup$};
		\coordinate (trw4dw) at (0.05,0.92);
		\coordinate (trw4de) at (0.25,0.92);
		\coordinate (trw4u) at (0.15,1.15);
		
		\node (tre4) at (-0.15,1) {\strut$\bigtriangleup$};
		\coordinate (tre4dw) at (-0.25,0.92);
		\coordinate (tre4de) at (-0.05,0.92);
		\coordinate (tre4u) at (-0.15,1.15);

		\coordinate (outw) at (-0.15, 3.5);
		\coordinate (oute) at (0.15, 3.5);

		\begin{knot}[clip width=4,clip radius =4pt]
		\strand [thick] (in) to (tr1d);
		\strand [thick] (tr1uww) to [out=90,in=-90] (tre4dw);
		\strand [thick] (tr1uw) to [out=90,in=-90] (trw4dw);
		\strand [thick] (tr1ue) to [out=90,in=-90] (tre4de);
		\strand [thick] (tr1uee) to [out=90,in=-90] (trw4de);
		\strand [thick] (trw4u) to [out=90,in=-90] (outw);
		\strand [thick] (tre4u) to [out=90,in=-90] (oute);
		\flipcrossings{1}
		\end{knot}
		\end{tikzpicture}
	},
	$$
	where in the first equality we slid the ring resulting from the projection-inclusion pair down and the second equality uses slicing (Lemma \ref{DCslicing}) to bring the ring out. The rest of the computation uses standard string diagram identities to bring the diagram into the desired form. For the analogous statement for $\zeta$, we read the diagrams top to bottom.
\end{proof}

This completes the proof of the main Theorem \ref{DCdc2foldmon}.

\end{document}